\documentclass[11pt,handout]{article}%
\usepackage{geometry}                
\usepackage{tikz} 
\usetikzlibrary{positioning}
\usepackage{graphicx}
\usepackage{amssymb}
\usepackage{epstopdf}
\usepackage{subfigure}  
\usepackage[sans]{dsfont}
\usepackage[english]{babel}
\usepackage{latexsym}
\usepackage{mathrsfs}
\usepackage{graphicx}
\usepackage{color}
\usepackage{float}
\usepackage{mathtools}
\usepackage{amsmath}
\usepackage{amsfonts}
\numberwithin{equation}{section}
\usepackage{enumerate}
\usepackage{amsthm}
\usepackage[title]{appendix}

\usepackage[bookmarks=true,colorlinks=true,linkcolor={blue},urlcolor={blue}, citecolor={blue},pdfstartview={XYZ null null 1.22}]{hyperref}%

\def\be#1\ee{\begin{equation}#1\end{equation}}

\newtheorem{proposition}{Proposition}

\theoremstyle{definition}
\newtheorem{alg}{Algorithm}[section]
\newtheorem{remark}{Remark}

\newcommand{\xx}{{\bf x}}

\newcommand{\zz}{{\bf z}}
\DeclareMathOperator{\sign}{sign}
\newcommand{\vv}{{\bf v}}
\newcommand{\BB}{{\bf B}}

\newcommand{\EE}{{\bf E}}
\newcommand{\UU}{{\bf U}}

\def\be{\begin{equation}}
	\def\ee{\end{equation}}
\def\bea{\begin{eqnarray}}
	\def\eea{\end{eqnarray}}

\title{Robust feedback control of collisional plasma dynamics in presence of uncertainties}

\author{G. Albi,\thanks{Department of Computer Science,
		University of Verona, Strada le Grazie 15, 37134 Verona, ITALY (giacomo.albi@univr.it).}\and
	G. Dimarco,\thanks{Department of Mathematics and Computer Science \& Center for Modeling, Computing and Statistics (CMCS), University of Ferrara, via Machiavelli 30, 44121 Ferrara, ITALY (giacomo.dimarco@unife.it).} \and
	F. Ferrarese, \thanks{Department of Mathematics and Computer Science \& Center for Modeling, Computing and Statistics (CMCS), University of Ferrara, via Machiavelli 30, 44121 Ferrara, ITALY (federica.ferrarese@unife.it).} \and
	L. Pareschi \thanks{
		Maxwell Institute for Mathematical Sciences and Department of Mathematics,
		School of Mathematical and Computer Sciences (MACS),
		Heriot-Watt University, Edinburgh, UK (L.Pareschi@hw.ac.uk), Department of Mathematics and Computer Science \& Center for Modeling, Computing and Statistics (CMCS), University of Ferrara, via Machiavelli 30, 44121 Ferrara, ITALY (lorenzo.pareschi@unife.it).
	}
}
\date{}
\begin{document}
		\maketitle
	\begin{abstract}
			Magnetic fusion aims to confine high-temperature plasma within a device, enabling the fusion of deuterium and tritium nuclei to release energy. Due to the very large temperatures involved, it is essential to isolate the plasma from the device walls to prevent structural damage and the external magnetic fields play a fundamental role in achieving this confinement. In realistic settings, the physical mechanisms governing plasma behavior are highly complex, involving numerous uncertain parameters and intricate particle interactions, such as collisions, that significantly affect both confinement efficiency and overall stability. In this work, we address particularly these challenges by proposing a robust feedback control strategy designed to steer the plasma towards a desired spatial region, despite the presence of uncertainties. From a modeling perspective, we consider a collisional plasma described by a Vlasov–Poisson–BGK system, which accounts for a self-consistent electric field and a strong external magnetic field, while incorporating uncertainty in the model. A key feature of the proposed control strategy is its independence from the random parameter, making it particularly suitable for practical applications. A series of numerical simulations confirms the effectiveness of our approach and demonstrates the ability of external magnetic fields to successfully confine plasma away from the device boundaries, even in the presence of uncertain conditions.
	\end{abstract}	
{\bf Keywords:} Vlasov-Poisson system, collisional plasma, instantaneous control, magnetic confinement, uncertainty quantification, particle methods.\\
\textbf{Mathematics Subject Classification}: 35Q83, 82D10, 65C30, 35Q93, 65M75

\tableofcontents

\section{Introduction}
In recent years, considerable effort has been devoted to the development of advanced numerical methods aimed at addressing the complex challenges arising in plasma physics simulations \cite{cheng1976integration,crouseilles2010conservative,degond2017asymptotic, filbet2003numerical,Dimarco2015}. A particular focus has been placed on the study of magnetized plasmas, due to their relevance in fusion energy applications, notably in confinement devices such as Tokamaks and Stellarators \cite{sonnendrucker1999semi,degond2017asymptotic,Stellarators}.
These devices rely on complex magnetic field configurations to confine and stabilize plasmas at very high temperatures. Reaching and maintaining optimal plasma conditions such as temperature, density, and confinement time is essential for the success of fusion experiments. However, the intrinsic complexity of plasma dynamics, characterized by turbulence, nonlinearity, and rapid transitions, combined with the influence of magnetic fields, poses important challenges both from the technical as well as from the simulation point of view \cite{crouseilles2004numerical,Degond2013}. Addressing these issues demands the design and the use of new numerical schemes and multi-scale modeling strategies capable of accurately capturing the evolution of magnetized plasmas \cite{belaouar2009asymptotically,blum1989numerical,cheng2014discontinuous,crouseilles2016multiscale, Chac2023}.
Additionally, the presence of uncertainty, ranging from limited knowledge of parameter settings, measurement errors in the magnetic field, the inability to define the exact initial configuration of the plasma, or the lack of accurate models to describe interactions with boundaries, to name a few, further complicates the analysis of the system. Nevertheless, such sources of uncertainty must be accounted for in the simulations \cite{medaglia2023stochastic, medaglia2024particle, dimarco2019multi, dimarco2020multiscale, Jin2017, yudin2023uq}.

In this work, to describe the time evolution of the plasma flow under uncertain conditions, we consider a Vlasov–Poisson--BGK system, which models the dynamics of charged particles subject to a self-consistent electric field, an externally applied magnetic field, and collisional effects described through a relaxation toward thermodynamic equilibrium \cite{saint2000gyrokinetic,crouseilles2016multiscale,crestetto2012kinetic, Hu2022,Chac2015,Dimarco2014}. Specifically, we focus on the time evolution of negatively charged particles, i.e. electrons, while we suppose the ions to constitute a fixed background. In this case, the governing equations read:
\begin{equation}\label{eq:vlasov_poisson}
	\begin{split}
		&\partial_t f(t,\xx,\vv,\zz)+ \vv \cdot \nabla_\xx f(t,\xx,\vv,\zz) + \big( \EE(t,\xx,\zz) +\\
		&\qquad \qquad + \vv \times \BB_{ext}(t,\xx) \big) \cdot \nabla_\vv f(t,\xx,\vv,\zz) = \frac{1}{\varepsilon}\mathcal{Q}(f(t,\xx,\vv,\zz)),\\\vspace{0.2cm}
		&-\Delta_\xx \phi(t,\xx,\zz) = \rho(t,\xx,\zz) - \rho_i(t,\xx,\zz), \qquad \EE(t,\xx,\zz) = -\nabla_\xx \phi(t,\xx,\zz),
	\end{split}
\end{equation}
where the ions density is space homogeneous and immutable, i.e. $\rho_i(t,\xx,\zz)=1$, while the electrons density is defined by
\begin{equation}\label{eq:charge_density_intro}
	\rho(t,\xx,\zz) = \int_{\mathbb{R}^{d_v}} f(t,\xx,\vv,\zz)\,d\vv,
\end{equation}
with $\zz \sim p(\zz)$ a random variable modeling the uncertainty, distributed according to a known probability density function $ p(\zz) $. In this formulation, $ \EE(t,\xx,\zz) $, is the self-consistent electric field obtained from the Poisson equation (third line of \eqref{eq:vlasov_poisson}), $ \BB_{ext}(t,\xx) $ is an external magnetic field, assumed to be independent of the uncertainty.
The collision operator $ \mathcal{Q}(f) = \mu(M - f) $ is of relaxation type driving the distribution $ f $ toward a local Maxwellian equilibrium $ M(t,\xx,\vv,\zz) $ and with $\mu$ the collision frequency. The Vlasov equation \eqref{eq:vlasov_poisson} is written in non-dimensional form with $ \varepsilon $ denoting the Knudsen number, which characterizes the relative importance of collisional effects, the Debye length is chosen equal to one in the Poisson equation and thus omitted.

The above depicted dynamics, in the deterministic case, can be approximated by means of different numerical schemes ranging from finite difference, finite volume to semi Lagrangian methods \cite{sonnendrucker1999semi,coughlin2022efficient,crouseilles2016multiscale,crouseilles2004numerical,degond2017asymptotic,russo2009semilagrangian,yang2014conservative,crouseilles2010conservative}. Uncertainty for the Vlasov equation has been considered for instance in \cite{Jin2017,Jin2018,medaglia2023stochastic}. In this paper, we focus on Particle-In-Cell (PIC) schemes \cite{Caflisch2008,Degond2010,Chac2013,Chac2015,albi2024instantaneous} with the aim of designing specific control strategies in the presence of uncertainties which in turns will be handled through stochastic collocation (SC) methods \cite{Xiu}. In PIC simulations, the plasma is represented by a large number of particles that carry physical quantities such as electric charge and mass. These particles move through a computational grid where the electromagnetic fields are solved, \cite{filbet2016asymptotically,filbet2017asymptotically,filbet2003numerical,gu2022hamiltonian}. To tackle the presence of uncertainty, we introduce Gauss-Legendre nodes and use quadrature formulas \cite{Xiu} to compute different quantity of interests such as the expectation and the variance with respect to the uncertain space.
In addition, robust optimization techniques are used to formulate an optimal control problem that accounts for uncertainties \cite{albi2017mean}. Specifically, the goal is to design control strategies that optimize the magnetic field configuration to achieve desired plasma parameters while accounting for the variability in the system. More in details, the control problem we aim to study is given by 
\begin{equation}\label{eq:control_pb2}
	\min_{\BB_{ext}\in \mathcal{B}^{adm}}\mathcal{J}(\BB_{ext}; f^0;f) ,\qquad  \textrm{subject to }\quad \eqref{eq:vlasov_poisson},
\end{equation} 
where $\mathcal{B}^{adm}$ is a set of admissible controls, $f^0$ the initial datum, and $\mathcal{J}(\cdot)$ is a functional which reads as follows
\begin{equation}\label{eq:control_pb}
	\begin{aligned}
		\mathcal{J}(\BB_{ext};f^0;f)= & \int_{0}^{t_f} \left( \mathcal{P}[\mathcal{D}(f,\psi)(t,\zz)] + \right. \\& \left.  + \frac{\gamma}{2} \mathcal{P}\Big[\int_{\Omega}  |\BB_{ext}(t,\xx)|^2  f(t,\xx,\vv,\zz) d\xx d\vv\Big] \right)dt
	\end{aligned}
\end{equation}
where 
\begin{equation}\label{eq:D_continuos} 
	\mathcal{D}(f,\psi)(t,\zz) = \int_{\Omega}    \psi(\xx,\vv,\zz) \left( f(t,\xx,\vv,\zz) - \hat{f}(t,\xx,\vv)  \right) d\xx d\vv ,
\end{equation}
is a running cost function, $\hat{f}$ being a given target distribution discussed next, $\psi(\xx,\vv,\zz)$ a function of the state variables, $\mathcal{P}(\cdot)$ a statistical operator counting for the uncertainties, $\Omega=\Omega_x \times \Omega_v$, $t_f$ a final prescribed time, and  $\gamma$ a weight penalizing the magnitude of the control given by the external magnetic field $\BB_{ext} $. The scope of the functional in \eqref{eq:control_pb} is to force the distribution function or its moments towards desired values through the choice of the function $\psi(\xx,\vv,\zz)$. Similar problems have been investigated from an analytical perspective in recent literature, in the absence of collisions and uncertainties, see, for example, \cite{caprino2012time, knopf2020optimal}. From a numerical standpoint, related approaches have also been explored in \cite{bartsch2024controlling, einkemmer2024suppressing, albi2024instantaneous, EINKEMMER2025}. In particular, in \cite{albi2024instantaneous} the authors of the present article, introduced a piecewise spatial control strategy based on the external magnetic field to confine a plasma governed by the Vlasov–Poisson model. The control targeted both the position and velocity of charged particles and was modeled as spatially constant but time-dependent, applied over a finite set of spatial cells. This setup reflects a physically realistic configuration, in which arbitrarily localized magnetic field values are not admissible. To design an effective instantaneous control strategy, all these constraints were explicitly incorporated into the control functional (see \ref{app:old_control} for further details).

The main objective of this work is to extend the results of \cite{albi2024instantaneous} and to develop a new optimal instantaneous control strategy in the presence of uncertainty. To the best of our knowledge, this is the first study addressing the control of out-of-equilibrium plasmas characterized by uncertain parameters. In greater detail, unlike \cite{albi2024instantaneous}, we first derive a pointwise instantaneous control that acts individually on each particle. Building on this fine-scale information, we then construct a coarse-grained control by averaging the local control fields obtained from the minimization procedure. This yields a piecewise constant-in-space, time-dependent magnetic control field, designed to approximate the effect of the optimal microscopic control over a realistic spatial configuration and in the presence of uncertainties, as motivated in the following.  
We demonstrate that this approach allows for effective steering of the plasma toward desired configurations, without requiring highly complex magnetic field structures.

The remainder of the paper is organized as follows.
In Section~\ref{sec:problem_setting}, we formulate the control problem within a two-dimensional setting, presenting its general structure and the modeling framework adopted.
In Section~\ref{sec:num_methods}, we introduce the numerical methods used to solve the Vlasov–Poisson system under uncertainty, with particular emphasis on Particle-In-Cell schemes and the Gauss–Legendre quadrature for handling stochastic parameters.
In Section~\ref{sec:control}, we discretize the control problem and derive a strategy capable of managing both uncertainty and collisional effects.
Section~\ref{sec:num_esp} presents a series of numerical experiments that validate the effectiveness of the proposed method.
In Section~\ref{sec:conclusion}, we summarize the main results and outline possible directions for future work. Finally, the   \ref{app:old_control} is devoted to the extension of the control strategy proposed in \cite{albi2024instantaneous} to the setting with uncertainty, and to a comparison with the new approach developed in this work, highlighting the improvements and the enhanced performance achieved by the latter. 

\section{Problem setting and numerical methods} \label{sec:problem_setting_numdisc}
In a Tokamak device, plasma confinement is governed by a magnetic field that constrains the motion of charged particles. While particles move almost freely along the field lines, their perpendicular motion is restricted to rapid Larmor gyrations. The magnetic configuration is determined by two components: a dominant toroidal field, generated by coils encircling the central axis of the torus, and a weaker poloidal field, produced both by external coils and by the plasma current. Although smaller in magnitude, the poloidal field plays a crucial role in confinement, since its combination with the toroidal component yields helical magnetic field lines that prevent plasma losses \cite{wesson2011tokamaks}.  

To simplify the analysis, we consider a two-dimensional phase-space framework. In this context, we assume that the magnetic field lines are approximately horizontal, so that the magnetic field component under consideration is  oriented orthogonally to the $xy$-plane. This assumption reflects the predominance of the toroidal component over the poloidal one, and allows us to capture the essential transport dynamics of the plasma while maintaining a tractable model for an axisymmetric toroidal configuration. 

In this section, we first discuss the setting and then we propose a suitable numerical discretization for the Vlasov--Poisson--BGK model with uncertainties. In the following section, based on the discussed discretization, we introduce and study an instantaneous control strategy. 

\subsection{Control of the Vlasov--Poisson--BGK model with uncertainties}\label{sec:problem_setting}
In the proposed configuration, the reference frame $(\xx_\perp, \vv_\perp)$ is chosen such that the external magnetic field remains always orthogonal to the plane in which the time evolution of the single-species plasma distribution function takes place, automatically ensuring that the divergence-free condition is satisfied. This gives
\begin{equation}\label{eq:density}
	f = f(t, \xx_\perp, \vv_\perp, \zz), \qquad f: \mathbb{R}_+ \times \Omega_x \times \Omega_v \times \Omega_z,
\end{equation}
where $\Omega_z \subseteq \mathbb{R}^{d_z}$ denotes the uncertainty space, $\Omega_x \subseteq \mathbb{R}^{d_x}$, $\Omega_v \subseteq \mathbb{R}^{d_v}$ the space and velocity domain, and where we set from now on  $d_x = d_v = 2$.
The evolution of $f(t, \xx_\perp, \vv_\perp, \zz)$ is then governed by the following collisional Vlasov-type equation
\begin{equation}\label{eq:Vlasov}
	\begin{split}
		&\partial_t f + \vv_\perp \cdot \nabla_{\xx_\perp} f + F(t, \xx_\perp, \vv_\perp, \zz) \cdot \nabla_{\vv_\perp} f = \frac{1}{\varepsilon} \mathcal{Q}(f), \\
		&f(0, \xx_\perp, \vv_\perp, \zz) = f_0(\xx_\perp, \vv_\perp, \zz),
	\end{split}
\end{equation}
where  $f_0$ denotes the initial distribution function, which depends on the uncertain parameter $\zz$, assumed to follow a given probability distribution $ p(\zz)$. For brevity, we omitted the explicit dependence of $f$ on $t, \xx_\perp,\vv_\perp$, and  $\zz$ in previous expression \eqref{eq:Vlasov}. The force field acting on the particles is given by:
\begin{equation}
	F(t, \xx_\perp, \vv_\perp, \zz) = \EE(t, \xx_\perp, \zz) + \vv_\perp \times \BB_{ext}(t, \xx_\perp),
\end{equation}
where $\BB_{ext}$ is the external magnetic field, assumed to be independent of the uncertainty $\zz$ as already stated, and $\EE$ is the electric field derived from the solution of the Poisson equation:
\begin{equation}\label{eq:Poisson}
	\EE(t, \xx_\perp, \zz) = -\nabla_{\xx_\perp} \phi(t, \xx_\perp, \zz), \qquad -\Delta_{\xx_\perp} \phi(t, \xx_\perp, \zz) = \rho(t, \xx_\perp, \zz)-1,
\end{equation}
with $\phi$ the electric potential, and the charge density defined as:
\begin{equation}\label{eq:charge_density}
	\rho(t, \xx_\perp, \zz) = \int_{\Omega_v } f(t, \xx_\perp, \vv_\perp, \zz)\, d\vv_\perp.
\end{equation}

The right-hand side of equation \eqref{eq:Vlasov}, $ \mathcal{Q}(f)$, describes the collisions among particles. In this work, we adopt a BGK-type operator defined as:
\begin{equation}\label{eq:BGK_operator}
	\mathcal{Q}(f) = \mu \left( \mathcal{M}(t, \xx_\perp, \vv_\perp, \zz) - f(t, \xx_\perp, \vv_\perp, \zz) \right),
\end{equation}
where $\mu$ is the collision frequency which in general may depend upon the macroscopic quantities and the uncertain parameters, i.e.  $\mu=\mu(\rho,T,\zz)$, and $\mathcal{M}$ is the local Maxwellian equilibrium distribution given by:
\begin{equation}\label{eq:maxwellian}
	\mathcal{M}(t, \xx_\perp, \vv_\perp, \zz) =  \frac{\rho(t, \xx_\perp, \zz)}{2\pi T(t, \xx_\perp, \zz)}  \exp\left( -\frac{|\vv_\perp - \UU(t, \xx_\perp, \zz)|^2}{2T(t, \xx_\perp, \zz)} \right),
\end{equation}
with $\UU$, and $T$ denoting the mean velocity and temperature, respectively, computed as:
\begin{equation}\label{eq:moments}
	\begin{split}
		\UU(t, \xx_\perp, \zz) &= \frac{1}{\rho(t, \xx_\perp, \zz)} \int_{\Omega_v} \vv_\perp f(t, \xx_\perp, \vv_\perp, \zz)\, d\vv_\perp, \\
		T(t, \xx_\perp, \zz) &= \frac{1}{2\rho(t, \xx_\perp, \zz)} \int_{\Omega_v} | \vv_\perp - \UU(t, \xx_\perp, \zz) |^2 f(t, \xx_\perp, \vv_\perp, \zz)\, d\vv_\perp.
	\end{split}
\end{equation}

The final goal is to control the dynamics described by \eqref{eq:Vlasov}-\eqref{eq:Poisson} using an external magnetic field to configure charged particles into a desired setting while keeping them as far as possible from the walls. To achieve this, we start from the following continuous control problem
\begin{equation}\label{eq:control_pb_continuos}
	\min_{\BB^{ext}\in \mathcal{B}_{adm}}    \mathcal{J}(\BB_{ext};f^0,f) ,\qquad 
	\textrm{s.t.}~\eqref{eq:Vlasov}-\eqref{eq:Poisson} \ \text{are satisfied}
\end{equation}
where 
\begin{equation}\label{eq:functional_continuos}
	\begin{split}
		&\mathcal{J}(\BB_{ext};f^0,f)  = \int_{0}^{t_f}\mathcal{P}\ \left[\sum_{\ell \in \{\text{x},\text{v}\}}\mathcal{D}(f,\psi_\ell)(t,\zz) \right] dt +\cr
		&\qquad + \frac{\gamma}{2} \int_{0}^{t_f}\mathcal{P}\left[\int_{\Omega} |\BB_{ext}(t,\xx_\perp) |^2 f(t,\xx_\perp,\vv_\perp,\zz)d\xx_\perp d\vv_\perp \right]dt,
	\end{split}
\end{equation}
with $\Omega = \Omega_x \times \Omega_v$, and 
where 
$\mathcal{D}(\cdot)$ aims at enforcing a specific configuration of the distribution function and of its moments, that is 
\begin{equation}\label{eq:D}
	\begin{split}
		\mathcal{D}(f,\psi_\ell)(t,\zz) = &\frac{\alpha_\ell}{2} | m(f,\psi_\ell)(t,\zz) - \hat{\psi}_\ell |^2 + \frac{\beta_\ell}{2} m_\sigma(f,\psi_\ell)(t,\zz),
	\end{split}
\end{equation}
with $\alpha_\ell,\beta_\ell \geq 0$ weighting parameters for $\ell = \{\text{x},\text{v}\}$, and
\begin{equation}\label{eq:m_sigma} 
	\begin{split}
		&m(f,\psi_\ell)(t,\zz) = \int_{\Omega} \psi_\ell(\xx_\perp,\vv_\perp,\zz) f(t,\xx_\perp,\vv_\perp,\zz) d\xx_\perp d\vv_\perp,\\
		& m_{\sigma}(f,\psi_\ell) (t,\zz) = \int_{\Omega} | \psi_\ell(\xx_\perp,\vv_\perp,\zz)-m(f,\psi_\ell)(t,\zz)|^2 f(t,\xx_\perp,\vv_\perp,\zz) d\xx_\perp d\vv_\perp.
	\end{split}
\end{equation}
The first equation in \eqref{eq:m_sigma} aims at enforcing the moments of the distribution function to assume some given values in different regions of the physical space through concentration of the distribution function at some given target $\hat \psi_\ell$. The second equation tries to minimizes the distance from the average value $m(f,\psi_\ell)$ both in space as well as in velocity space around these targets. 

The statistical operator $\mathcal{P}[\cdot]$ is introduced to ensure the robustness of the control strategy with respect to the uncertainty parameter $ \zz$. One natural choice for  $\mathcal{P}[\cdot]$ is the mathematical expectation in the random space, defined as:
\begin{equation}\label{eq:R_mean}
	\mathcal{P}[\varphi_f(\cdot,\zz)] = \int_{\Omega_z} \varphi_f(\cdot,\zz)\, p(\zz)\, d\zz,
\end{equation}
where $\varphi_f(\cdot,\zz)$ is a measurable quantity depending on the uncertainty $ \zz$.
An alternative measure of robustness that we consider in this work is based on the worst-case scenario idea, which corresponds to minimizing a given macroscopic quantity related to the plasma evolution over all admissible realizations of $\zz$. In this case, the operator takes the form:
\begin{equation}\label{eq:R_max}
	\mathcal{P}[\varphi_f(\cdot,\zz)]=\varphi_f(\cdot,\zz_0), \qquad \text{with }~  \zz_0 = \arg \max_\zz T_b(\cdot,\zz),
\end{equation} 
being $T_b(\cdot,\zz)$ the temperature close to the boundary for a fixed value of the uncertainty. 

\subsection{A stochastic collocation particle-MC method}\label{sec:num_methods}
We focus now on the numerical solution of the Vlasov equation \eqref{eq:Vlasov} in the presence of uncertainty. The proposed numerical approach combines Particle-In-Cell (PIC) techniques for the transport part \cite{chacon2016,filbet2016asymptotically} with Direct Simulation Monte Carlo (DSMC) methods to handle the collisional term \cite{pareschi2001time,pareschi2005numerical,Caflisch2008,Dimarco2010}.
To incorporate uncertainty, several strategies have been proposed in the literature about Vlasov or related kinetic type equations with random inputs, ranging from intrusive Stochastic Galerkin methods \cite{Jingwei2018,Chung2020,Xiao2021,Dimarco2024} to non-intrusive techniques such as Monte Carlo sampling \cite{dimarco2019multi,dimarco2020multiscale,Jin2024}.  In this work we focus on an alternative non-intrusive approach: the stochastic collocation (SC) method \cite{Xiu,NTW} which up to our knowledge it has never been used for treating uncertainties in the context of Vlasov-type equations. 

For a quantity of interest $ \varphi_f(t, \xx_\perp, \vv_\perp, \zz) $, we first define its expected value and variance in the uncertain space as:
\begin{equation}\label{eq:mean}
	\mathbb{E}[\varphi_f](t, \xx_\perp, \vv_\perp) = \int_{\Omega_z} \varphi_f(t, \xx_\perp, \vv_\perp, \zz) p(\zz)\, d\zz,
\end{equation}
\begin{equation}\label{eq:var}
	\sigma^2[\varphi_f](t, \xx_\perp, \vv_\perp) = \int_{\Omega_z} \left(\varphi_f(t, \xx_\perp, \vv_\perp, \zz) - \mathbb{E}[\varphi_f](t, \xx_\perp, \vv_\perp)\right)^2 p(\zz)\, d\zz.
\end{equation}
Additional observable can be defined in the same manner. Typical quantities of interest include the distribution function $f$ itself, as well as its moments: density, mean velocity, and temperature for example. In the latter cases, the dependence on $ \vv_\perp $ is dropped in \eqref{eq:mean}-\eqref{eq:var}. 
We then select $ N_z$ quadrature nodes $ \{ z_k \}_{k=1}^{N_z} $ (where we suppose to use only one index to describe the multidimensional space $\Omega_z$) with corresponding weights $ \{ w_k \} $ for a given quadrature rule adapted to the distribution $p(\zz)$ and solve $ N_z $ independent deterministic Vlasov problems. We also suppose, to easily describe the proposed approach, that uncertainty affects only the initial data, i.e. $ f_0(\xx_\perp, \vv_\perp, \zz) $. For each node we have then:
\begin{equation}\label{eq:vlasov_zi} 
	\partial_t f^k(t,\xx_\perp,\vv_\perp) + \vv_\perp \cdot \nabla_{\xx_{\perp}} f^k + F^k(t,\xx_\perp,\vv_\perp) \cdot \nabla_{v_\perp} f^k 
	= \frac{1}{\varepsilon} \mathcal{Q}(f^k),
\end{equation}
being  $ F^k(t,\xx_\perp,\vv_\perp) = (\EE^k(t,\xx_\perp) + \vv_\perp \times \BB_{ext}(t,\xx_\perp))$, 
with initial data $f_0^k = f_0(\cdot, \cdot, z_k)$. The SC method to approximate expected value and variance in the uncertain space is then obtained by the following algorithm:
\begin{alg}[Stochastic collocation for the Vlasov equation with random inputs]~ \label{alg:MC_2}
	\begin{enumerate}
		\item Consider $ N_z $ collocation nodes $ z_k $ and weights $ w_k $ using Gauss quadrature according to $p(\zz)$.
		\item For each node $ z_k $, solve the deterministic Vlasov problem with the preferred numerical technique to get $ \tilde f^{k,n}$.
		\item Compute the expected value and variance of the quantity of interest. For example in the case of the computation of expectation and variance for the distribution function, one has
		\[
		E_{N_z}[\tilde f^n] = \sum_{k=1}^{N_z} w_k \tilde f^{n,k}, \quad 
		\sigma^2_{N_z}[\tilde f^{n}]= \sum_{k=1}^{N_z} w_k \left(\tilde f^{n,k}- E_{N_z}[\tilde f^{n}]\right)^2.
		\]
		Other statistics with respect the random space and related to the distribution function or its moments can be computed in the same manner.
	\end{enumerate}
\end{alg}
This method benefits from spectral convergence when the solution depends smoothly on the uncertain parameter $\zz$, i.e., the error decays faster than any power of $ N_z^{-1} $, depending on regularity of the solution. The typical error estimate reads, in the case in which the other discretization errors are neglected, as:
\[
|\mathbb{E}[f] - E_{N_z}[\tilde f])| \leq C p^{-2N_z}
\]
where $C$ and $p$ are two constants. 

In order to approximate the $ N_z $ deterministic collisional Vlasov equations \eqref{eq:vlasov_zi}, we rely on a particle method. This approach consists first in approximating the initial condition $ f_0(\xx_\perp, \vv_\perp) $ in \eqref{eq:vlasov_zi} by a discrete sum of Dirac masses (we omit from now on the apex $k$ indicating the point in the random space):
\begin{equation}\label{eq:PIC}
	f_N^0(\xx_\perp, \mathbf{v}_\perp) := \sum_{m=1}^{N} \omega_m \delta\left(\mathbf{x}_\perp - \mathbf{x}_m^0\right) \delta\left(\mathbf{v}_\perp - \mathbf{v}_m^0\right),
\end{equation}
where $ \left(\mathbf{x}_m^0, \mathbf{v}_m^0\right)_{1 \leq m \leq N} $ represent the initial positions and velocities of the particles, and $ \omega_m $ are the associated particle weights.
Next, we introduce a spatial discretization grid composed of $ M_c $ cells $ \mathcal{C}_j $, $ j = 1, \ldots, M_c $, used to compute the electric field, together with a time discretization of the interval $[0, t_f]$ with step size $ h > 0 $.
In this setting, the approximate solution of the Vlasov equation at time level $ n $ is given by the empirical density function:
\begin{equation}\label{eq:approx_density}
	f_N^{n}(\xx_\perp, \vv_\perp) = \sum_{m=1}^{N} \omega_m \delta(\xx_\perp - \xx_m^{n}) \delta(\vv_\perp - \vv_m^{n}),
\end{equation}
where the particle positions and velocities at time level $ n $ are updated through a splitting procedure between the collision and transport steps.

The collision step is defined by:
\begin{equation}\label{collision}
	\begin{cases}
		\partial_t f_N^* = \dfrac{1}{\varepsilon} \mathcal{Q}(f_N^*),\\
		f_N^*(0, \xx_\perp, \vv_\perp) = f_N^n(\xx_\perp, \vv_\perp),
	\end{cases}
\end{equation}
while the transport step is given by:
\begin{equation}\label{vlasov}
	\begin{cases}
		\partial_t f_N^{**} + \vv_\perp \cdot \nabla_{\xx_\perp} f_N^{**} + (\EE^{**} + \vv_\perp \times \BB_{ext}^{**}) \cdot \nabla_{\vv_\perp} f_N^{**} = 0,\\
		f_N^{**}(0, \xx_\perp, \vv_\perp) = f_N^*(\xx_\perp, \vv_\perp),
	\end{cases}
\end{equation}
thus yielding the solution at time $ t^{n+1} $:
\[
f_N^{n+1}(\xx_\perp, \vv_\perp) = \mathcal{T}_{h}\left( \mathcal{Q}_{h}(f_N^{n})(\xx_\perp, \vv_\perp) \right).
\]
The discretization of the collision step gives:
\begin{equation}\label{eq:collision_step}
	f_N^*(\xx_\perp, \vv_\perp) = e^{-\nu h} f_N^{n}(\xx_\perp, \vv_\perp) + (1 - e^{-\nu h}) \mathcal{M}_N^n(\xx_\perp, \vv_\perp),
\end{equation}
where $ \nu = \mu/\varepsilon $ and $ \mathcal{M}_N^n(\cdot) $ is the local Maxwellian distribution \eqref{eq:maxwellian} computed from the knowledge of the moments of $f$ and $^*$ indicates that \eqref{eq:collision_step} furnishes the solution of the sole collisional part \eqref{collision} which will be then used as an initial data for the solution of the Vlasov equation \eqref{vlasov}. It is important to observe that the collisional term modifies only the particle velocities, while preserving the first three moments of $ f $, namely mass, momentum, and energy. This means that $\mathcal{M}$ is constant over the relaxation step \eqref{vlasov}. From a probabilistic viewpoint, relation \eqref{eq:collision_step} is interpreted as follows \cite{MC,Caflisch}: with probability $ e^{-\nu h} $, the velocity $ \vv_m $ of a particle remains unchanged, while with probability $ 1 - e^{-\nu h} $, it is replaced by a velocity sampled from the local Maxwellian. This can be expressed as:
\begin{equation}\label{eq:discr_collision_step}
	\vv_m^{*} = \chi(\eta_m < e^{-\nu h}) \vv_m^n + (1 - \chi(\eta_m < e^{-\nu h})) \left( \UU_m^n + \xi_m \sqrt{T_m^n} \right),
\end{equation}
where $ \eta_m \sim \mathcal{U}([0,1]) $ is uniformly distributed, $ \xi_m \sim \mathcal{N}(0,1) $ is a standard normal random variable, and $ \chi(\cdot) $ denotes the indicator function.
Here, $ \UU_m^n $ and $ T_m^n $ are defined as:
\[
\UU_m^n = \sum_{j=1}^{M_c} \tilde{\UU}_j^n \chi(\xx^n_m \in \mathcal{C}_j), \qquad
T_m^n = \sum_{j=1}^{M_c} \tilde{T}_j^n \chi(\xx^n_m \in \mathcal{C}_j),
\]
where:
\begin{equation}\label{eq:discr_moments}
	\begin{split}
		\tilde{\UU}_j^n &= \dfrac{1}{\tilde{\rho}_j^n} \sum_{m=1}^{N} \omega_m \vv_m^n \chi(\xx_m^n \in \mathcal{C}_j),\\
		\tilde{T}_j^n &= \dfrac{1}{2 \tilde{\rho}_j^n} \sum_{m=1}^{N} \omega_m | \vv_m^n - \tilde{\UU}_j^n |^2 \chi(\xx_m^n \in \mathcal{C}_j),
	\end{split}
\end{equation}
represent the local momentum and temperature at time $ t^n $ within cell $ \mathcal{C}_j $, and
\[
\tilde{\rho}_j^n = \sum_{m=1}^{N} \omega_m \chi(\xx_m^n \in \mathcal{C}_j)
\]
is the local particle density.

The transport phase in phase space is finally recovered by approximating the particle trajectories according to the characteristic curves of the Vlasov equation:
\begin{equation}\label{eq:char_curves}
	\begin{split}
		\frac{d\xx_m(t)}{dt} &= \vv_m(t),\\ 
		\frac{d\vv_m(t)}{dt} &= \EE(t, \xx_m)+\vv_m\times \BB_{ext}(t, \xx_m).
	\end{split}
\end{equation}
Here, the electric field $ \EE(t, \xx) $ is obtained by approximating the Poisson equation with a finite difference method on the spatial grid. Then the value of $ \EE $ at the particle position $ \xx_m $ is taken as the value at the center of the cell $ \mathcal{C}_j $ containing $ \xx_m $. At each time step, the approximated density, needed to solve the Poisson equation, is reconstructed over the spatial cells $ \mathcal{C}_j $ from the updated particle positions and velocities.
To discretize \eqref{eq:char_curves}, we use a semi-implicit first-order scheme \cite{filbet2016asymptotically}:
\begin{equation}\label{eq:filbet_scheme}
	\begin{split}
		\xx_m^{n+1} &= \xx_m^n + h \vv_m^{n+1},\\
		\vv_m^{n+1} &= \vv_m^* + h \left( \vv_m^{n+1} \times \BB_{ext}(t^n, \xx_m^n) + \EE(t^n, \xx_m^n) \right),
	\end{split}
\end{equation}
where $ \BB_{ext}(t^n, \xx_m^n) $ denotes the external magnetic field computed at the particle location, which will be determined in the control problem discussed in the next Section \ref{sec:control}. In the rest of the paper we will drop for simplicity the $_\perp$ in the notation and consequently we will write $\xx$ instead of $\xx_\perp$ and $\vv$ instead of $\vv_\perp$.

\section{Robust feedback control by an external magnetic field}\label{sec:control} 


We now describe the robust feedback control strategy for the uncertain Vlasov--Poisson--BGK system introduced in Section \ref{sec:problem_setting}. We begin our analysis by computing the external magnetic field $\BB_{\text{ext}}$ which is, by hypothesis orthogonal to the plane identified by $\xx$ and consequently we set $\BB_{\text{ext}}=(0,0,B)$. 

We discretize the time interval $[0, t_f]$ into subintervals of length $h$, and solving the following sequence of instantaneous optimal control problems:
\begin{equation}\label{eq:control_pb_discretized} 
	\begin{split}
		&\min_{B \in \mathcal{B}_{adm}}    \mathcal{J}(B;f_N^{0},f_N): =  \int_{t}^{t+h}  \mathcal{P}\left[\sum_{\ell \in \{\text{x},\text{v}\}}  \mathcal{D}(f_N,\psi_\ell)(\tau,\zz)\right] d\tau + \\& \qquad \qquad \quad  + \frac{\gamma}{2} \int_{t}^{t+h}  \mathcal{P} \left[ \int_{\Omega}  |B(\tau,\xx)|^2  f_N(\tau,\xx,\vv,\zz) d\xx d\vv\right]   d\tau,
	\end{split}
\end{equation}
where $\mathcal{P}$ is defined as in \eqref{eq:R_mean}–\eqref{eq:R_max}, and ${B}_{adm} = \{ B | B \in [-M,M], M>0\}$, with $M$ denoting the maximum admissible magnetic field strength. The function $\mathcal{D}(f_N, \psi_\ell)(\tau, \zz)$ is defined in \eqref{eq:D}, and $f_N$ denotes the empirical density. 

Applying a semi-implicit rectangle rule to approximate the time integrals in \eqref{eq:control_pb_discretized}, we obtain the discrete cost functional
\eqref{eq:control_pb_discretized}  as follows
\begin{equation}\label{eq:J_discretized_h} 
	\begin{aligned}
		\mathcal{J}^h(B;f_N^{0},f_N) &= h   \mathcal{P}\left[ \sum_{\ell \in \{\text{x},\text{v}\}} \mathcal{D}(f^{n+1}_N,\psi_\ell)(\tau,\zz)\right] + \\& + \frac{h \gamma}{2}  \mathcal{P} \left[ \int_{\Omega}  |B^{n+1}(\xx)|^2  f_N^{n+1}(\xx,\vv,\zz) d\xx d\vv\right].
	\end{aligned}
\end{equation}

\subsection{Derivation of an approximated feedback control}

Our next step is to insert the explicit form of the empirical distribution \eqref{eq:approx_density} into \eqref{eq:J_discretized_h}, under the constraint imposed by the semi-implicit dynamics \eqref{eq:collision_step}–\eqref{eq:filbet_scheme}. However, the semi-implicit scheme gives rise to strong nonlinear couplings, which make the associated optimality system analytically intractable and prevent a closed-form computation of the control.
To overcome this difficulty, we introduce a simplified explicit integrator. In contrast to the semi-implicit method, the explicit scheme allows us to compute the optimal magnetic field in closed form at each particle location. This field is then inserted into the semi-implicit scheme for the actual evolution of the dynamics. In this way, we preserve the accuracy and stability properties of the semi-implicit scheme, while using the explicit scheme only as a tool for control synthesis, since it is not suitable for the time evolution of the dynamics due to its well-known stability limitations. 

From a practical view point, we replace the numerical scheme \eqref{eq:collision_step}–\eqref{eq:filbet_scheme} with the explicit scheme:
\begin{equation}\label{eq:explicit_scheme}
	\begin{split}
		&x_m^{n+1}(\zz) = x_m^n(\zz) + h v_{x_m}^{n+1}(\zz),\\
		&y_m^{n+1}(\zz) = y_m^n(\zz) + h v_{y_m}^{n+1}(\zz),\\
		&v_{x_m}^{n+1}(\zz) = v_{x_m}^*(\zz)+h v_{y_m}^*(\zz)B^{n+1}_m + h E_{x_m}^n(\zz),\\
		&v_{y_m}^{n+1}(\zz) = v_{y_m}^*(\zz) - h v_{x_m}^*(\zz) B^{n+1}_m + h E_{y_m}^n(\zz),\\
		& v_{x_m}^*(\zz) = \theta_m^{\nu,h} v_{x_m}^n(\zz) + (1-\theta_m^{\nu,h}) (U_{x_m}^n(\zz) +\xi_m \sqrt{T_m^n(\zz)}),\\
		& v_{y_m}^*(\zz) = \theta_m^{\nu,h} v_{y_m}^n(\zz) + (1-\theta_m^{\nu,h}) (U_{y_m}^n(\zz) +\xi_m \sqrt{T_m^n(\zz)}),
	\end{split}
\end{equation}
where $B_m$ indicates the value of the parallel part of the magnetic field at the particle position $(x_m,y_m)$ and with $\theta_m^{\nu,h} = \chi(\eta_m<e^{-\nu h})$ being the stochastic jump process parameter used for handling the collision part. 

By direct computation assuming in \eqref{eq:D} $\psi_\emph{x} = y^{n+1}(\zz)$, $\psi_\emph{v} = v_y^{n+1}(\zz)$, $\hat{\psi}_\emph{x} = \hat{y}$,  $\hat{\psi}_\emph{v} = 0$, being $\hat{y}$ a certain function of the particles state related to the fact that we want to avoid the plasma to reach the boundaries by imposing a target average position and a target average velocity, we get 
\begin{equation}\label{eq:J_discr} 
	\begin{split}
		&\mathcal{J}^{h,N}(B_m) =   h\mathcal{P} \left[\sum_{m=1}^N \left(  \frac{\alpha_\emph{x}}{2} \Big \vert y_m^{n+1}(\zz) -\hat{y} \Big \vert^2 + \frac{\beta_\emph{x}}{2} \Big \vert y_m^{n+1}(\zz) -\bar{y}^n(\zz)\Big \vert^2 \right. \right. + \\
		& \qquad +\left. \left. \frac{\alpha_\emph{v}}{2} \Big \vert v_{y_m}^{n+1}(\zz) \Big \vert^2 + \frac{\beta_\emph{v}}{2} \Big \vert v_{y_m}^{n+1}(\zz) -\bar{v}_y^*(\zz)\Big \vert^2 \right) \right]
		+ \frac{h\gamma}{2} \mathcal{P} \left[\sum_{m=1}^N  \vert B_m\vert^2\right],
	\end{split}
\end{equation}
with $\alpha_\emph{x},\alpha_\emph{v}, \beta_\emph{x},\beta_\emph{v}, \gamma \geq 0$, and where 
\begin{equation*}
	\bar{y}^n(\zz) = \frac{1}{N} \sum_{m=1}^N y_m^n(\zz), \qquad 	\bar{v}_y^*(\zz) = \frac{1}{N} \sum_{m=1}^N v_{y_m}^*(\zz).
\end{equation*} 
Now, the following Proposition holds true.
\begin{proposition}\label{prop:istctrl}
	Assume the parameters to scale as 
	\begin{equation}\label{eq:scaling_new}
		\alpha_\emph{x} \rightarrow \frac{\alpha_\emph{x}}{h}, \qquad \beta_\emph{x} \rightarrow \frac{\beta_\emph{x}}{h}, \qquad \gamma \rightarrow \gamma h, \qquad \nu \rightarrow \frac{\nu}{h},
	\end{equation}
	then the feedback control $B_m$ associated to \eqref{eq:J_discr} reads as follows
	\begin{equation}\label{eq:L2_control_new}
		B_m = \mathbb{P}_{[-M,M]}\left( \frac{\mathcal{P}[\mathcal{R}^{m,n}_{\emph{x}}(\zz)  + \mathcal{R}^{m,n}_{\emph{v}}(\zz)] }{\gamma +\mathcal{P}[\mathcal{S}^{m,n}_{\emph{x}}(\zz) + \mathcal{S}^{m,n}_{\emph{v}}(\zz)] } \right), 
	\end{equation}
	where for any $m=1,\ldots,N$, $\gamma>0$, we have 
	\begin{equation}\label{eq:terms_in_B_new}
		\begin{split}
			\mathcal{R}^{m,n}_{\emph{x}}(\zz) = &  \alpha_\emph{x} (y_m^n(\zz) + h(v_{y_m}^{\star,n}(\zz)+hE_{y_m}^n(\zz))-\hat{y})v^{\star,n}_{x_m}(\zz) +  \\ &   +\beta_\emph{x} (y_m^n(\zz) +h(v_{y_m}^{\star,n}(\zz)+hE_{y_m}^n(\zz))-\bar{y}^n(\zz)) v^{\star,n}_{x_m}(\zz),\\
			\mathcal{R}^{m,n}_{\emph{v}}(\zz) = &  \alpha_\emph{v} (v_{y_m}^{\star,n}(\zz) +hE_{y_m}^n(\zz))v^{\star,n}_{x_m}(\zz)  + \\ &   +\beta_\emph{v} (v_{y_m}^{\star,n}(\zz)+hE_{y_m}^n(\zz)-\bar{v}_y^{\star,n}(\zz)) v^*_{x_m}(\zz),\\
			\mathcal{S}^{m,n}_{\emph{x}}(\zz) = &  (\alpha_\emph{x}+\beta_{\emph{x}})  (h v_{x_m}^{\star,n}(\zz))^2,\\
			\mathcal{S}^{m,n}_{\emph{v}}(\zz) = &  (\alpha_\emph{v}+\beta_{\emph{v}})  h ( v_{x_m}^{\star,n}(\zz))^2,
		\end{split}
	\end{equation}
	with $\alpha_{\emph{x}}, \alpha_{\emph{v}}, \beta_{\emph{x}}, \beta_{\emph{v}} \geq0$, $\mathbb{P}_{[-M,M]}(\cdot)$ denoting the projection over the interval $[-M,M]$,  $\mathcal{P}(\cdot)$ as in \eqref{eq:R_mean}-\eqref{eq:R_max}, and with 
	\begin{equation}\label{eq:v_collisions_scaled}
		\begin{split}
			& v_{x_m}^{\star,n}(\zz) = \theta_m^{\nu} v_{x_m}^n(\zz) + (1-\theta_m^{\nu}) (U_{x_m}^n(\zz) +\xi_m \sqrt{T_m^n(\zz)}),\\
			& v_{y_m}^{\star,n}(\zz) = \theta_m^{\nu} v_{y_m}^n(\zz) + (1-\theta_m^{\nu}) (U_{y_m}^n(\zz) +\xi_m \sqrt{T_m^n(\zz)}),
		\end{split}
	\end{equation}
	being $\theta_m  = \chi (\eta_m < e^{-\nu})$, with $\eta \sim \mathcal{U}([0,1])$.
	In the limit $h\to 0$ the control at the continuous level reads, 
	\begin{equation}\label{eq:L2_control_continuos_new}
		B_m(t) =  \mathbb{P}_{[-M,M]}\left(\frac{1}{\gamma}\left(  \mathcal{P}[\mathcal{R}_{\emph{x}}^m(t,\zz)  +  \mathcal{R}_{\emph{v}}^m (t,\zz)]\right)  \right), 
	\end{equation}
	with 
	\begin{equation}
		\begin{split}
			\mathcal{R}_{\emph{x}}^m(t,\zz) = &  [\alpha_\emph{x} (y_m(t,\zz) -\hat{y}(t,\zz)) +\beta_\emph{x} (y_m(t,\zz)-\bar{y}(t,\zz))] v^{\star}_{x_m}(t,\zz) ,\\
			\mathcal{R}_{\emph{v}}^m(t,\zz) = &  [\alpha_\emph{v} v_{y_m}^{\star}(t,\zz) +\beta_\emph{v} (v_{y_m}^{\star}(t,\zz)-\bar{v}_y^*(t,\zz))] v^{\star}_{x_m}(t,\zz),
		\end{split}
	\end{equation}
	being 
	\begin{equation*}
		\begin{split}
			& v_{x_m}^{\star}(t,\zz) = \theta_m^{\nu} v_{x_m}(t,\zz) + (1-\theta_m^{\nu}) (U_{x_m}(t,\zz) +\xi_m \sqrt{T_m(t,\zz)}),\\
			& v_{y_m}^{\star}(t,\zz) = \theta_m^{\nu} v_{y_m}(t,\zz) + (1-\theta_m^{\nu}) (U_{y_m}(t,\zz) +\xi_m \sqrt{T_m(t,\zz)}),
		\end{split}
	\end{equation*}
	where  $\theta_m  = \chi (\eta_m < e^{-\nu})$, with $\eta \sim \mathcal{U}([0,1])$.
\end{proposition}
\begin{proof}
	We introduce the augmented Lagrangian 
	\begin{equation}
		\mathcal{L}(B_m,\lambda_m) = \mathcal{J}^{h,N}(B_m) + \lambda_m (\vert B_m \vert - M), 
	\end{equation}
	with $\lambda_m$ the Lagrangian multiplier. 
	Then, for any $m=1,\ldots,N$, we solve the optimality system 
	\begin{equation}\label{eq:Lagr_system}
		\begin{cases}
			\partial_{B_m}\mathcal{L}(B,\lambda)=0\\
			\left(\partial_{\lambda_m}\mathcal{L}(B,\lambda) = 0\ \text{ and }\lambda_{m}\geq 0\right) \text{ or } \left( \partial_{\lambda_m}\mathcal{L}(B,\lambda) <0 \text{ and } \lambda_{m} = 0\right).
		\end{cases}
	\end{equation} 
	We note that both the operators $\mathcal{P}[\cdot]$ in \eqref{eq:R_mean}-\eqref{eq:R_max} satisfies 
	\begin{equation*}
		\partial_{B_m} \mathcal{P}[\psi(\cdot,\zz)] = \mathcal{P}[	\partial_{B_m}\psi(\cdot,\zz)],
	\end{equation*}
	for any function $\psi(\cdot,\zz)$.
	Hence, from the first equation in \eqref{eq:Lagr_system} we get 
	\begin{equation*}
		\begin{split}
			\partial_{B_m} \mathcal{L} =&  h\mathcal{P}\left[   \alpha_\text{x} \left( y_m^n + h(v_{y_m}^*-hv_{x_m}^* B_m + h E_{y_m}^n )-\hat{y} \right)(-h^2 v_{x_m}^*) +     \right. \\ 
			&   \left. + \beta_\text{x} \left( y_m^n + h(v_{y_m}^* - h v_{x_m}^* B_m +h E_{y_m}^n) - \bar{y}^n  \right) (-h^2 v_{x_m}^*)  + \right.  \\
			& \left. + \alpha_\text{v} \left( v_{y_m}^* -h v_{x_m}^* B_m^n + h E_{y_m}^n\right) (-hv_{x_m}^*) + \right. \\
			& \left. + \beta_\text{v} \left( v_{y_m}^* -h v_{x_m}^* B_m^n + h E_{y_m}^n - \bar{v}_{y}^*\right) (-hv_{x_m}^*) \right] + h\gamma B_m  +\lambda_m\sign(B_m) = 0,
		\end{split}
	\end{equation*} 
	where we omit for simplicity the explicit dependence of the variables on $\zz$, and where the operator $\mathcal{P}(\cdot)$ is defined as in \eqref{eq:R_mean}-\eqref{eq:R_max}. 
	Then, if $\lambda_m \geq 0$, we get from the optimality condition \eqref{eq:Lagr_system} $|B_m|=M$. Consequently considering the two cases $B_m=M$ and $B_m=-M$ separately, under the scaling in \eqref{eq:scaling_new}, we get, setting $B_m=-M$ 
	\begin{equation}\label{eq:proof_1} 
		\begin{split}
			\frac{\lambda_m}{h^2} =&  \mathcal{P}\left[   \alpha_\text{x} \left( y_m^n + h(v_{y_m}^{\star,n}+hv_{x_m}^{\star,n} M + h E_{y_m}^n )-\hat{y} \right)(- v_{x_m}^{\star,n})  +    \right. \\ 
			&   \left. + \beta_\text{x} \left( y_m^n + h(v_{y_m}^{\star,n} + h v_{x_m}^{\star,n} M +h E_{y_m}^n) - \bar{y}^n  \right) (-v_{x_m}^{\star,n}) +  \right.  \\
			& \left. + \alpha_\text{v} \left( v_{y_m}^{\star,n} +h v_{x_m}^{\star,n} M + h E_{y_m}^n\right) (-v_{x_m}^{\star,n}) + \right. \\
			& \left. + \beta_\text{v} \left( v_{y_m}^{\star,n} + h v_{x_m}^{\star,n} M + h E_{y_m}^n - \bar{v}_{y}^{\star,n}\right) (-v_{x_m}^{\star,n}) \right] - \gamma M\geq 0,
		\end{split}
	\end{equation}  
	with $v_{x_m}^{\star,n},v_{y_m}^{\star,n}$ as in \eqref{eq:v_collisions_scaled}.
	From \eqref{eq:proof_1} we recover
	\begin{equation*}
		\frac{\mathcal{P}[\mathcal{R}^{m,n}_{\text{x}}(\zz)  + \mathcal{R}^{m,n}_{\text{v}}(\zz)] }{\gamma +\mathcal{P}[\mathcal{S}^{m,n}_{\text{x}}(\zz) + \mathcal{S}^{m,n}_{\text{v}}(\zz)] }  \leq -M, 
	\end{equation*}
	with $\mathcal{R}^{m,n}_{\ell}(\zz)$, $\mathcal{S}^{m,n}_{\ell}(\zz)$ for $\ell \in \{\text{x},\textrm{v}\}$ as in \eqref{eq:terms_in_B_new}.
	If conversely $\lambda_m \geq 0$ and $B_k=M$, by scaling the parameters as in \eqref{eq:scaling_new}, following a similar argument we have
	\begin{equation}
		\frac{\mathcal{P}[\mathcal{R}^{m,n}_{\text{x}}(\zz)  + \mathcal{R}^{m,n}_{\text{v}}(\zz)] }{\gamma +\mathcal{P}[\mathcal{S}^{m,n}_{\text{x}}(\zz) + \mathcal{S}^{m,n}_{\text{v}}(\zz)] } \geq M.
	\end{equation}
	If finally $\lambda_m=0$, and by assuming the parameters to scale as in \eqref{eq:scaling_new}, from the first equation in \eqref{eq:Lagr_system} we have 
	\begin{equation}\label{eq:control_L2} 
		B_m =    \frac{\mathcal{P}[\mathcal{R}^{m,n}_{\text{x}}(\zz)  + \mathcal{R}^{m,n}_{\text{v}}(\zz)] }{\gamma +\mathcal{P}[\mathcal{S}^{m,n}_{\text{x}}(\zz) + \mathcal{S}^{m,n}_{\text{v}}(\zz)] },
	\end{equation}
	and 
	\begin{equation}
		-M\leq B_m \leq M.
	\end{equation}
	All in all we get $B_m$ defined as in \eqref{eq:L2_control_new}.  Finally, in the limit $h\to 0$ we recover equation \eqref{eq:L2_control_continuos_new}.
\end{proof}

\subsection{Realization of the control via spatial interpolation}

Proposition \ref{prop:istctrl}  provides a feedback control law based on the pointwise evaluation of the magnetic field at the position of each particle. However, such a requirement is not feasible from a technological perspective in realistic applications. To relax this constraint, we introduce an interpolation strategy by defining a set of fictitious spatial cells $C_k$, $k=1,\ldots,N_c$, partitioning the physical domain. Within this framework, the magnetic field acting on a particle located at position $\mathbf{x}_m$ is approximated by the constant value of the field within the cell $C_k$ that contains the particle. While this choice increases the complexity of the control problem, it leads to a more realistic and implementable setting. As mentioned in Section \ref{sec:problem_setting}, we adopt the simplification that the magnetic field lines in the two dimensional setting considered can be approximated by horizontal lines, since the contribution of the poloidal component of the magnetic field is much more smaller than the toroidal one. Following the idea introduced in \cite{lu2025piecewise}, we suppose the grid on which the magnetic field is reconstructed aligns with the magnetic field lines. Thus, for each $k = 1, \ldots, N_c$, we define the piecewise constant control  $\hat{B}_k$ over the cell $C_k$ by interpolating the pointwise control values $B_m$  within the corresponding cell. For convenience, we denote by  
\begin{equation}\label{eq:L2_control_interp}
	\BB = \mathcal{I}^\ell(\BB_N, \xx_c),
\end{equation}
the vector $\BB = [\hat{B}_1, \ldots, \hat{B}_{N_c}]$ collecting the interpolated controls $\hat{B}_k$. Here, $\mathcal{I}^\ell(\cdot)$ denotes a piecewise-constant interpolation operator of order $\ell$, $\xx_c$ is the vector of the positions of the centers of the cells $C_k$, and $\BB_N = [B_1, \ldots, B_N]$ is the vector of pointwise controls $B_m$ as defined in equation~\eqref{eq:L2_control_continuos_new}. 
We observe that the penalization term acting on the control is weighted by the distribution function $f$. Consequently, in regions where $f$ is small, the corresponding magnetic field intensity $B$ could, in principle, take large values without significantly affecting the overall cost. By introducing fictitious cells and enforcing a piecewise-constant control over each of them, the penalization acts uniformly at the cell level, ensuring a globally balanced regularization of the control. More generally, in the absence of spatial discretization, areas with low particle density contribute only marginally to the overall system dynamics. Hence, allowing locally larger control amplitudes in such regions would not substantially modify the global behavior of the solution.
Algorithm \ref{alg:num_scheme} outlines the described procedure.   
\begin{alg}[PIC scheme and control in uncertain setting]~ \label{alg:num_scheme}
	\begin{itemize}
		\item Consider $N_z$ collocation nodes $z_j$ as outline in Algorithm \ref{alg:MC_2}.
		\item  For each node $z_j$, sample the position and velocity of $N$ particles $(x^0_m,v^0_m)$ from the initial density, and compute the quantities in \eqref{eq:terms_in_B_new}. 
		\item  Compute the value of $B_m$ as in \eqref{eq:L2_control_continuos_new}, choosing a suitable statistical operator to account for the uncertainty.
		\item Interpolate the values of the control $B_m$ over the fictitious spatial cells $C_k$ to obtain a piecewise constant control $\hat{B}_k$ as in \eqref{eq:L2_control_interp}.
		\item    \texttt{for} $n=1$ \texttt{to}  $N_t$
		\begin{enumerate}
			\item \texttt{for} $j=1$ \texttt{to} $N_z$
			\begin{itemize}
				\item Associate to each particle in cell $C_k$ the value of the control $\hat{B}_k$.  
				\item Update the particle position and velocities according to the semi-implicit PIC scheme \eqref{eq:filbet_scheme}, using the control previously computed.
				\item Recompute the quantities in \eqref{eq:terms_in_B_new}. 
			\end{itemize}
			\texttt{end for}
			\item Compute the value of $B_m$ as in \eqref{eq:L2_control_continuos_new}, choosing a suitable statistical operator to account for the uncertainty.
			\item Interpolate the values of the control $B_m$ over the fictitious spatial cells $C_k$ to obtain a piecewise constant control $\hat{B}_k$ as in \eqref{eq:L2_control_interp}.
			\item Reconstruct the quantities of interest as highlighted in Algorithm \ref{alg:MC_2}. 
		\end{enumerate}
		\texttt{end for}
	\end{itemize}
\end{alg}

We emphasize that the control expression introduced in equation~\eqref{eq:L2_control_interp} is, in general, sub-optimal with respect to the reduced-horizon control problem \eqref{eq:J_discretized_h}--\eqref{eq:explicit_scheme} in the limit $h \to 0$, as it relies on an additional interpolation of the dynamics. Nevertheless, since the main objective is to confine the entire ensemble of particles within the physical domain and to assess the robustness of the control strategy, we employ the control defined in \eqref{eq:L2_control_interp} within the semi-implicit dynamics \eqref{eq:discr_collision_step}--\eqref{eq:filbet_scheme}, demonstrating its effectiveness in steering the plasma toward a desired configuration.

A comparison with an alternative control strategy, originally proposed in \cite{albi2024instantaneous} and extended here to the uncertain setting, is provided in \ref{app:old_control}. This highlights the improvements achieved by the feedback-based approach presented in this section.

\begin{remark}
	We remark that the control $B_m$
	in equation \eqref{eq:L2_control_new} is of feedback type, since it depends on the particle positions and velocities at time $t_n$. From a numerical point of view, these quantities are readily available at each iteration of the scheme and are used to compute the control online, as outlined in Algorithm \ref{alg:num_scheme}. Physically, this would correspond to accessing particle positions and velocities at each time step, which may not be directly feasible in practice. In general, however, the control could be reconstructed from sampled particle positions and velocities, so that knowledge of the full microscopic state of the system would not be required.	
	To further bridge the gap between the idealized numerical setting and more realistic scenarios, we also include a test in which the control field is computed and then held fixed for a prescribed number of time steps before being updated again. This strategy mimics the effect of limited measurement or actuation frequency and demonstrates that the feedback approach remains effective under such practical constraints.
\end{remark}
\begin{remark}
	Note that by increasing the number of fictitious spatial cells $C_k$, it is possible to obtain a continuous representation of the control value, reconstructed on the physical spatial grid.
\end{remark}
\section{Numerical experiments}\label{sec:num_esp} 
In this section, we present several numerical experiments aimed at demonstrating the effectiveness of the instantaneous control strategy, as described in the previous sections, in steering the plasma toward desired configurations.  
In particular, we focus on a variant of the classical Sod shock tube test and a modified version of the Kelvin--Helmholtz instability, both enhanced by the inclusion of collisions, electromagnetic effects, and uncertainty. We compare the system's behavior with and without control to highlight the impact of the proposed strategy.  
Regarding uncertainty, we assume randomly distributed initial data following a uniform distribution over the interval $[0,1]$, and we sample $ N_z $ Gaussian quadrature nodes accordingly.

\subsection{2D Sod shock tube test}\label{sec:2D_sod_shock} 
We focus on a two-dimensional Sod shock tube test, where we denote the position and velocity variables by $ \xx = (x, y) $ and $ \vv = (v_x, v_y) $, respectively. We consider the spatial domain $ (x, y) \in [0, 1.5]^2 $, imposing reflective boundary conditions in the $ y $-direction and periodic boundary conditions in the $ x $-direction. The initial distribution is given by
\begin{equation}\label{eq:sod_2D_initial_distr}
	f_0(\xx,\vv,z) =\frac{\rho_0(\xx)}{2\pi T_0(\xx,z)}\exp\left(-\frac{ \vert \vv\vert ^2}{2T_0(\xx,z)}\right),
\end{equation}
with the initial density and temperature defined as
\begin{equation}\label{eq:rho0_T0_2D}
	\begin{split}
		\rho_0(\xx) =& 
		\begin{cases}
			0.125, & \text{if } y \in [0, 0.5) \cup (1, 1.5], \\
			1,     & \text{if } y \in [0.5, 1], \\
		\end{cases} \\[5pt]
		T_0(\xx,z) = &
		\begin{cases}
			0.1 + 0.25z, & \text{if } y \in [0, 0.5) \cup (1, 1.5], \\
			1 + 0.25z,   & \text{if } y \in [0.5, 1]. \\
		\end{cases}
	\end{split}
\end{equation}
The simulations are performed using $ N = 10^7 $ particles and a grid of $ m_x \times m_y $ cells, with $ m_x = m_y = 128 $, for the reconstruction of macroscopic quantities. The final time is set to $ t_f = 2 $, with a time step $ h = 0.05 $.
At each time step, for every node $z_k$, the empirical density \eqref{eq:approx_density} is reconstructed using a piecewise constant representation. The quantities of interests, such as the mean or the variance, are then computed following the procedure outlined in Algorithm \ref{alg:MC_2}.
\subsubsection{Accuracy of the stochastic collocation method}\label{app:gauss_legendre} 
First we validate the accuracy of the stochastic Gauss-Legendre formula when applied to the solution of the collisional Vlasov--Poisson equation with uncertainty, in particular in the context of a control strategy based on a worst-case scenario. We denote by $ E_{N_z^{\text{ref}}}[\rho](T,\xx) $ the reference solution obtained with $ N_z^{\text{ref}} = 256 $ Gauss--Legendre nodes, and by $ E_{N_z}[\rho](T,\xx) $ the corresponding approximation obtained with a lower number of nodes.  
The error at time $ T = 0.2 $, as a function of $ N_z $, is defined as
\begin{equation}\label{eq:error}
	err_{N_z}(T) =  \left\Vert  E_{N_z^{\text{ref}}}[\rho](T,\xx) - E_{N_z}[\rho](T,\xx) \right\Vert_\infty.
\end{equation}
Figure~\ref{fig:accuracy} displays the error defined in~\eqref{eq:error} for two collisional regimes: $ \nu = 0 $ (left) and $ \nu = 1000 $ (right).  
We test the controlled case, where the magnetic field is defined as in~\eqref{eq:L2_control_interp}, using the operator $ \mathcal{P}(\cdot) $ defined either as in~\eqref{eq:R_max} (referred to as $ B_{\text{max}} $) or as in~\eqref{eq:R_mean} (referred to as $ B_{\text{mean}} $).
The results confirm that the method exhibits spectral accuracy with respect to the number of quadrature nodes $ N_z $.

\begin{figure}[h!]
	\centering
	\includegraphics[width=0.35\linewidth]{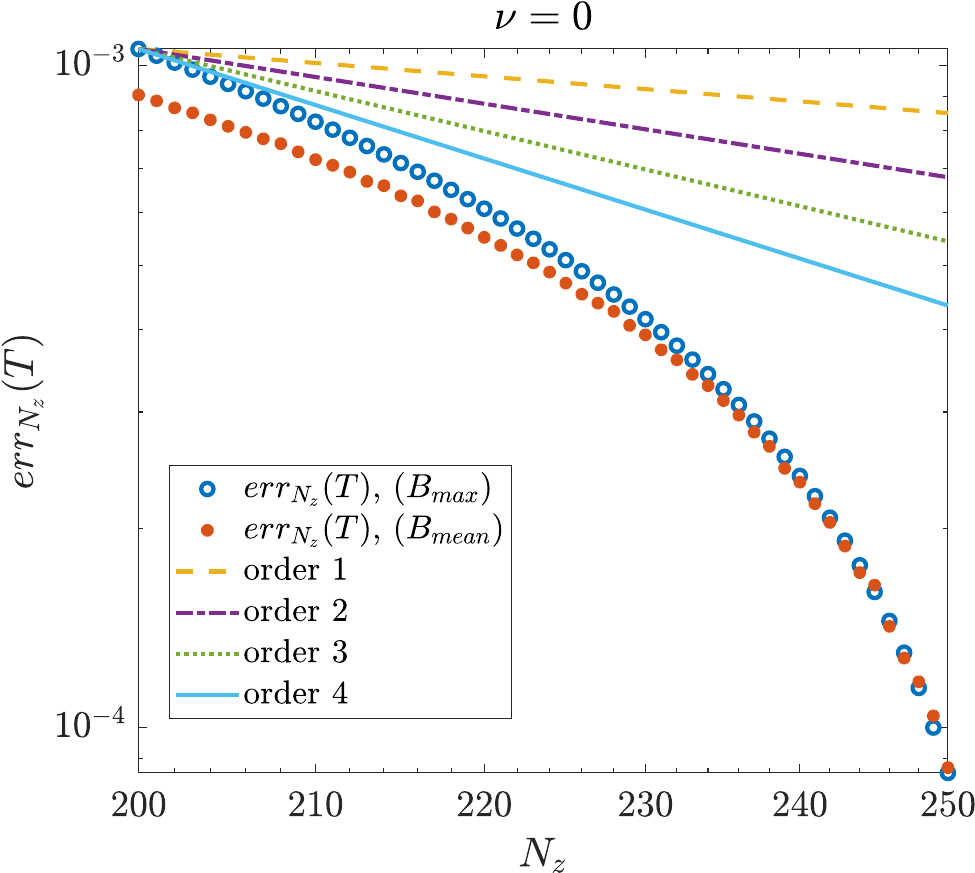}
	\includegraphics[width=0.35\linewidth]{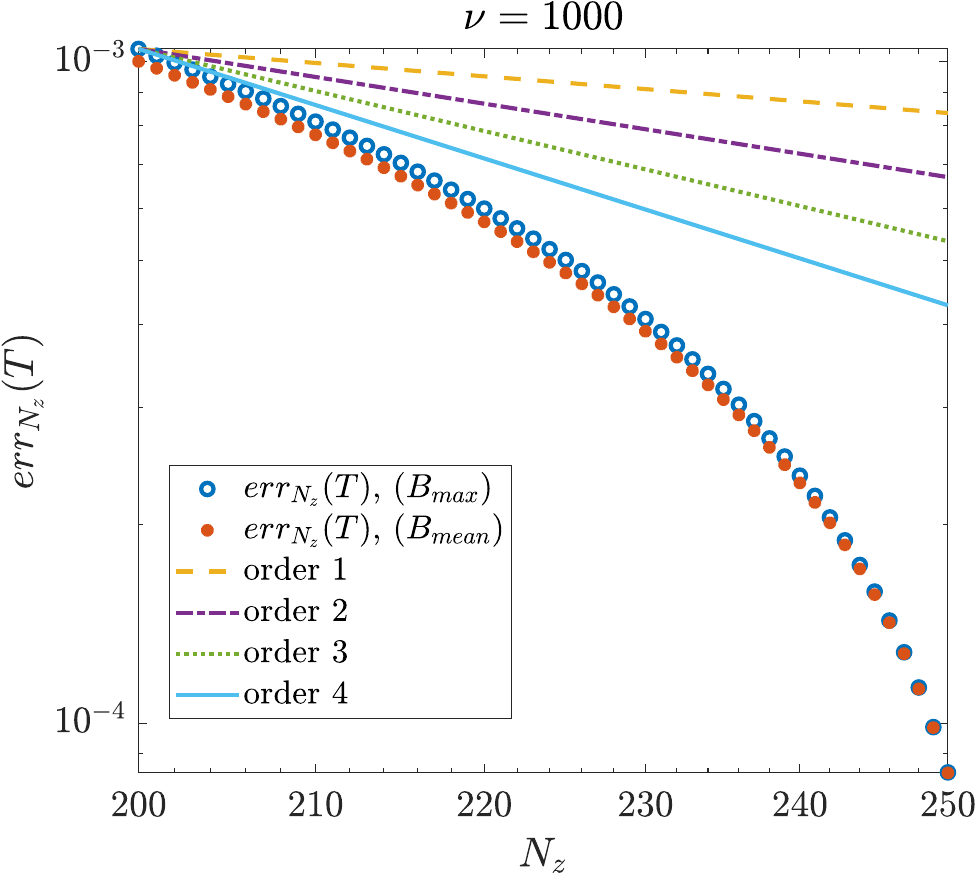}  
	\caption{Two-dimensional Sod shock tube test with control. Norm of the error for $\nu = 0$ (on the left), $\nu = 1000$ (on the right). }
	\label{fig:accuracy}
\end{figure} 

\subsubsection{Effectiveness of the control strategy}

To assess the effectiveness of the control strategy, we compute the thermal energy at the boundaries for each value $ z_i $, with $ i = 1, \ldots, N_z $, as
\begin{equation}\label{eq:energy}
	\mathcal{E}_b^n(z_i) = \frac{1}{2 N N_b} \sum_{\mathcal{C}_j \in \Omega_b} \sum_{m=1}^N \left\vert \vv_m^n(z_i) - U_b^n(z_i) \right\vert^2 \chi(\xx_m^n \in \mathcal{C}_j),
\end{equation}
\begin{figure}[h!]
	\centering
	\includegraphics[width=0.31\linewidth]{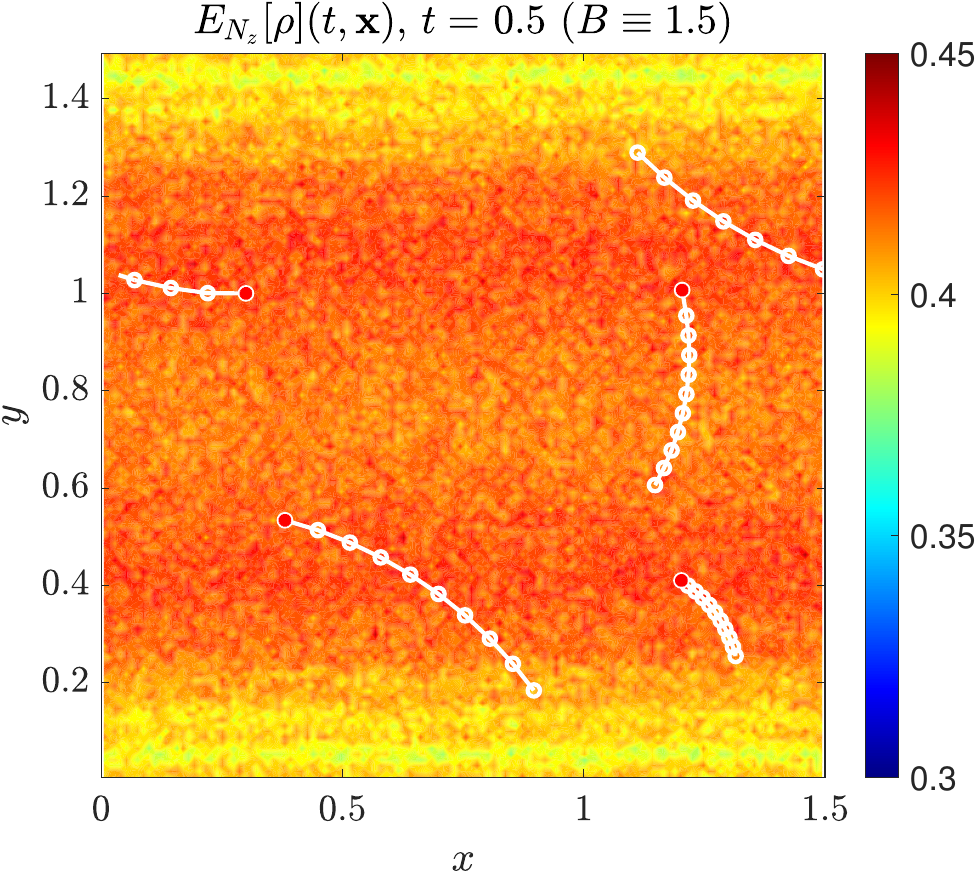}
	\includegraphics[width=0.31\linewidth]{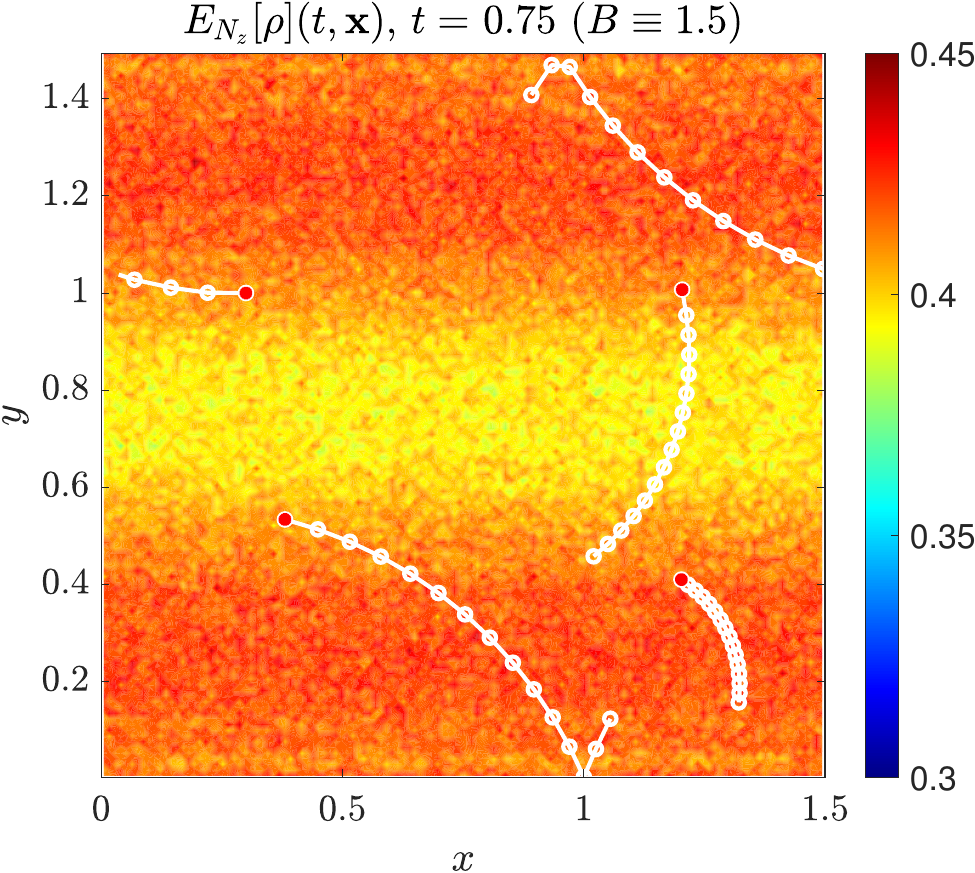}
	\includegraphics[width=0.31\linewidth]{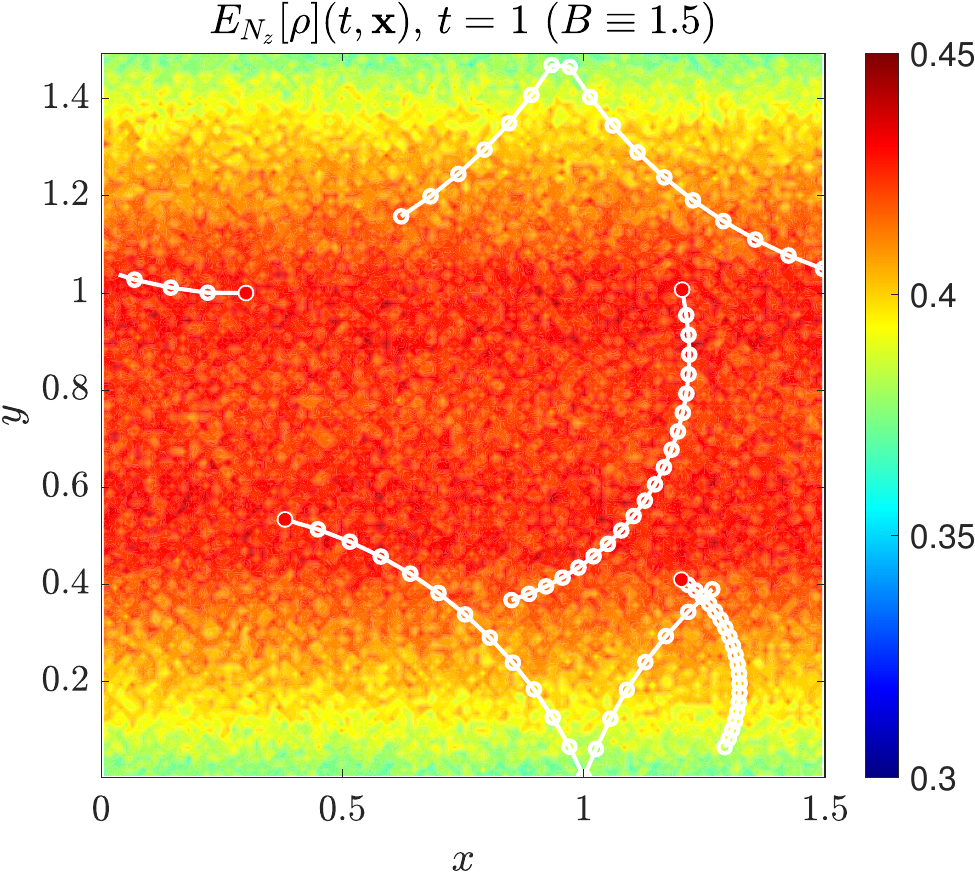}\\
	\includegraphics[width=0.3\linewidth]{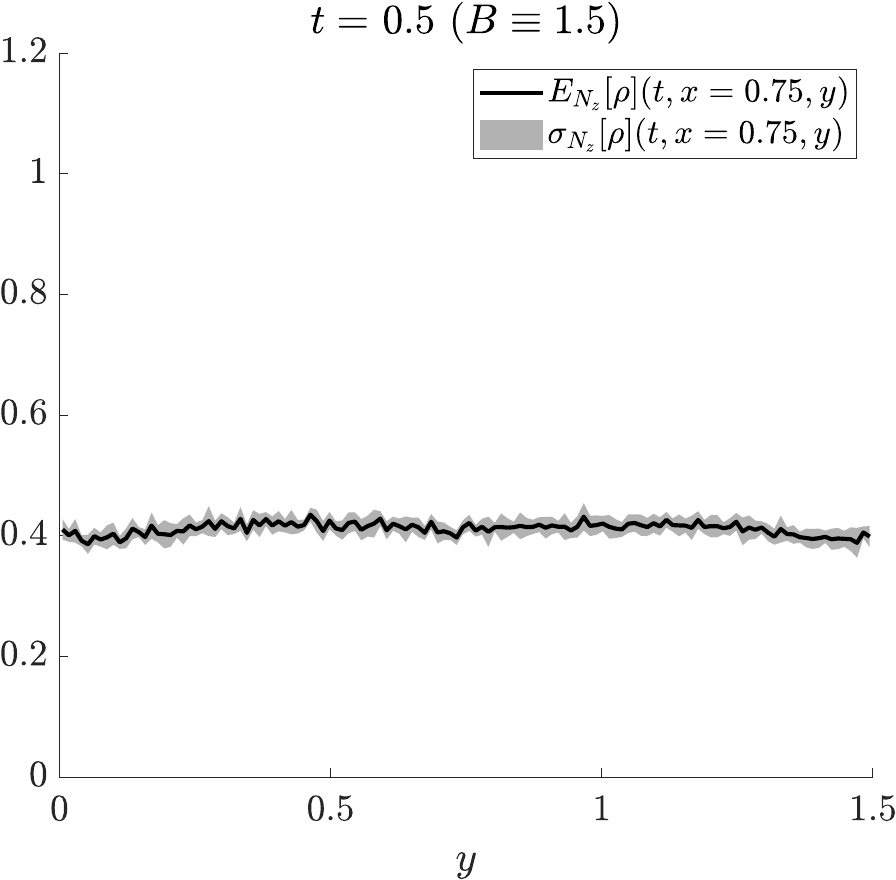}
	\includegraphics[width=0.3\linewidth]{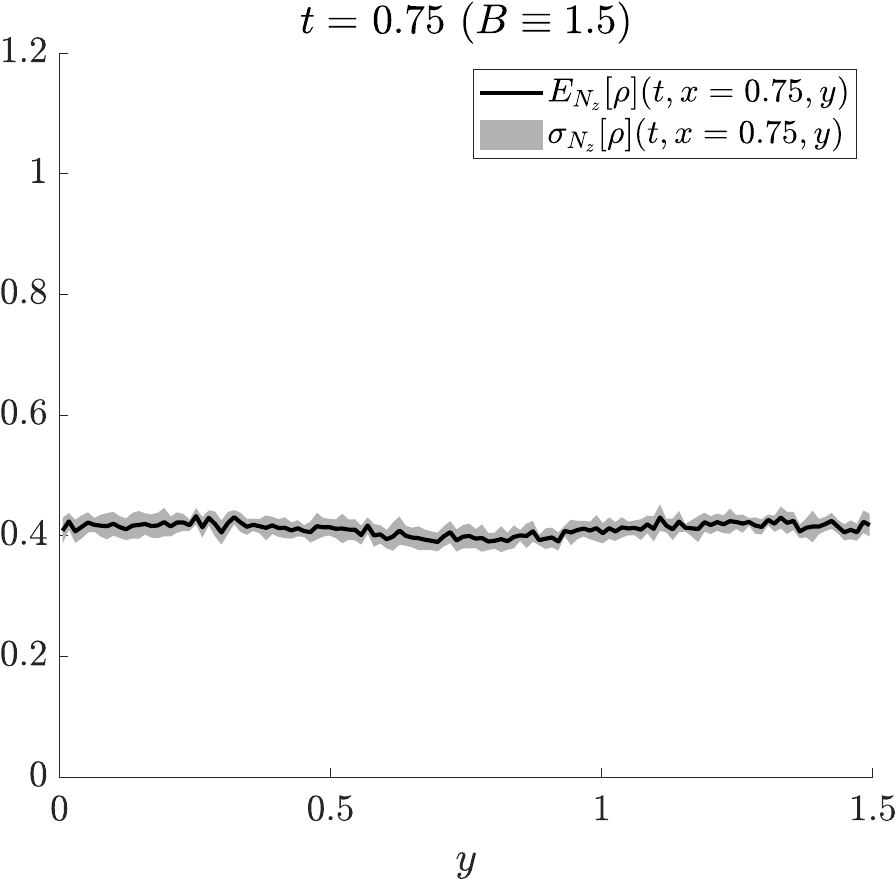}
	\includegraphics[width=0.3\linewidth]{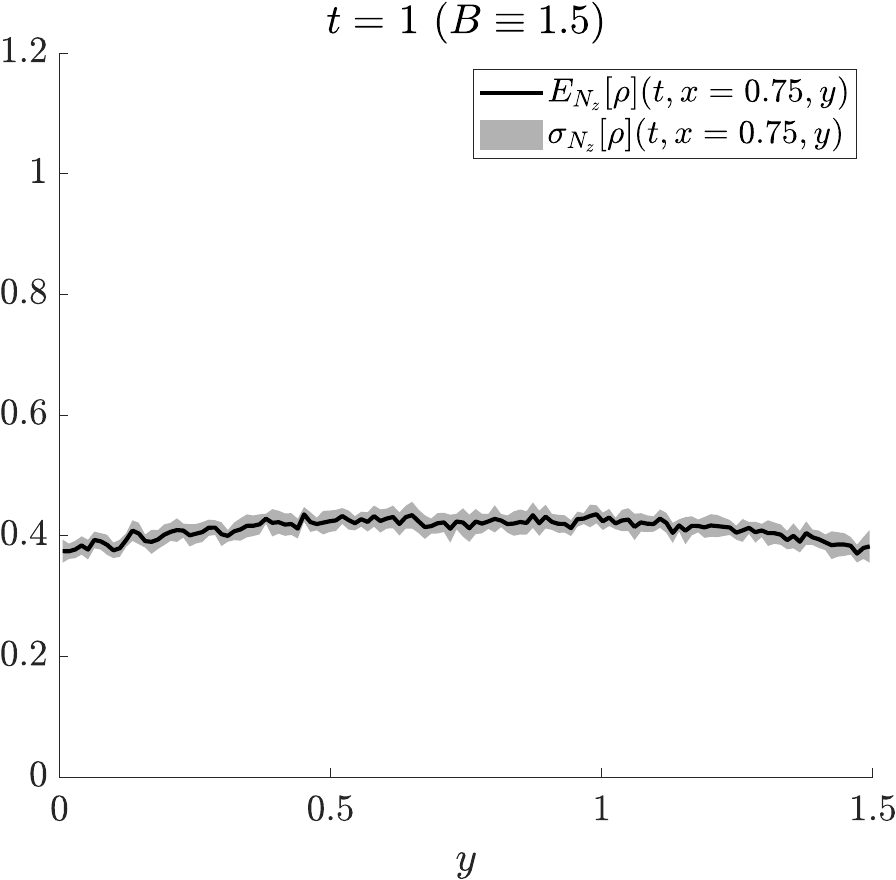}
	\caption{Two-dimensional Sod shock tube test with $ B(t,\xx) = 1.5 $ and $ \nu = 0 $.  Top row: snapshots of the mean density at different time instants. Bottom row: slices of the mean density at $ x = 0.75 $, with the corresponding standard deviation shown as a shaded area.
	}
	\label{fig:sod2D_B15_nu0}
\end{figure}
\begin{figure}[h!]
	\centering
	\includegraphics[width=0.31\linewidth]{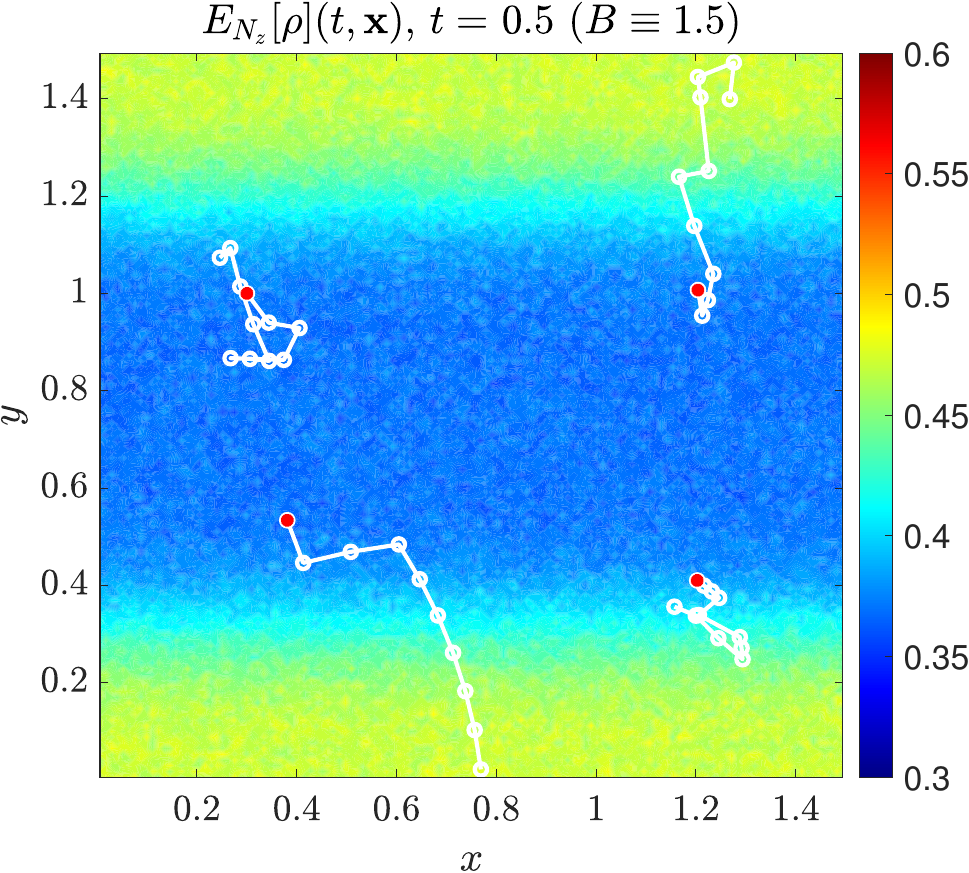}
	\includegraphics[width=0.31\linewidth]{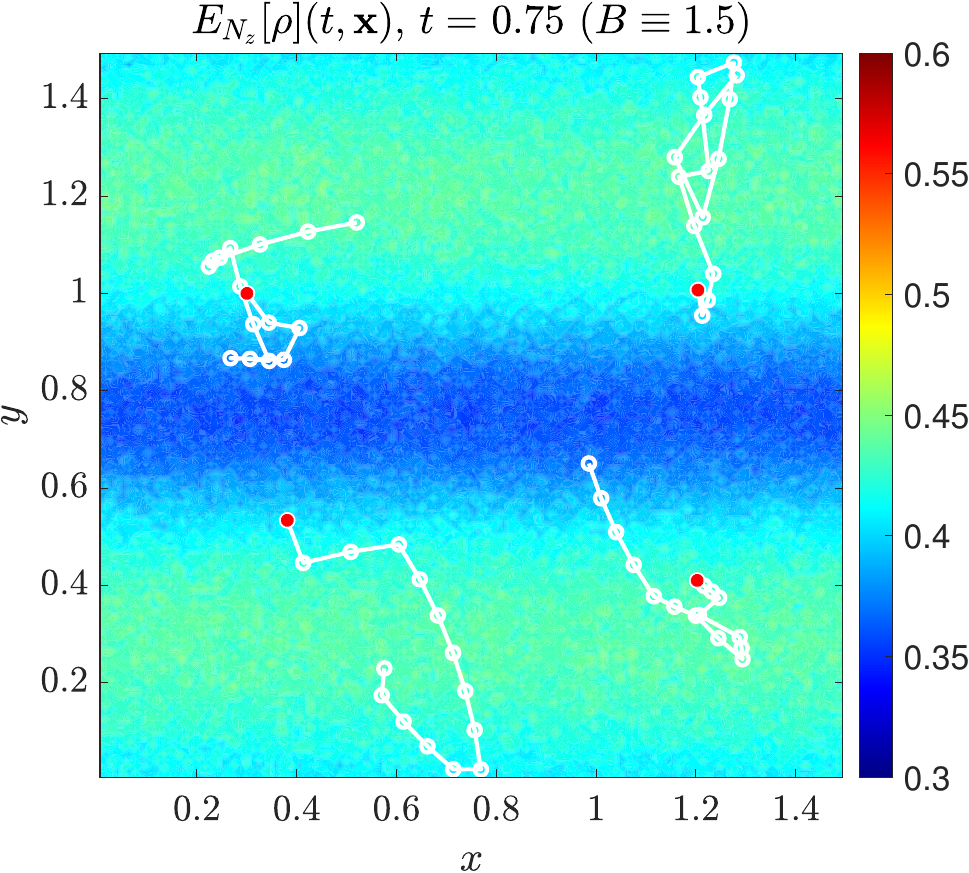}
	\includegraphics[width=0.31\linewidth]{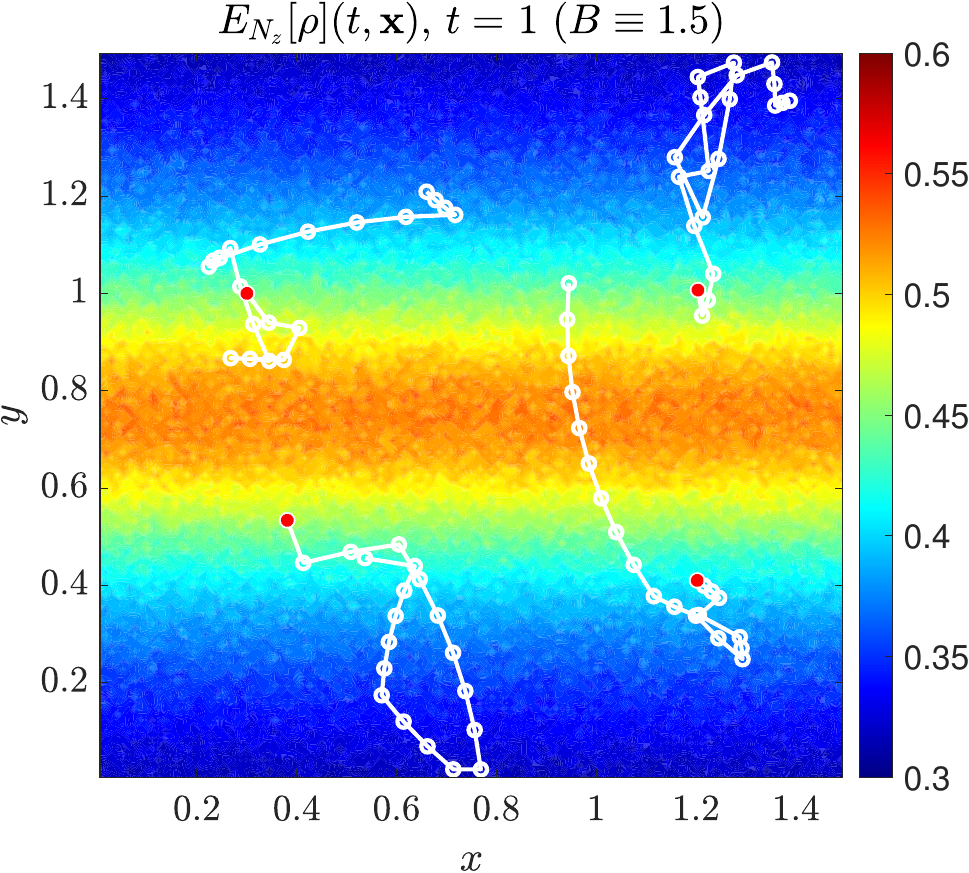}\\
	\includegraphics[width=0.3\linewidth]{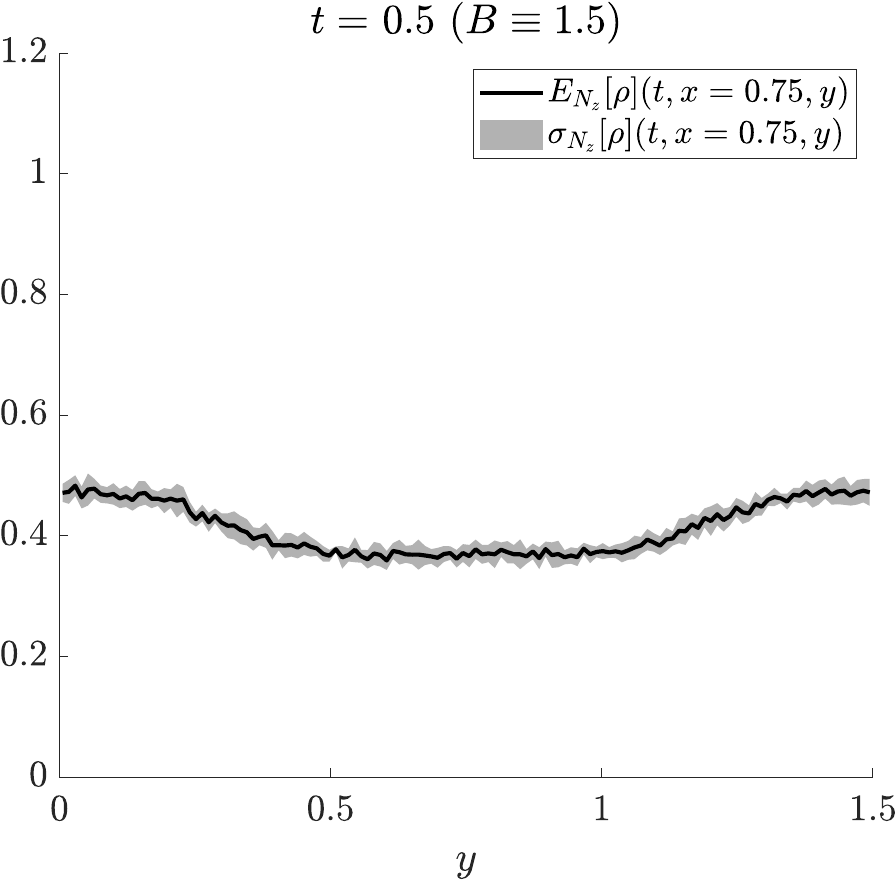}
	\includegraphics[width=0.3\linewidth]{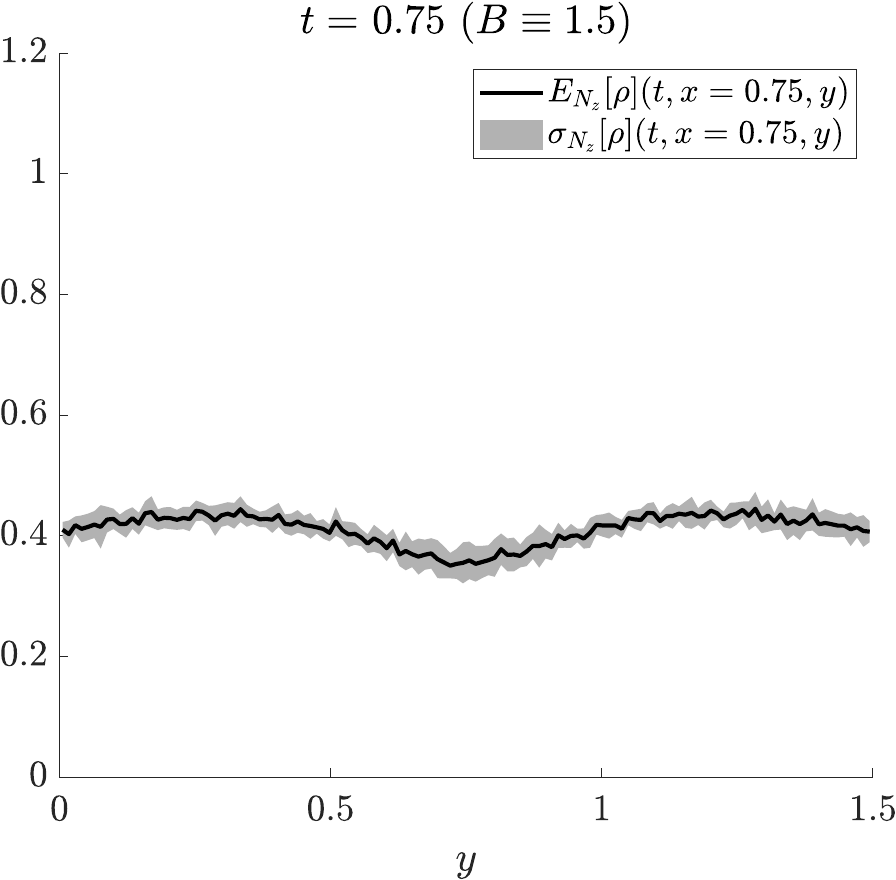}
	\includegraphics[width=0.3\linewidth]{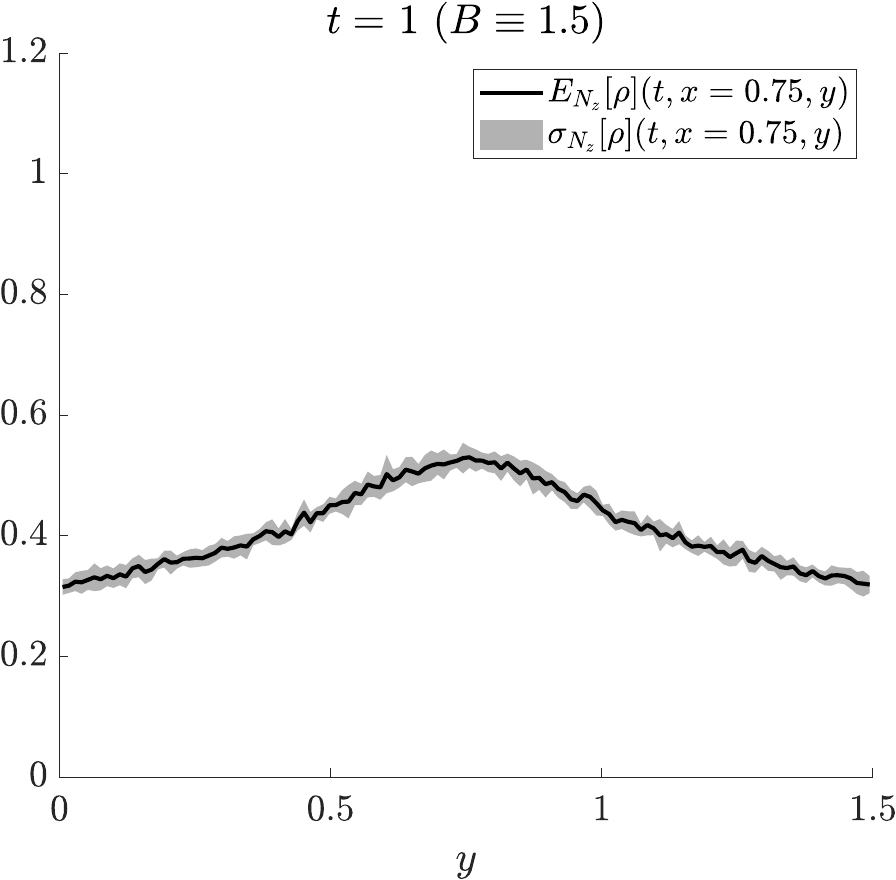}
	\caption{Two-dimensional Sod shock tube test with $ B(t,\xx) = 1.5 $ and $ \nu = 10 $.  Top row: snapshots of the mean density at different time instants. Bottom row: slices of the mean density at $ x = 0.75 $, with the corresponding standard deviation shown as a shaded area.}
	\label{fig:sod2D_B15_nu10}
\end{figure}
\begin{figure}[h!]
	\centering
	\includegraphics[width=0.31\linewidth]{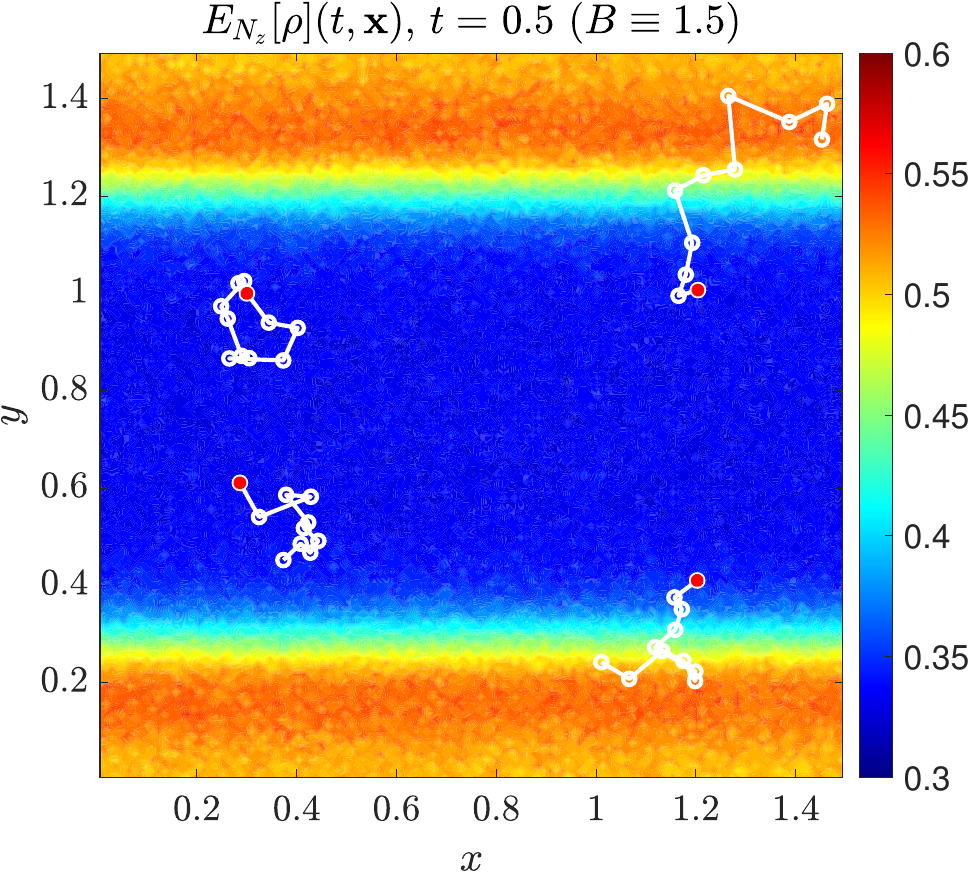}
	\includegraphics[width=0.31\linewidth]{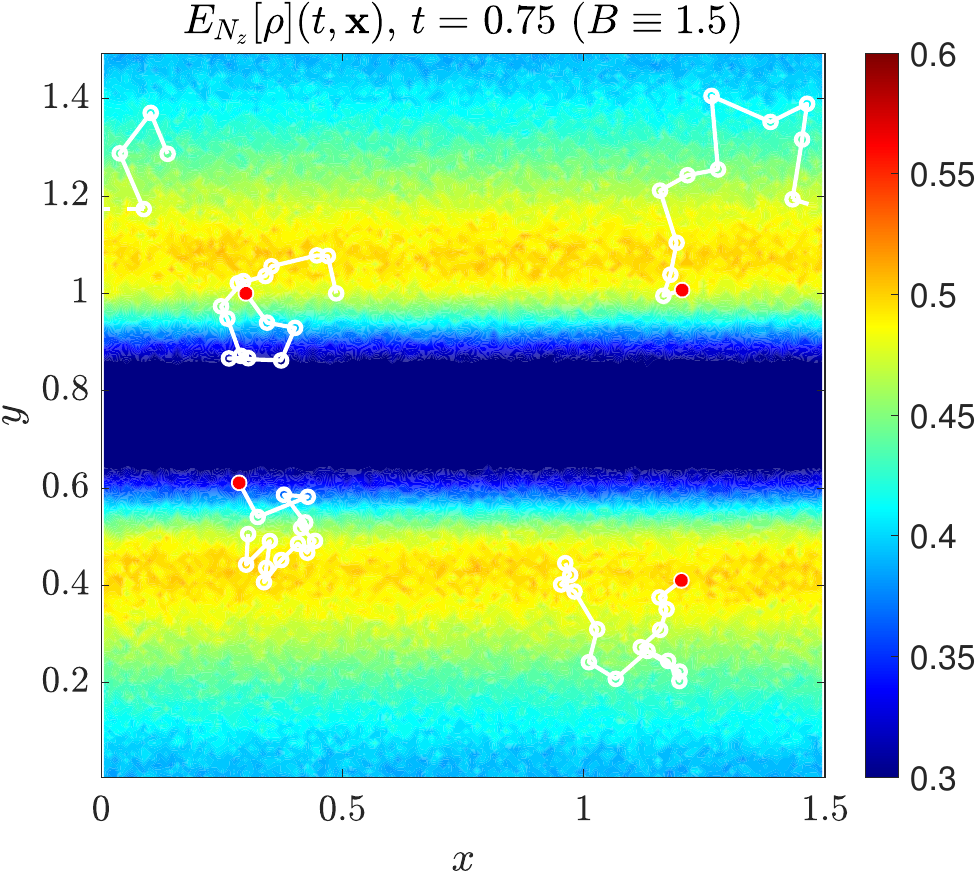}
	\includegraphics[width=0.31\linewidth]{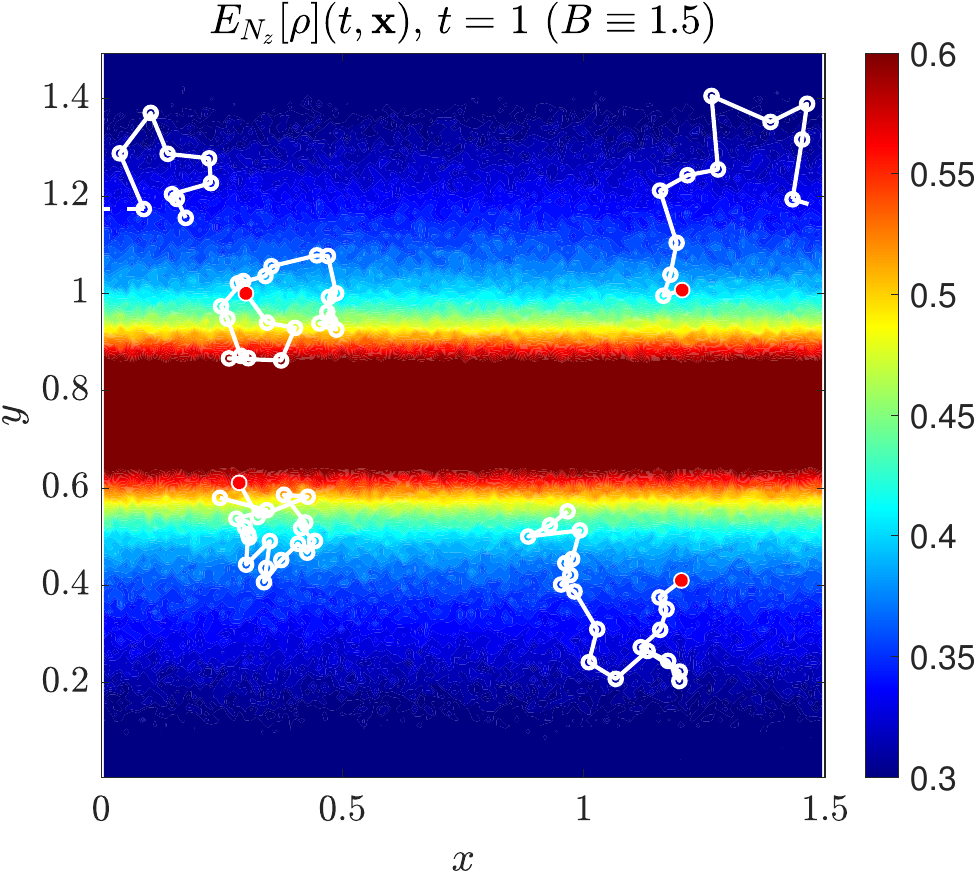}\\
	\includegraphics[width=0.3\linewidth]{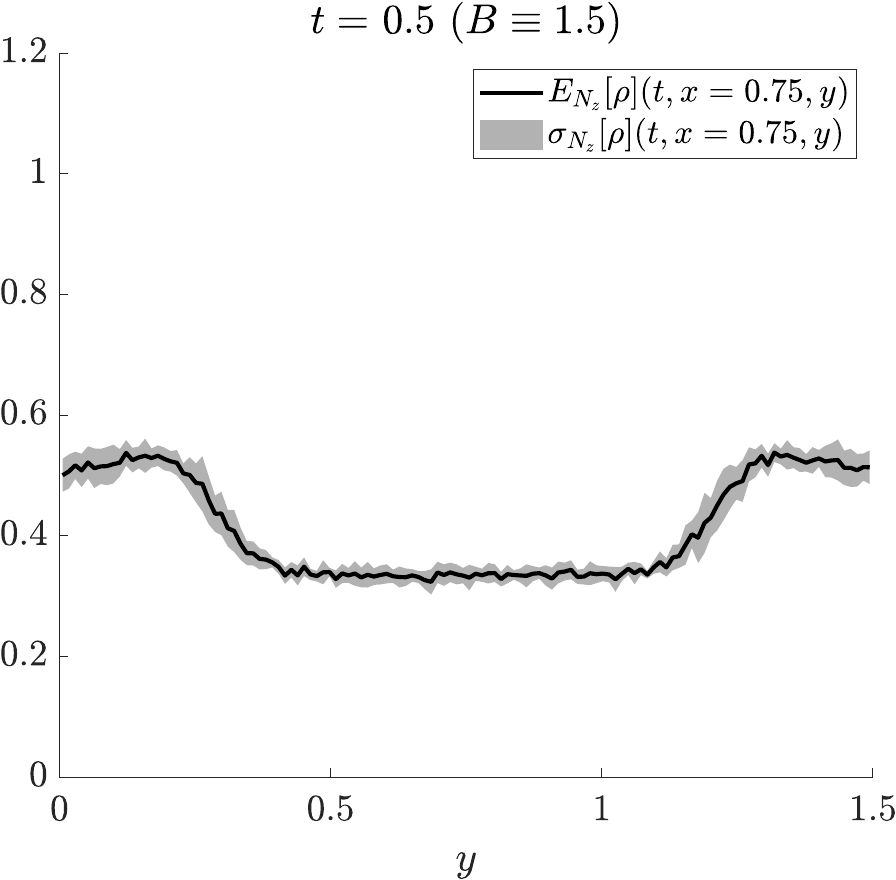}
	\includegraphics[width=0.3\linewidth]{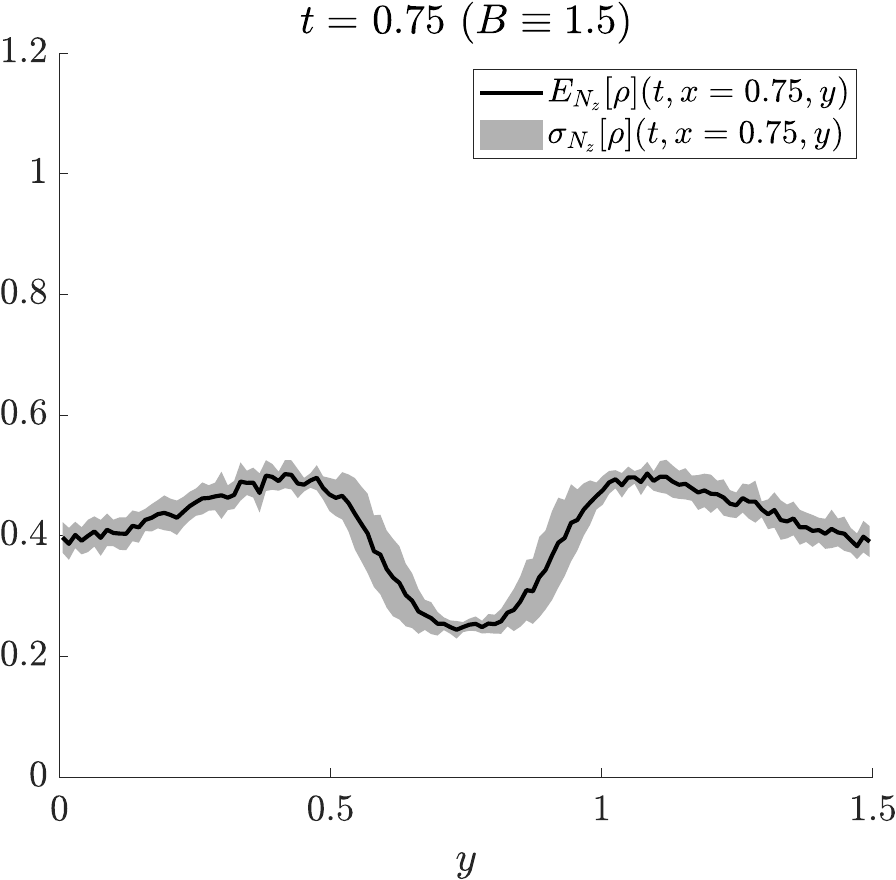}
	\includegraphics[width=0.3\linewidth]{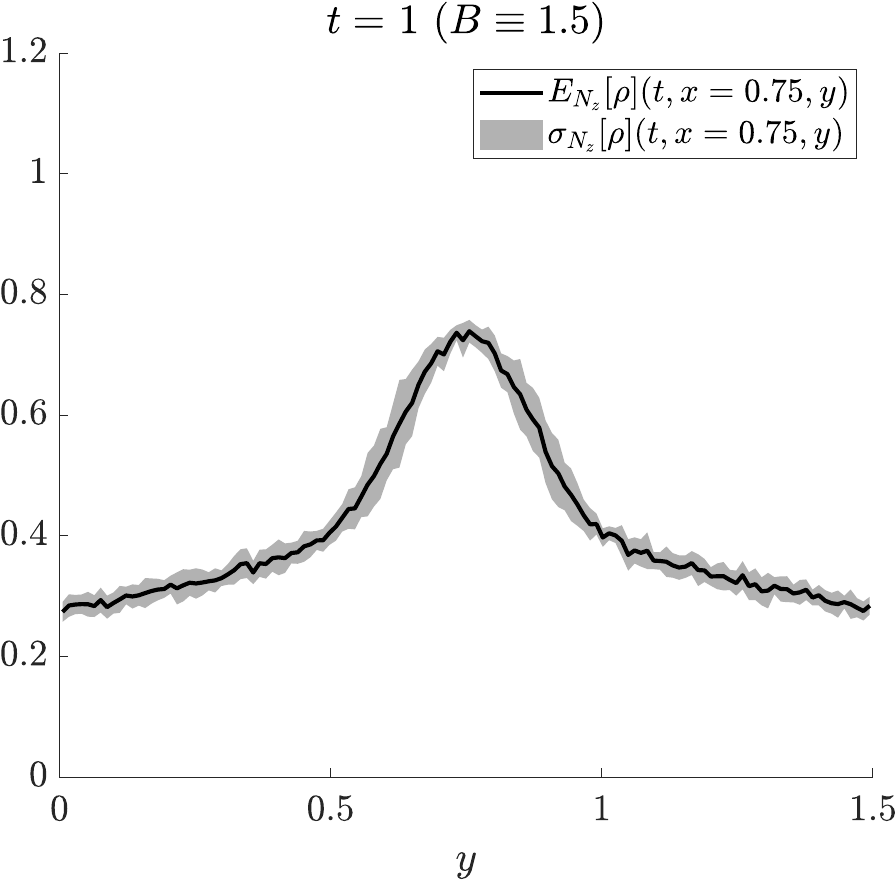}
	\caption{Two-dimensional Sod shock tube test with $ B(t,\xx) = 1.5 $ and $ \nu =1000 $.  Top row: snapshots of the mean density at different time instants. Bottom row: slices of the mean density at $ x = 0.75 $, with the corresponding standard deviation shown as a shaded area.}
	\label{fig:sod2D_B15_nu1000}
\end{figure}
where
\begin{equation}\label{eq:mass_boundary}
	\begin{split}
		U_b^n(z_i) &= \frac{1}{N N_b \rho_b^n(z_i)} \sum_{\mathcal{C}_j \in \Omega_b} \sum_{m=1}^N \vv_m^n(z_i) \chi(\xx_m^n \in \mathcal{C}_j), \\
		\rho_b^n(z_i) &= \frac{1}{N N_b} \sum_{\mathcal{C}_j \in \Omega_b} \sum_{m=1}^N \chi(\xx_m^n \in \mathcal{C}_j),
	\end{split}
\end{equation}
and $ N_b $ denotes the number of cells $ \mathcal{C}_j \in \Omega_b $, with $ \Omega_b = [0, \Delta_y] \cup [1.5-\Delta_y, 1.5] $. This choice corresponds to a region near the $ y $-boundaries with width equal to one cell size $ \Delta_y = 0.234 $.In the uncontrolled case, enlarging the boundary region $\Omega_b$ is expected to increase the thermal energy at the boundaries, since particles entering this region contribute to local velocity fluctuations, thereby enhancing the effective kinetic energy and temperature. In contrast, under control we expect particles to remain confined near the center of the domain, which keeps the thermal energy at the boundaries low.

We begin by considering the uncontrolled case, setting $ B(t, \xx) = 1.5 $.  
Figures~\ref{fig:sod2D_B15_nu0}--\ref{fig:sod2D_B15_nu10}--\ref{fig:sod2D_B15_nu1000} illustrate the system dynamics under three different collisional regimes: $ \nu = 0 $, $ \nu = 10 $, and $ \nu = 1000 $.  
In each figure, the first row displays snapshots of the mean density at times $ t = 0.5 $, $ t = 0.75 $, and $ t = 1 $. Superimposed in white are the mean trajectories of four randomly selected particles up to time $ t $, with their initial mean positions highlighted in red.
The second row shows slices of the mean density function at $ x = 0.75 $, taken at the same time instants, with the associated standard deviation represented as a shaded area.
Initially, particles move toward the upper and lower boundaries of the domain, where they are reflected due to the imposed boundary conditions. As the collisional frequency $ \nu $ increases, the diffusion of particles across the domain is progressively reduced. In the collisionless case ($ \nu = 0 $), particles exhibit strong diffusive behavior. In contrast, in the highly collisional regime ($ \nu = 1000 $), particles do not diffuse but instead oscillate around the center of the domain. The intermediate case ($ \nu = 10 $) corresponds to a quasi-collisional regime, where both diffusion and collisional effects are simultaneously present.
Figure~\ref{fig:sod2D_noControl_energy} shows the mean thermal energy at the boundaries along with its standard deviation for $ \nu = 0 $ (left), $ \nu = 10 $ (center), and $ \nu = 1000 $ (right).  
The thermal energy increases whenever particles collide with the boundaries.  
In the collisionless regime, the spatial diffusion of particles results in only minor variations in thermal energy.  
In contrast, in the quasi-collisional and fully collisional regimes, the thermal energy exhibits an oscillatory pattern, reflecting the behavior of the particle trajectories.
\begin{figure}[h!]
	\centering
	\includegraphics[width=0.3\linewidth]{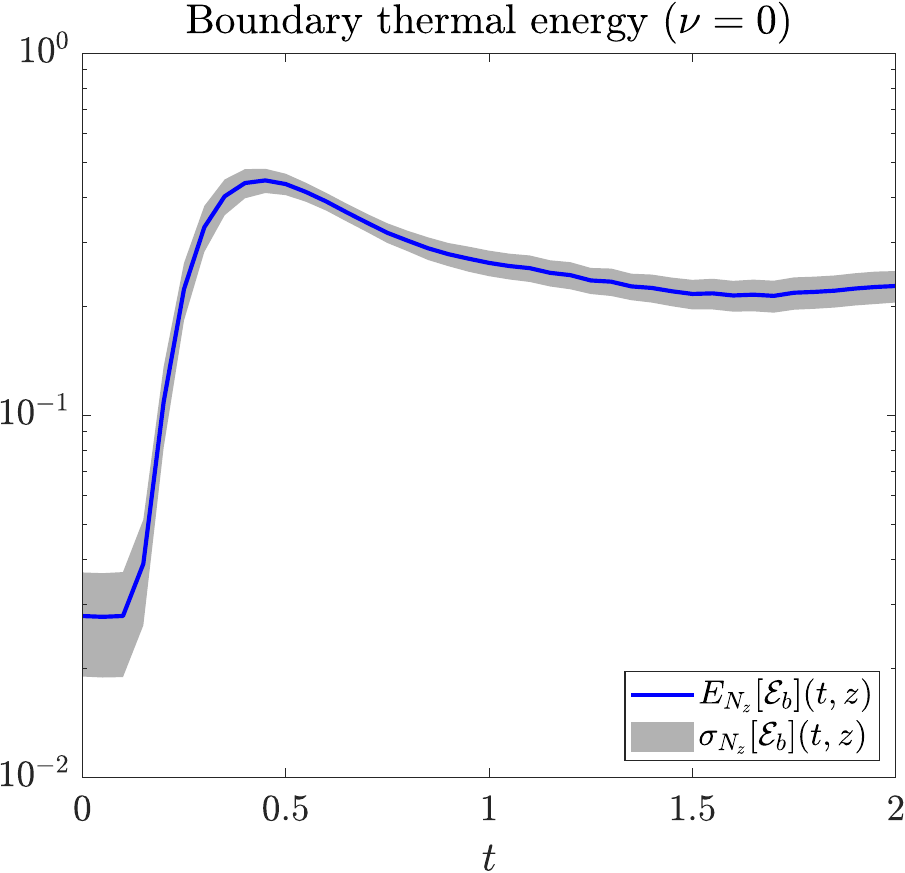}
	\includegraphics[width=0.3\linewidth]{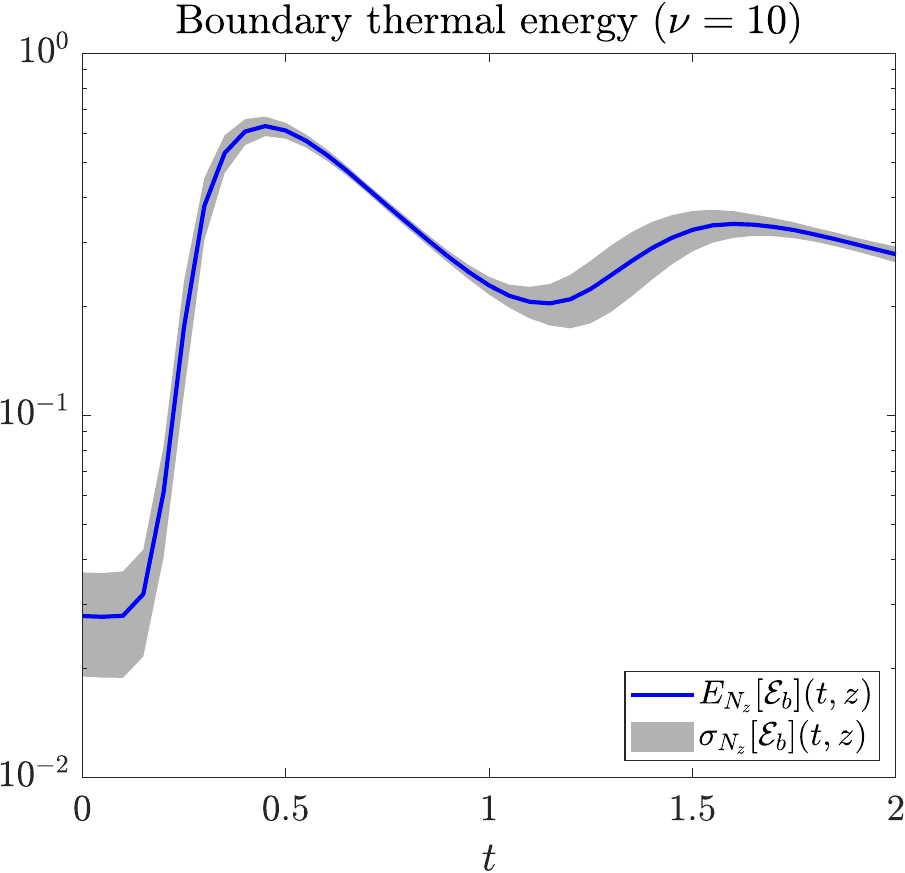}
	\includegraphics[width=0.3\linewidth]{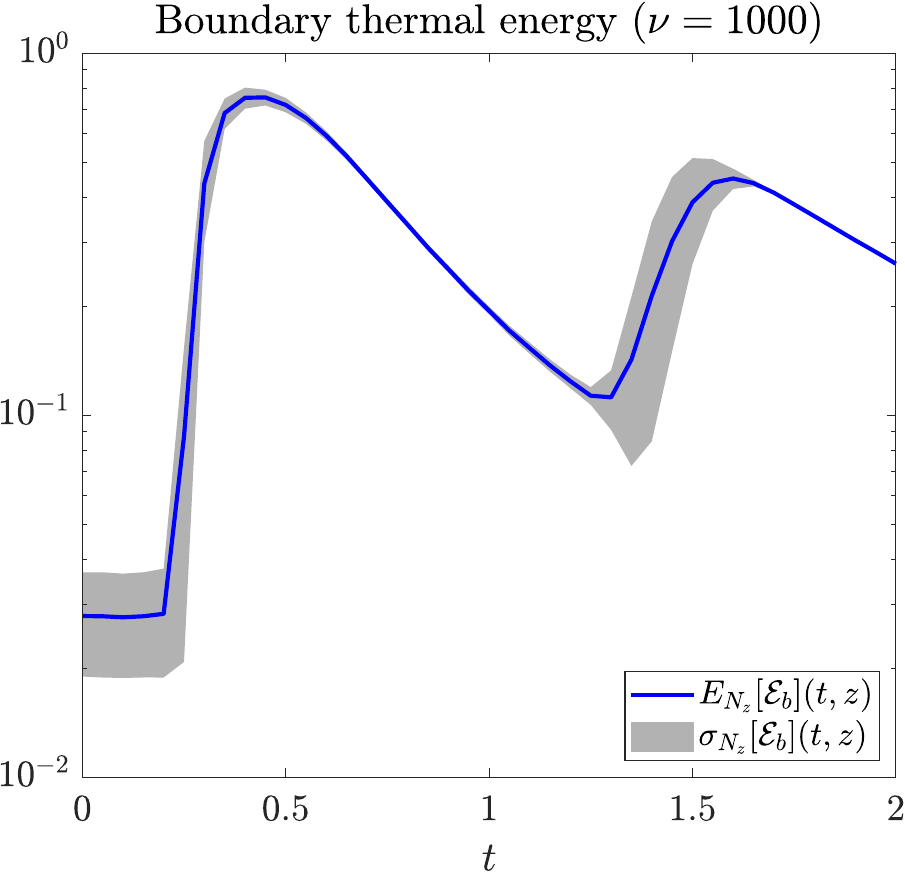}
	\caption{Two-dimensional sod shock tube test without control. Thermal energy at the boundaries for $\nu = 0$ (on the left), $\nu = 10$ (in the centre) and $\nu = 1000$ (on the right). The mean value is depicted in blue, while the standard deviation as a shaded area.}
	\label{fig:sod2D_noControl_energy}
\end{figure}
We now consider the case in which the magnetic field is computed as the solution of a control problem aimed at minimizing the percentage of mass reaching the lower and upper boundaries of the domain.  
To this end, we assume that the spatial region where the control is active is divided into $ N_c $ horizontal cells, and we set $ N_c = 4 $, unless otherwise specified.  We then perform a set of preliminary experiments to assess the sensitivity to the choice of the control parameters $\alpha_\textrm{x}$, $\beta_\textrm{x}$, $\alpha_\textrm{v}$, $\beta_\textrm{v}$, and $\gamma$. For simplicity, we assume to be in the fully collisional regime ($\nu = 1000$). Similar results can be obtained in the non-collisional regime and in the quasi-collisional one.  
Figure \ref{fig:energy_alfax_alfav_gamma} shows the mean boundary thermal energy computed at time $t=2$ as in \eqref{eq:energy}. On the left, we test the correlation between $\alpha_\textrm{x} = \beta_\textrm{x} = 2,\ldots, 10$ and $\alpha_\textrm{v} = \beta_\textrm{v} = 12,\ldots,20$, setting $\gamma = 2.5\times 10^{-3}$ and $M=50$. On the right, we fix   $ \alpha_\text{x} = 5 $, $ \beta_\text{x} = 2 $, $ \alpha_\text{v} = 15 $, $ \beta_\text{v} = 12 $, $M=50$, and we let $\gamma$ to vary between $10^{-4}$ and $10^{-1}$. Across all considered scenarios, the mean boundary thermal energy at the final time remains small and scales inversely with $\alpha_\textrm{v}$, while showing no apparent dependence on $\alpha_\textrm{x}$. Moreover, it decreases more rapidly and attains lower values as $\gamma$ decreases. 
\begin{figure}
	[h!]
	\centering
	\includegraphics[width=0.35\linewidth]{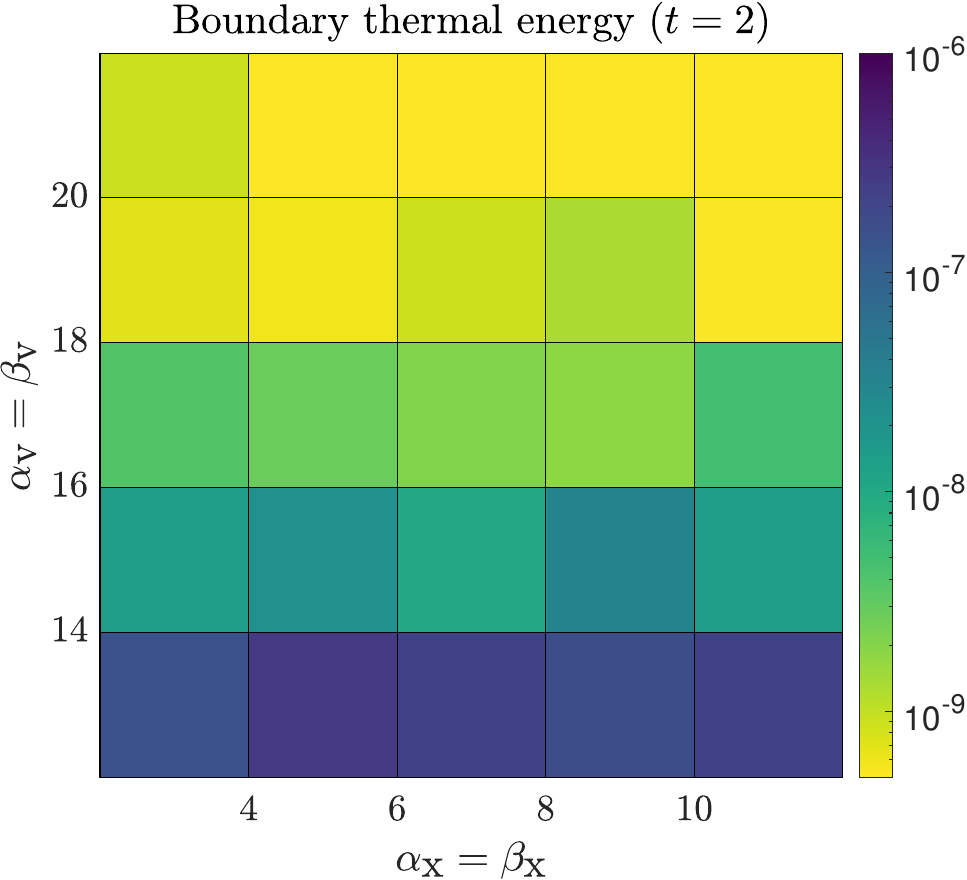}
	\includegraphics[width=0.335\linewidth]{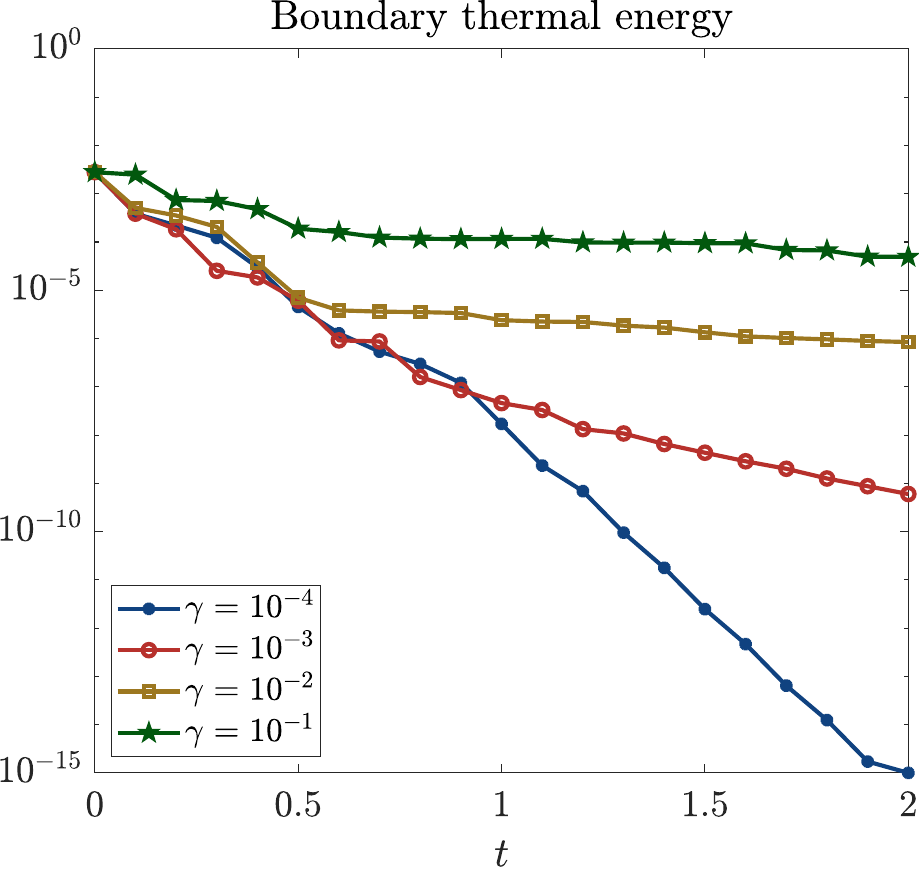}
	\caption{ Two-dimensional Sod shock tube: preliminary test. On the left, mean boundary thermal energy at time $t=2$ as $\alpha_\textrm{x}=\beta_\textrm{x}$ and $\alpha_\textrm{v} = \beta_\textrm{v}$ vary for $\gamma = 2.5\times 10^{-3}$, and $M=50$. On the right, mean boundary thermal energy for  $ \alpha_\text{x} = 5 $, $ \beta_\text{x} = 2 $, $ \alpha_\text{v} = 15 $, $ \beta_\text{v} = 12 $, $M=50$, as $\gamma$ varies between $10^{-4}$ and $10^{-1}$. } 
	\label{fig:energy_alfax_alfav_gamma}
\end{figure}  
We then examined the correlation between the pairs $(\alpha_\textrm{x},\beta_\textrm{x})$ and $(\alpha_\textrm{v},\beta_\textrm{v})$. The thermal energy at the boundaries remained essentially constant across all parameter combinations; accordingly, the corresponding plots are omitted for brevity. Other tests will be conducted later in this section to assess the effectiveness of the strategy as the maximum control strength and the initial temperature vary.  
\\\\
In the numerical experiments below, the parameters are specified as follows.
The maximum control magnitude is fixed at $M = 50$, and the control parameters are chosen as follows: $ \alpha_\text{x} = 5 $, $ \beta_\text{x} = 2 $, $ \alpha_\text{v} = 15 $, $ \beta_\text{v} = 12 $, and $ \gamma = 2.5 \times 10^{-3} $.  
The target position is set to $ \hat{y} = 0.75 $, in order to drive the mass toward the center of the domain.
Unless stated otherwise, the operator $ \mathcal{P}(\cdot) $ is defined as in equation~\eqref{eq:R_max}.  
The temperature at the boundaries is computed as
\[
T_b^n(z_i) = \rho_b^n(z_i)\, \mathcal{E}_b^n(z_i),
\]
for any $ i = 1, \ldots, N_z $, where $ \rho_b^n(\cdot) $ and $ \mathcal{E}_b^n(\cdot) $ are defined in equations~\eqref{eq:energy}–\eqref{eq:mass_boundary}.
Figures~\ref{fig:sod2D_control_Bmax_nu0}–\ref{fig:sod2D_control_Bmax_nu10}–\ref{fig:sod2D_control_Bmax_nu1000} show the controlled dynamics for three different collisional regimes: $ \nu = 0 $, $ \nu = 10 $, and $ \nu = 1000 $.  
In each figure, the first row displays snapshots of the mean density at times $ t = 0.5 $, $ t = 0.75 $, and $ t = 1 $.  
The mean trajectories of four randomly selected particles are shown in white, with their initial positions highlighted in red.
The second row shows slices of the mean density function at $ x = 0.75 $, corresponding to the same time instants, with the associated standard deviation represented as a shaded area. Additionally, the values of the magnetic field $ B(t,\xx) $ in each cell $ C_k $ are displayed using a colorbar.
\begin{figure}[h!]
	\centering
	\includegraphics[width=0.31\linewidth]{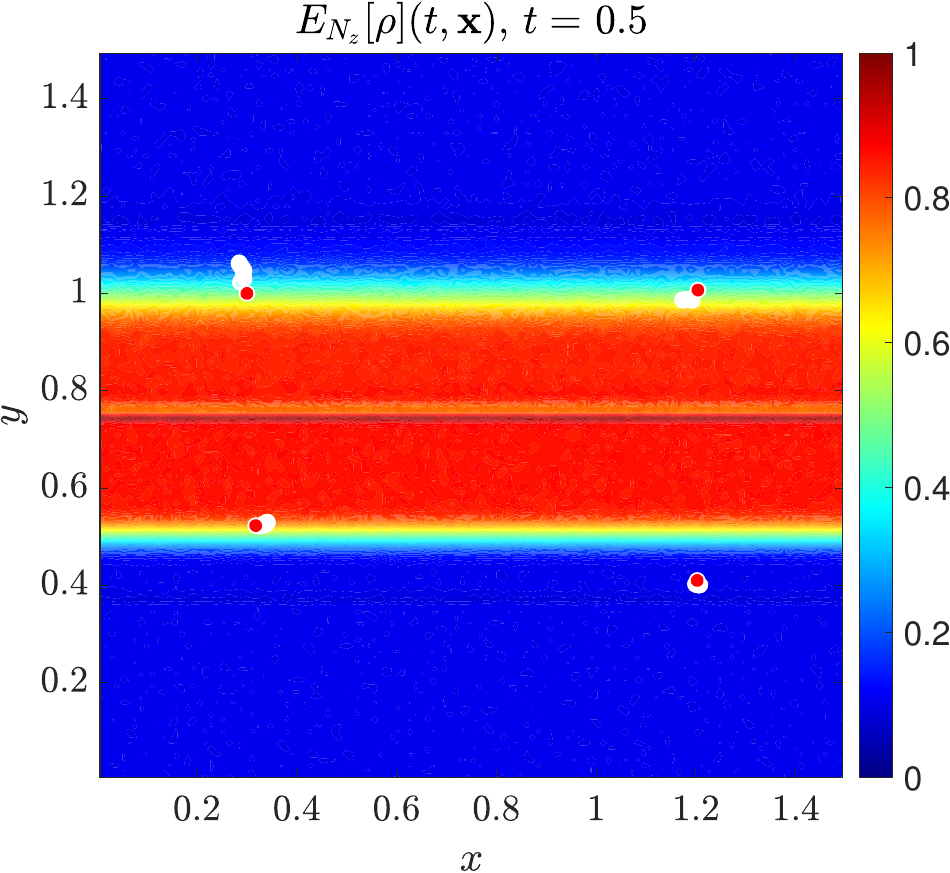}
	\includegraphics[width=0.31\linewidth]{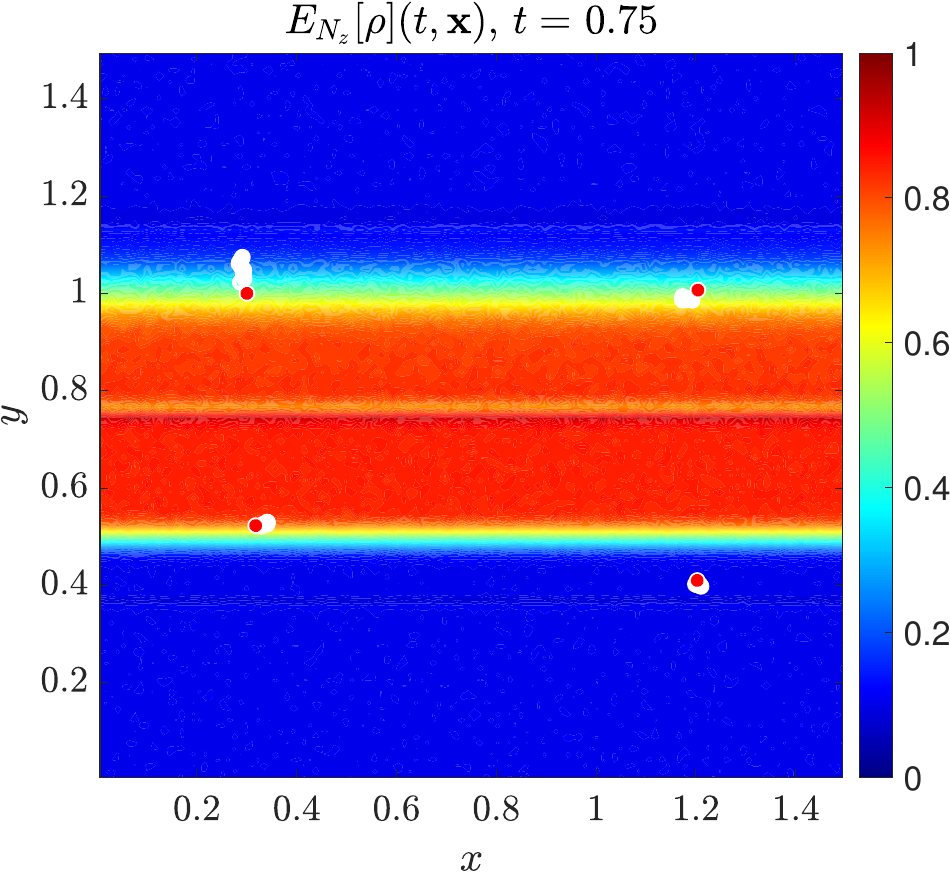}
	\includegraphics[width=0.31\linewidth]{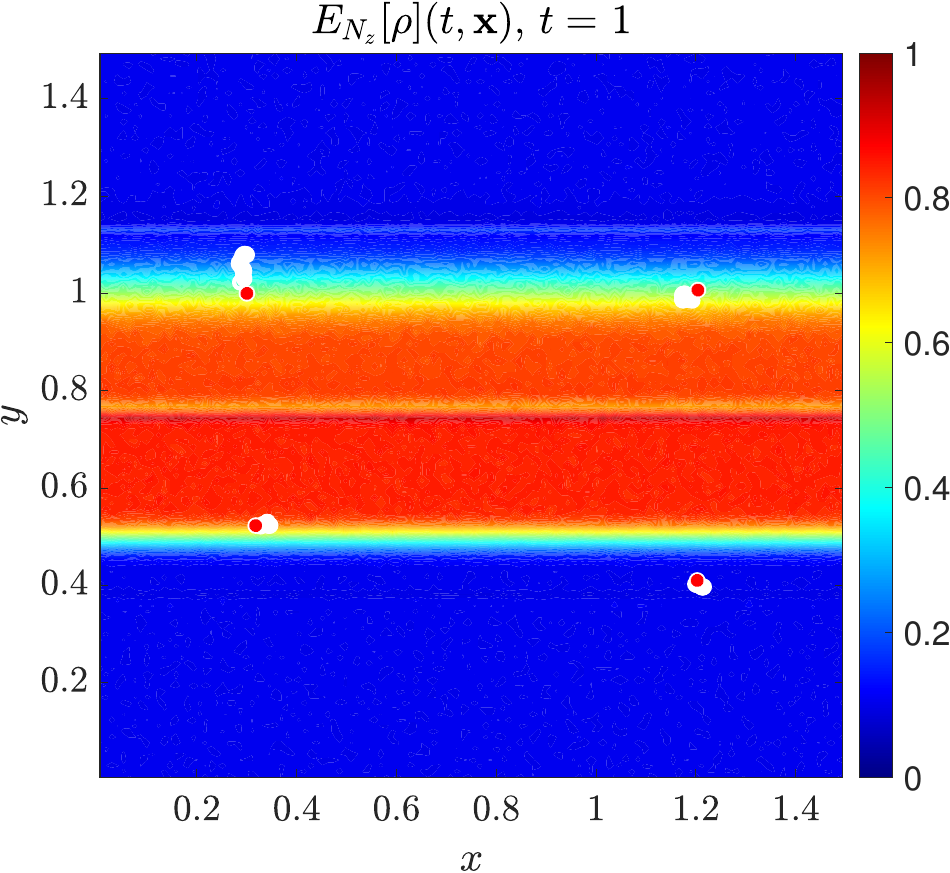}\\
	\includegraphics[width=0.3\linewidth]{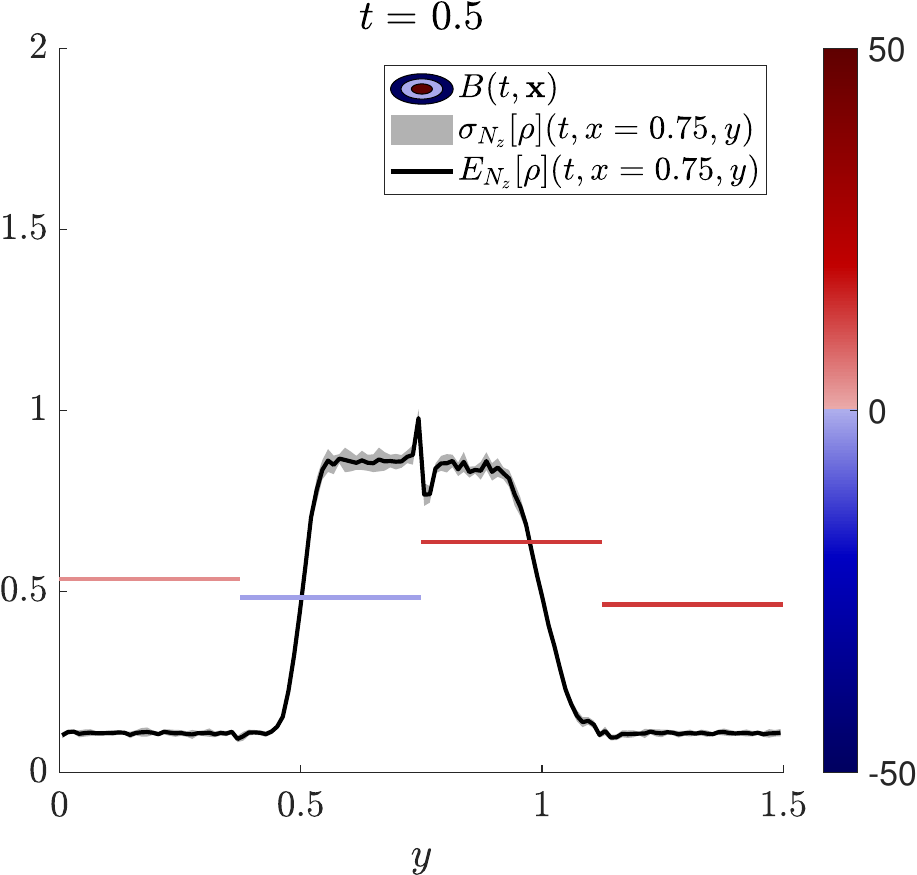}
	\includegraphics[width=0.3\linewidth]{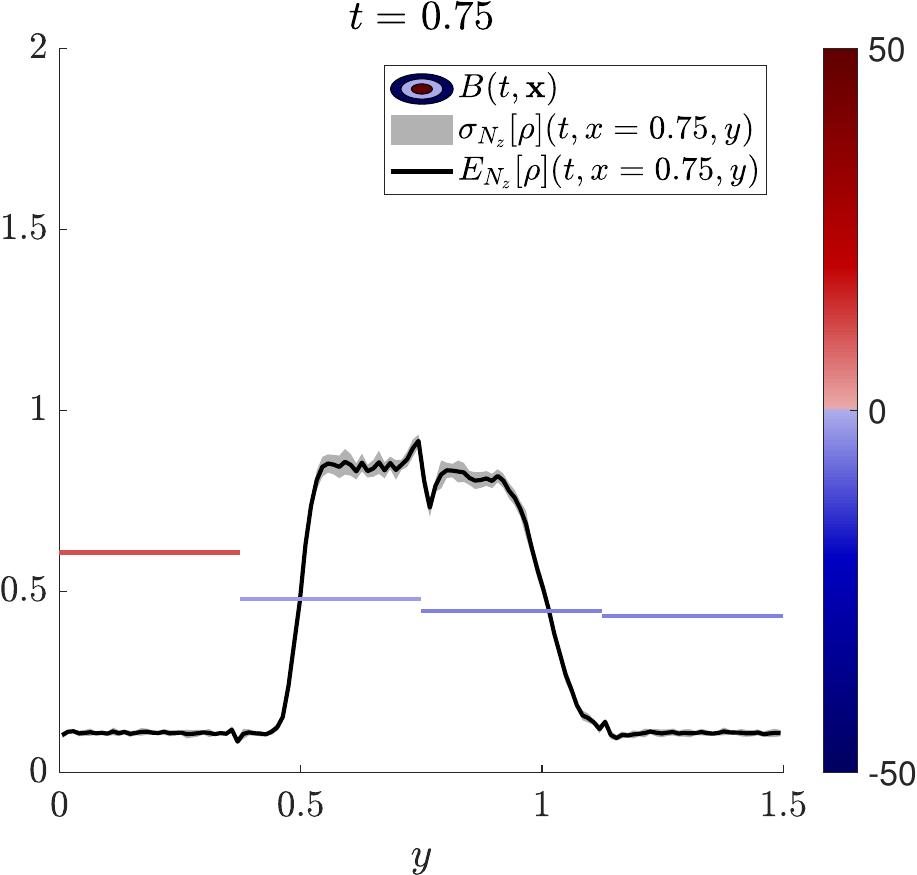}
	\includegraphics[width=0.3\linewidth]{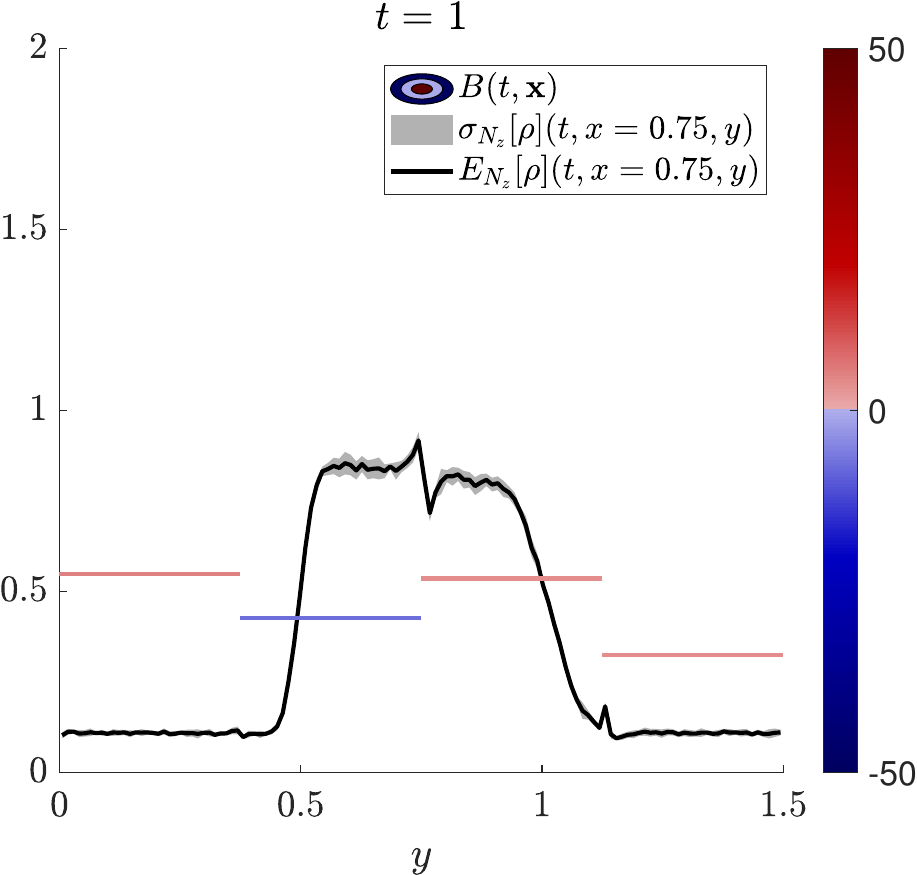}
	\caption{Two-dimensional Sod shock tube test with control and $ \nu =0 $.  Top row: snapshots of the mean density at different time instants. Bottom row: slices of the mean density at $ x = 0.75 $, with the corresponding standard deviation shown as a shaded area. The intensity of $B(t,\xx)$ in each cell $C_k$ is represented by the colorbar.}
	\label{fig:sod2D_control_Bmax_nu0}
\end{figure}
\begin{figure}[h!]
	\centering
	\includegraphics[width=0.31\linewidth]{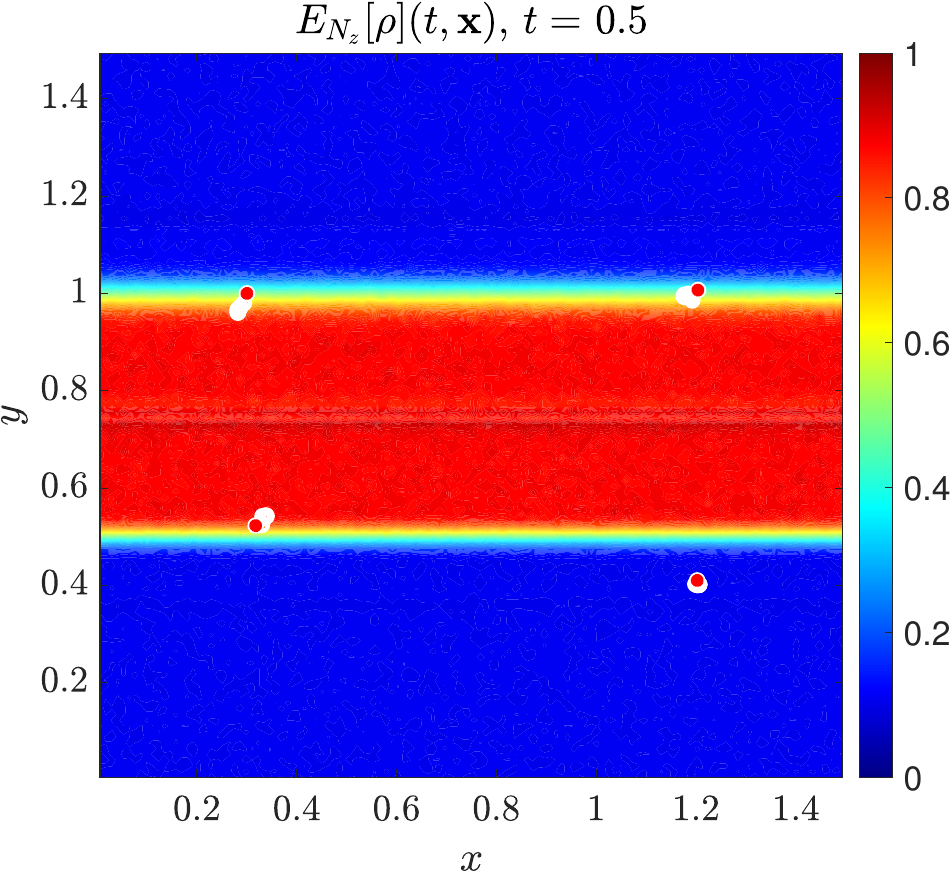}
	\includegraphics[width=0.31\linewidth]{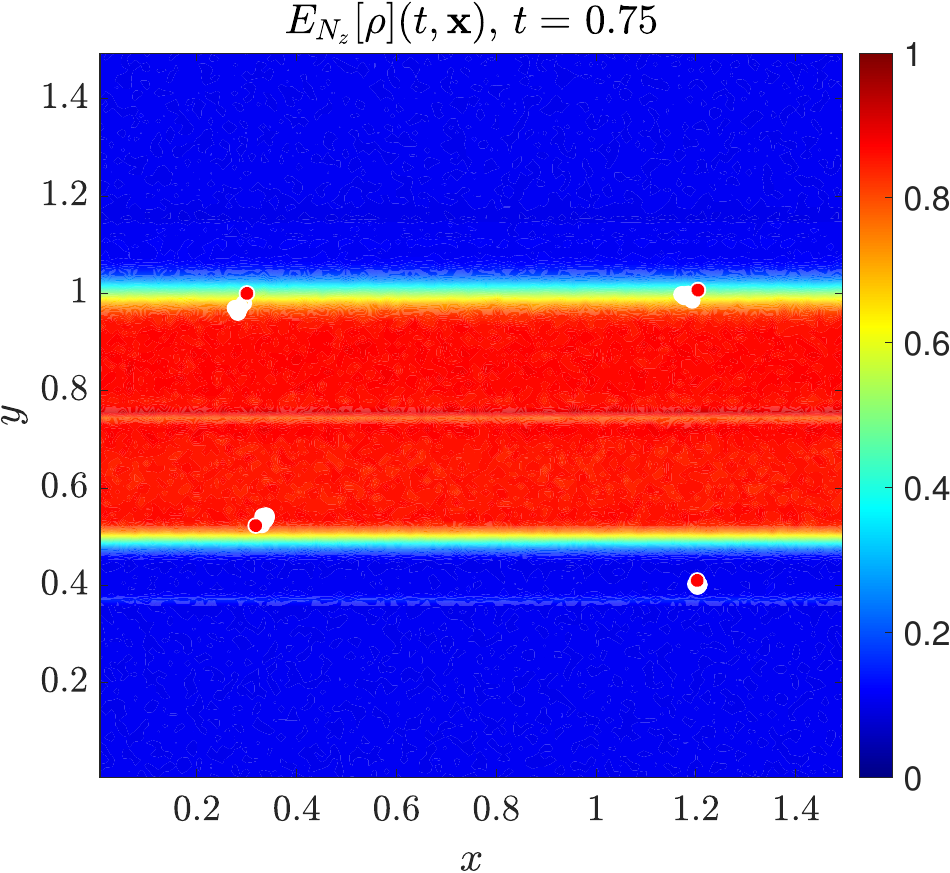}
	\includegraphics[width=0.31\linewidth]{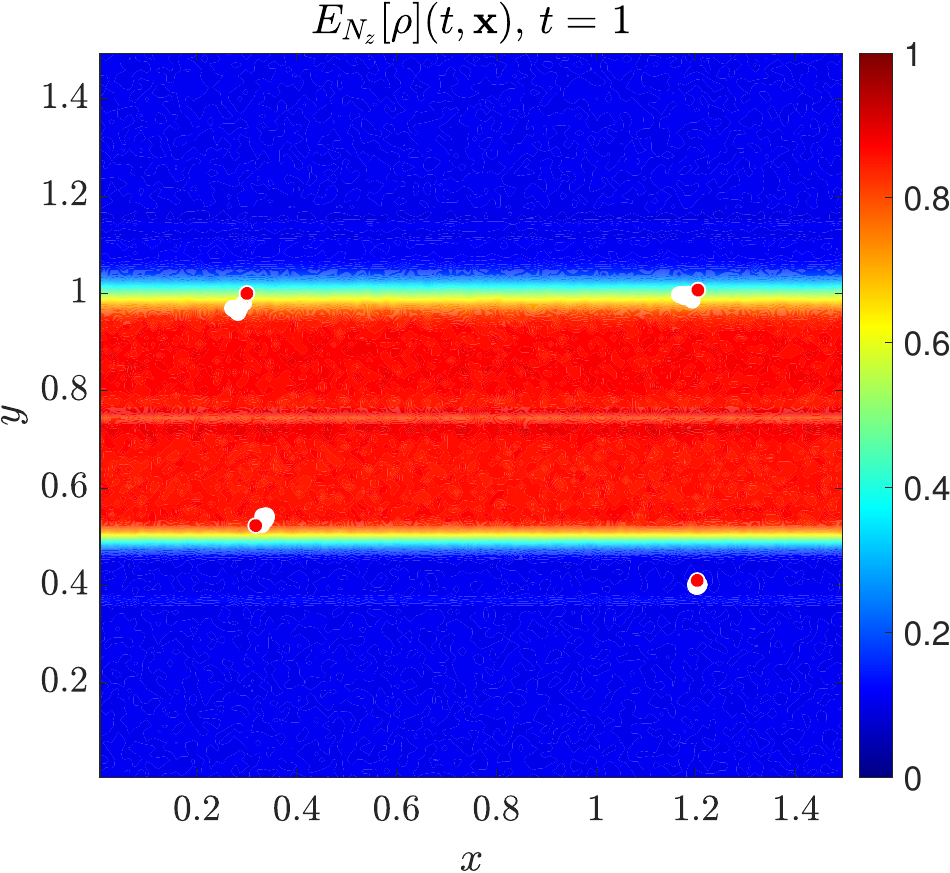}\\
	\includegraphics[width=0.3\linewidth]{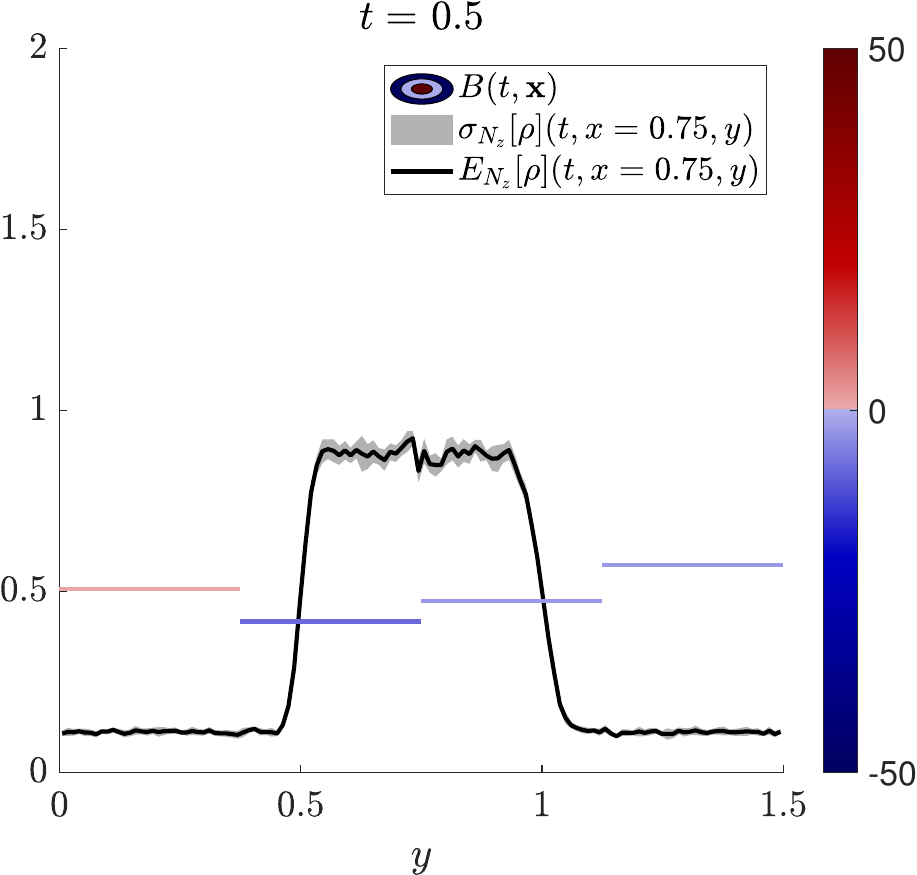}
	\includegraphics[width=0.3\linewidth]{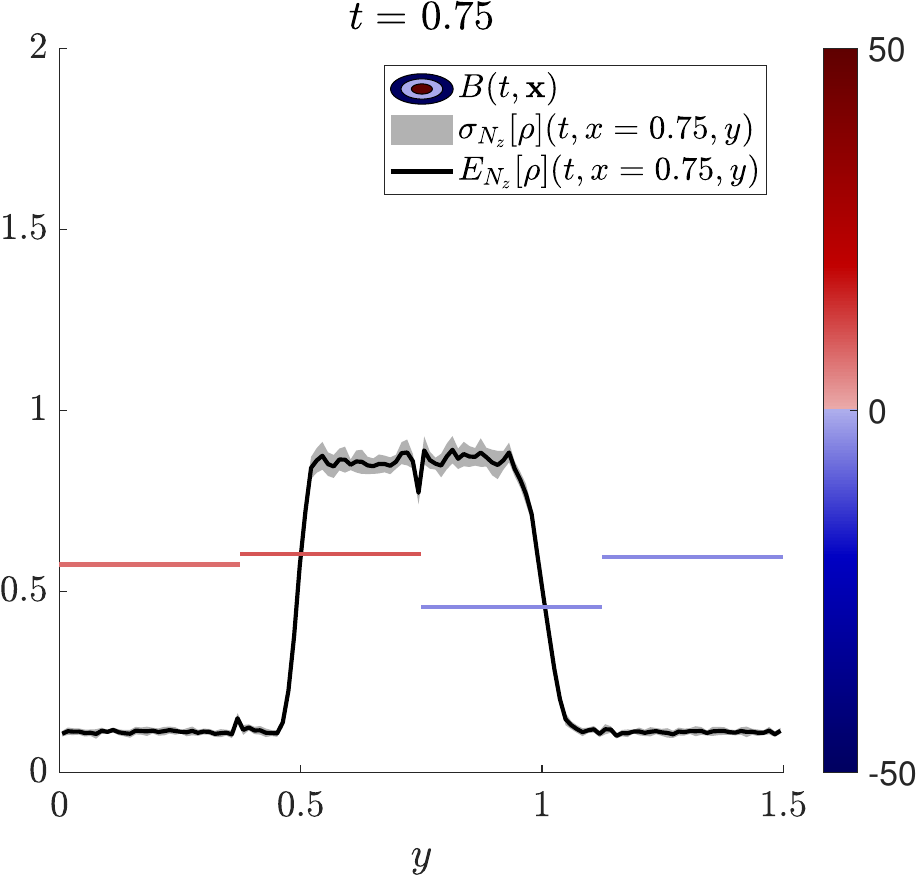}
	\includegraphics[width=0.3\linewidth]{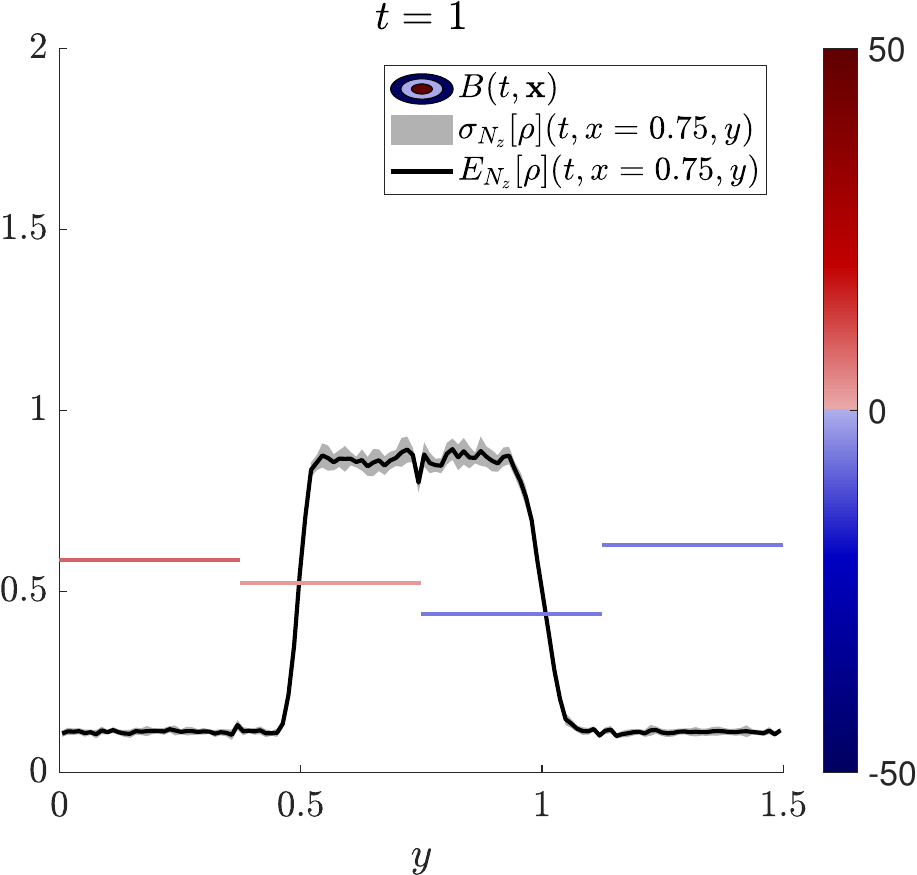}
	\caption{Two-dimensional Sod shock tube test with control and $ \nu =10 $.  Top row: snapshots of the mean density at different time instants. Bottom row: slices of the mean density at $ x = 0.75 $, with the corresponding standard deviation shown as a shaded area. The intensity of $B(t,\xx)$ in each cell $C_k$ is represented by the colorbar. }
	\label{fig:sod2D_control_Bmax_nu10}
\end{figure}
\begin{figure}[h!]
	\centering
	\includegraphics[width=0.31\linewidth]{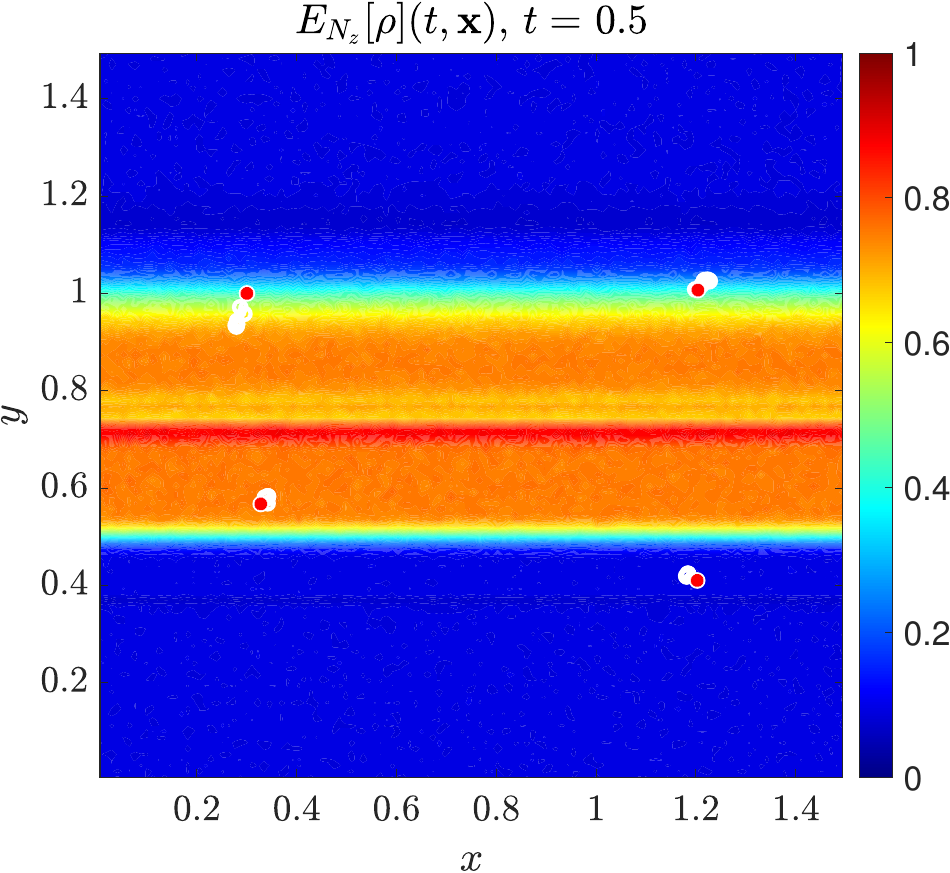}
	\includegraphics[width=0.31\linewidth]{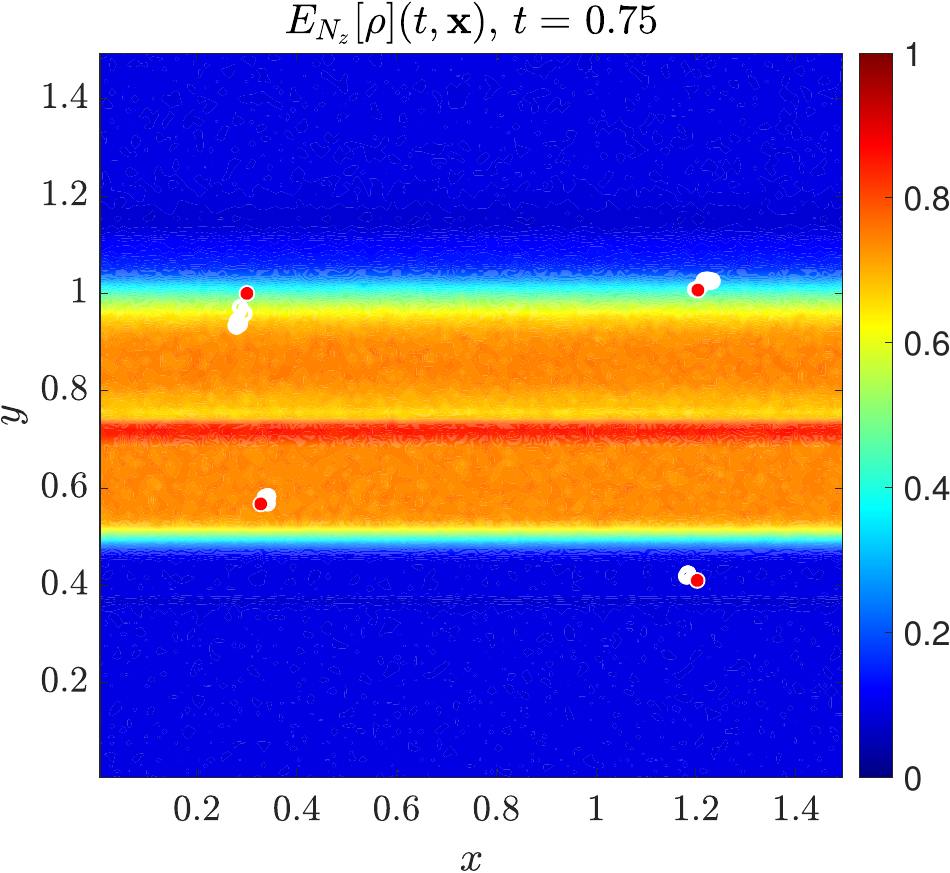}
	\includegraphics[width=0.31\linewidth]{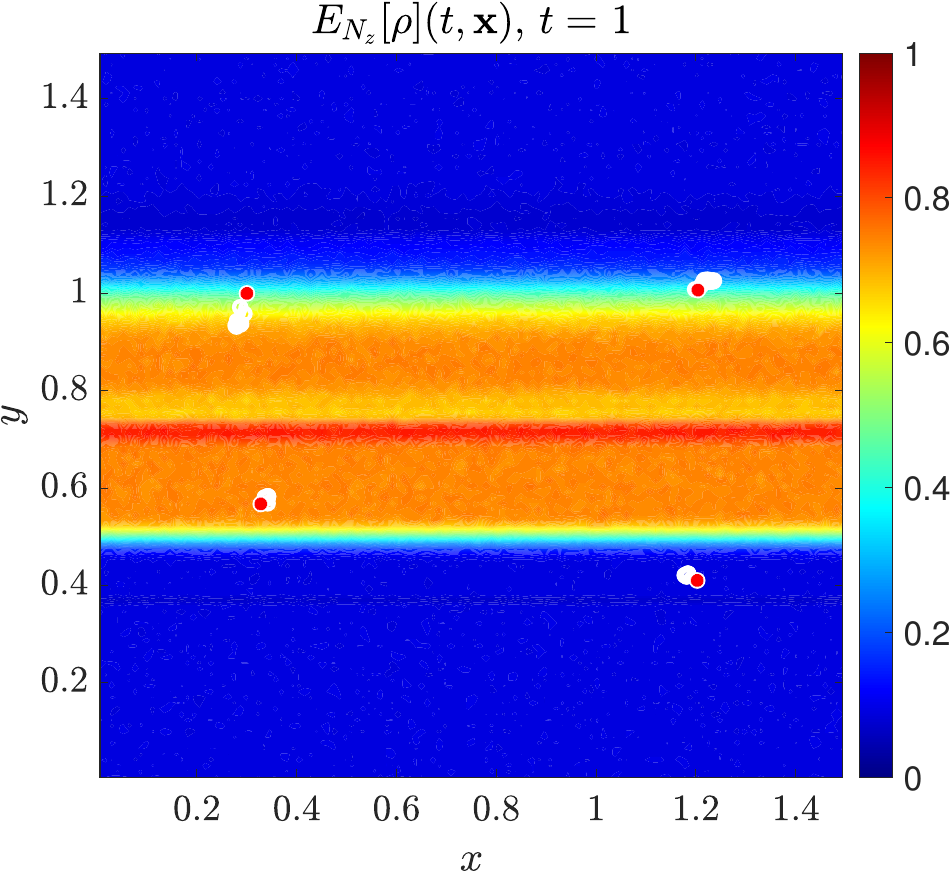}\\
	\includegraphics[width=0.3\linewidth]{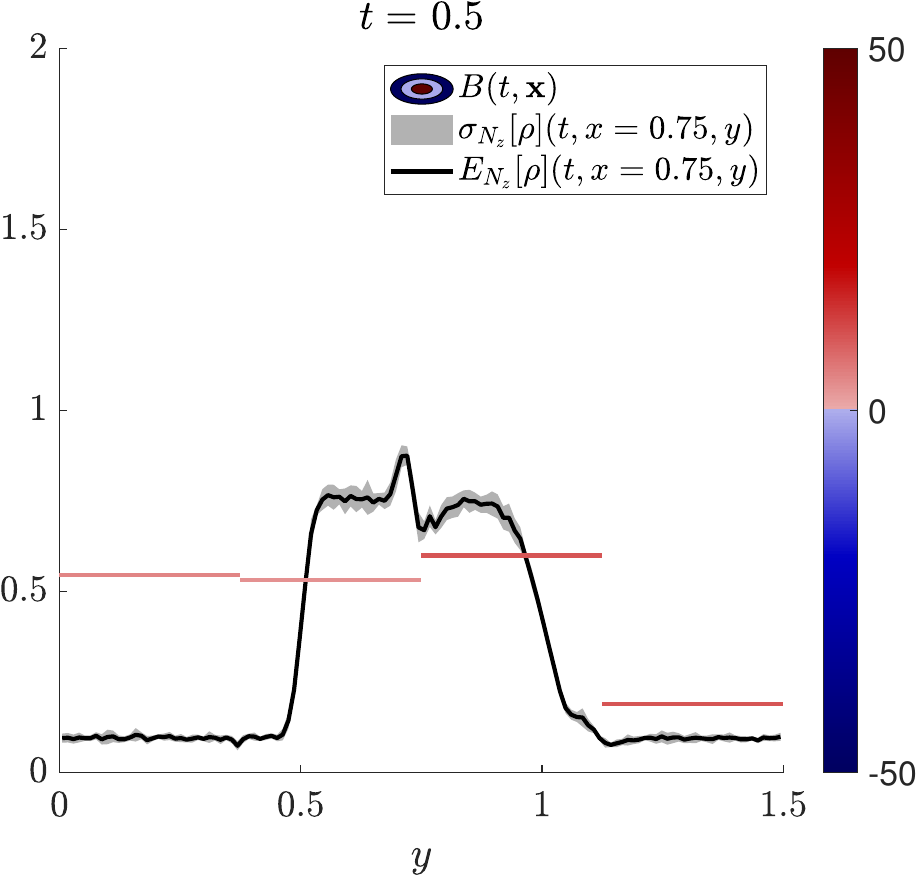}
	\includegraphics[width=0.3\linewidth]{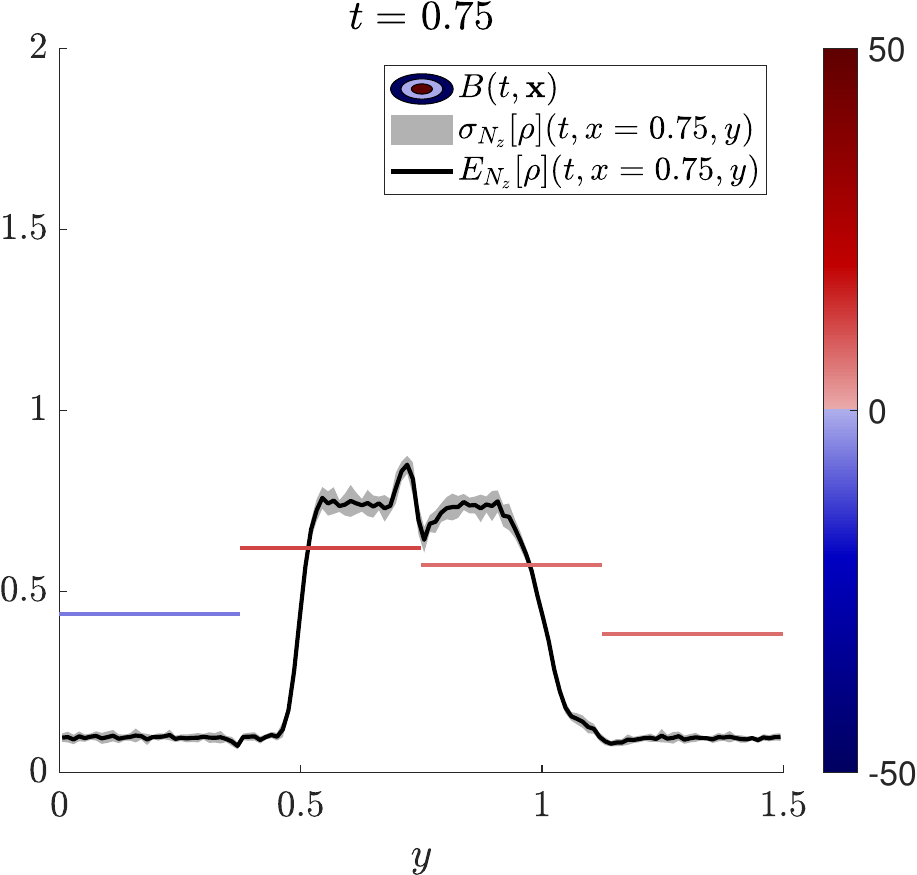}
	\includegraphics[width=0.3\linewidth]{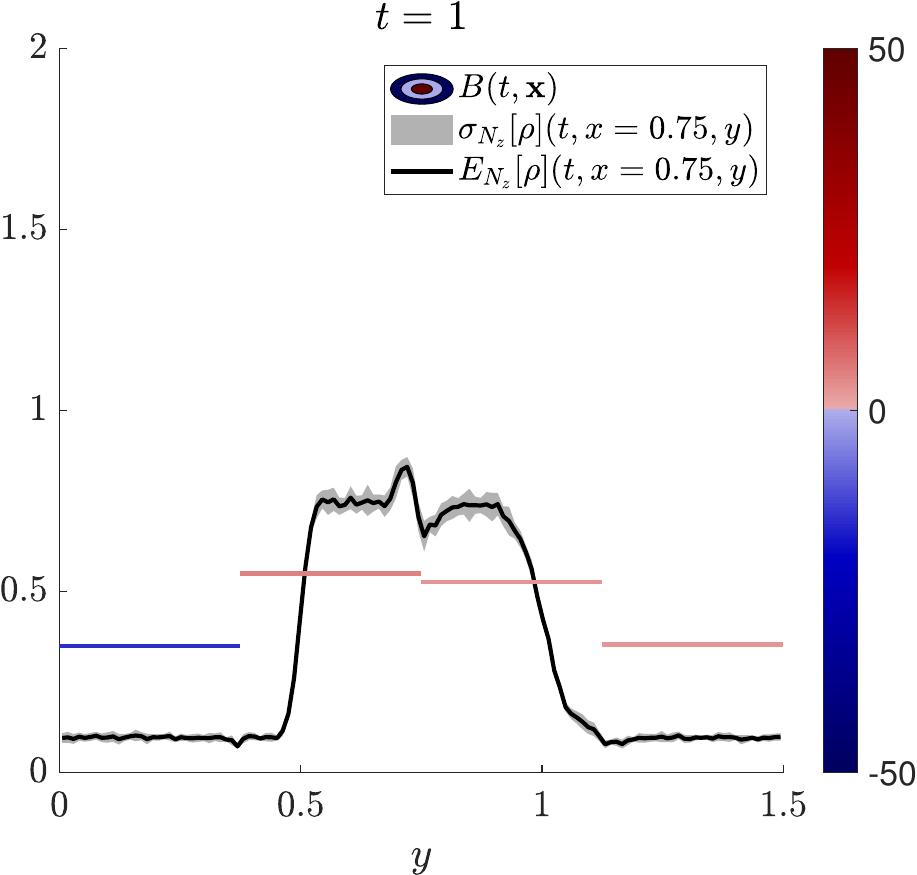}
	\caption{Two-dimensional Sod shock tube test with control and $ \nu =1000 $.  Top row: snapshots of the mean density at different time instants. Bottom row: slices of the mean density at $ x = 0.75 $, with the corresponding standard deviation shown as a shaded area. The intensity of $B(t,\xx)$ in each cell $C_k$ is represented by the colorbar.  }
	\label{fig:sod2D_control_Bmax_nu1000}
\end{figure}
Figure~\ref{fig:sod2D_control_Bmax_energy} shows, in the first row, the mean thermal energy at the boundaries, computed as in equation~\eqref{eq:energy}, for different values of $ \nu $, along with the corresponding standard deviation depicted as a shaded area.  
In all cases considered, the thermal energy decreases over time as a result of the confinement of particles near the center of the domain, highlighting the effectiveness of the adopted control strategy.  
In the second row, the temporal evolution of the magnetic field is shown for the same values of $ \nu $.
\begin{figure}[h!]
	\centering
	\includegraphics[width=0.3\linewidth]{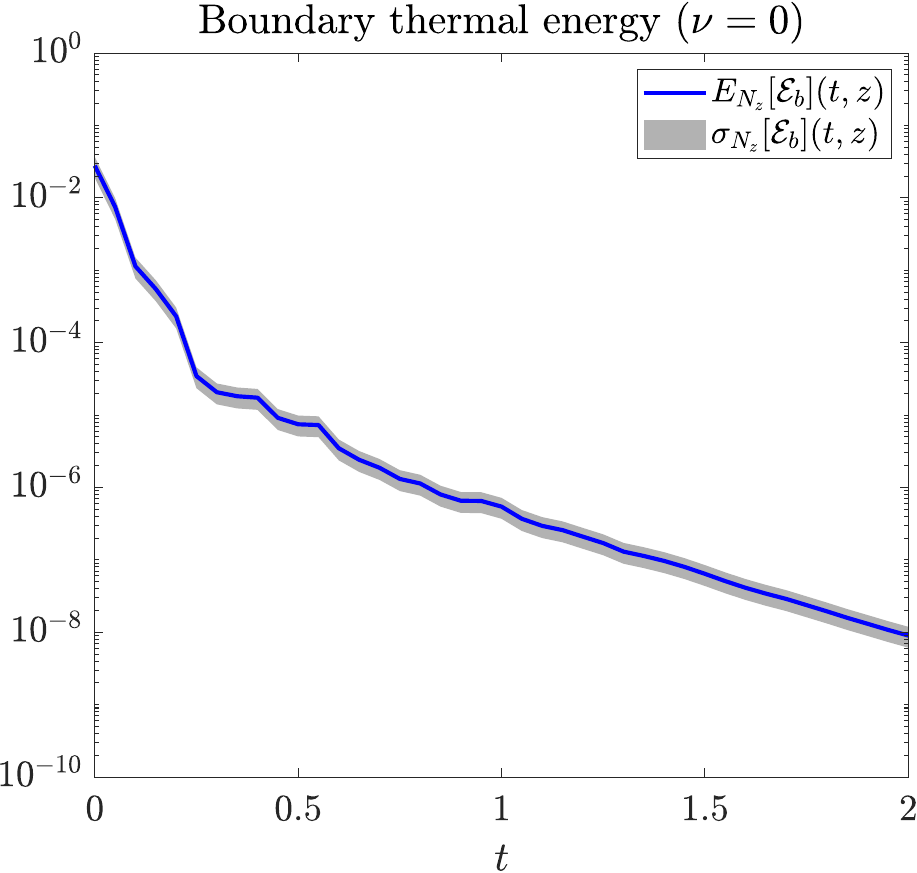}
	\includegraphics[width=0.3\linewidth]{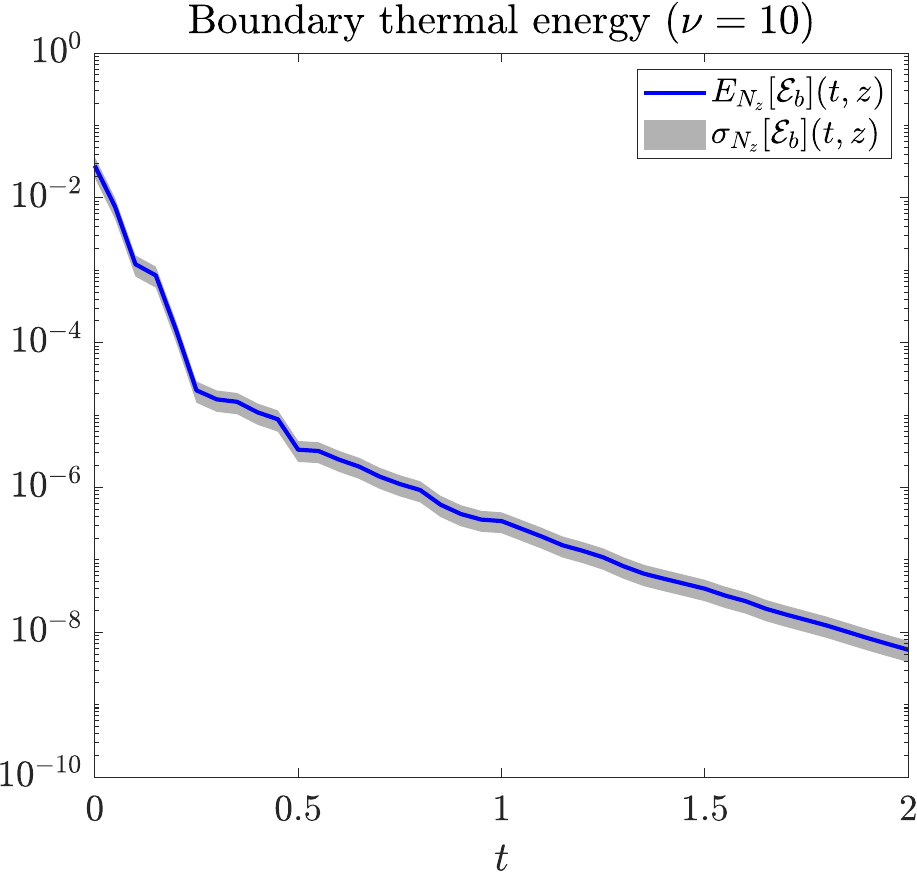}
	\includegraphics[width=0.3\linewidth]{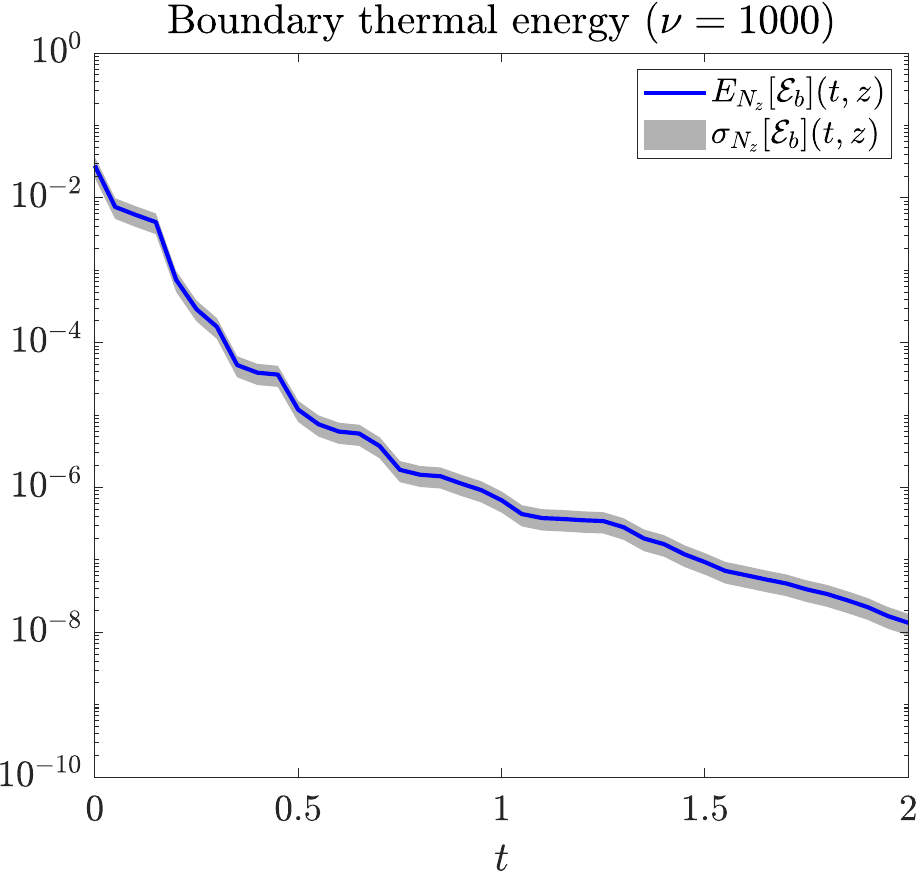}\\
	\includegraphics[width=0.3\linewidth]{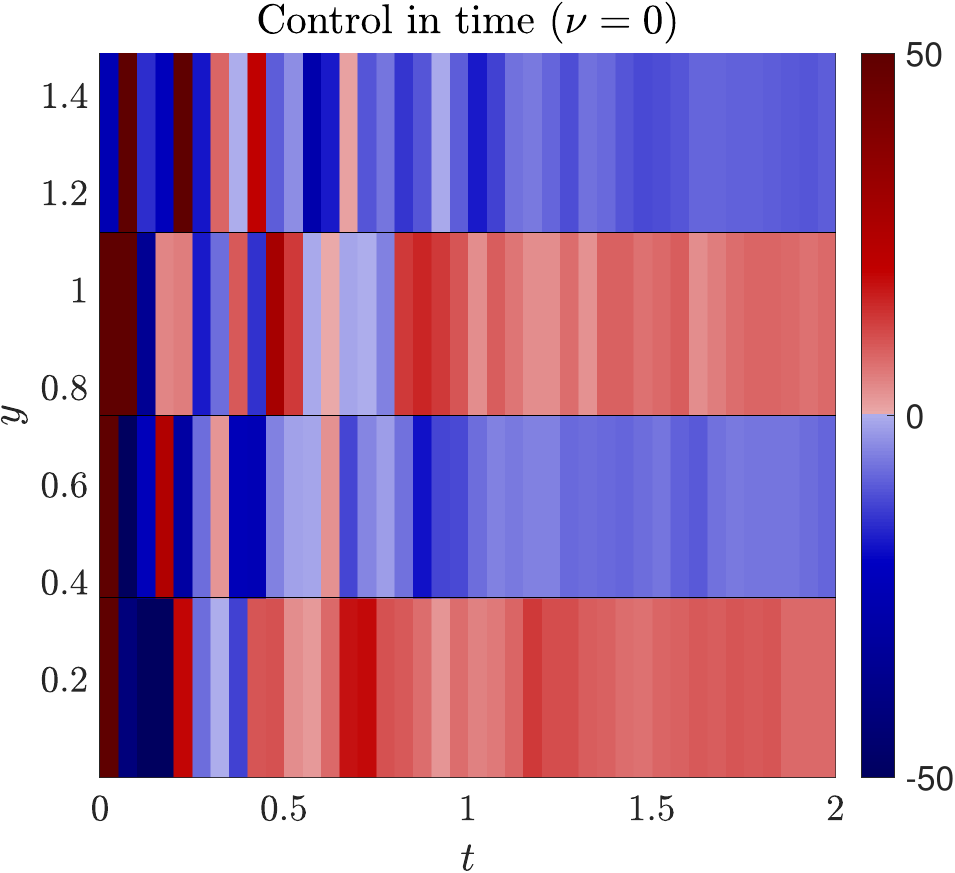} 	\includegraphics[width=0.3\linewidth]{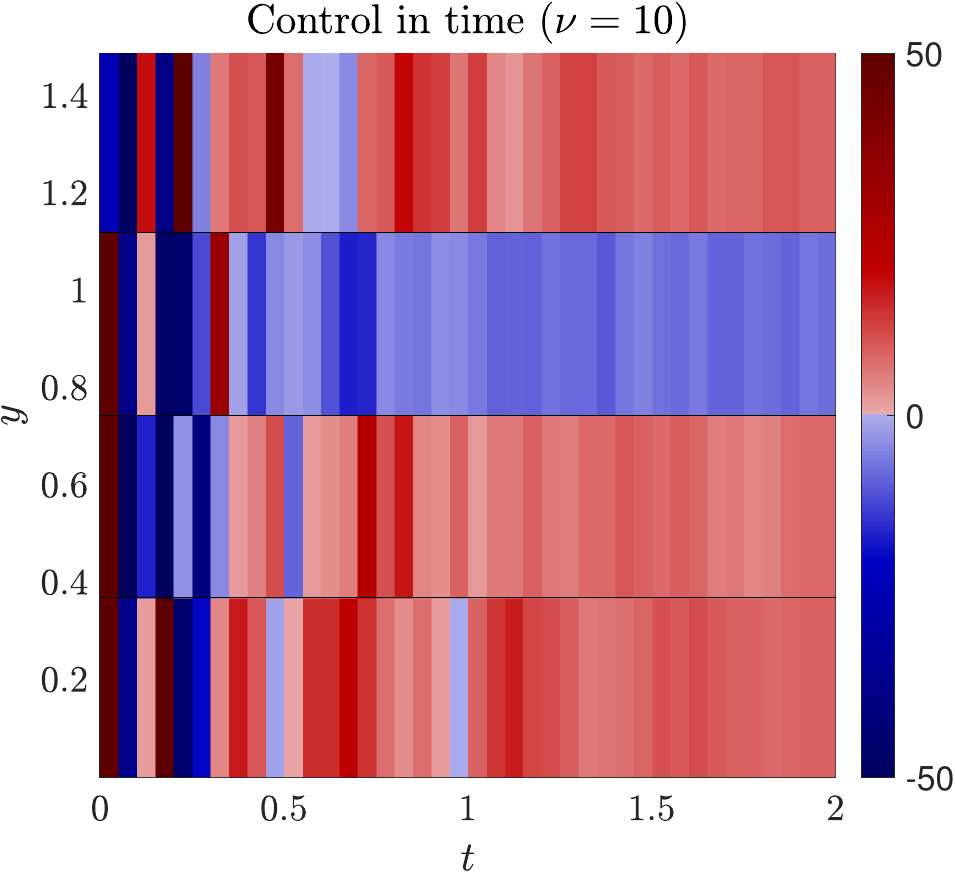} 
	\includegraphics[width=0.3\linewidth]{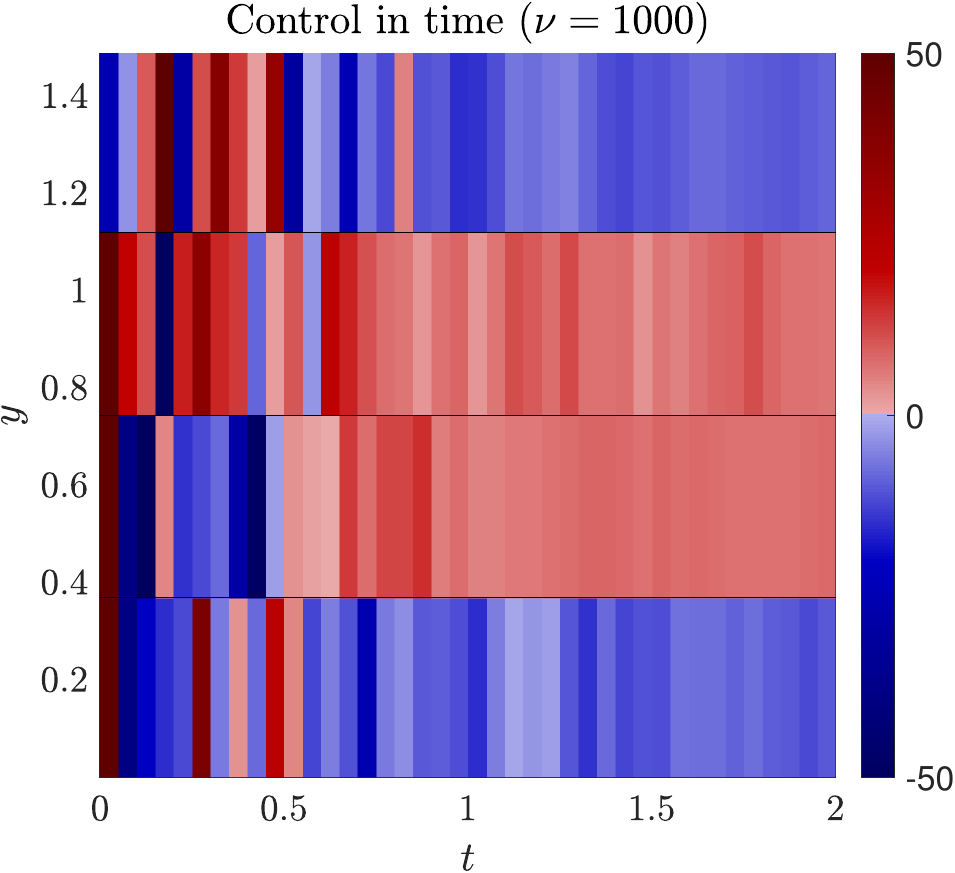} 
	\caption{Two dimensional sod shock tube test with control. First row: mean thermal energy at the boundaries and the relative standard deviation depicted as a shaded area are shown. Second row: the value of the magnetic field is reported. On the left $\nu = 0$, in the centre $\nu = 10$, and on the right $\nu = 1000$.}
	\label{fig:sod2D_control_Bmax_energy}
\end{figure} 

In Figure~\ref{fig:comparison_Bmax_Bmean}, we compare the thermal energy at the boundaries obtained using two different functionals within the control strategy. In particular, we consider the cases where $ \mathcal{P}[\cdot] $ is defined as in equations~\eqref{eq:R_mean} and~\eqref{eq:R_max}, and denote the corresponding controls by $ B_{\text{mean}} $ and $ B_{\text{max}} $, respectively.
For both cases, we plot the mean thermal energy at the boundaries along with the corresponding standard deviation, represented as a shaded area, under the same three collisional regimes previously considered.  
In all scenarios, the control strategy based on minimizing the worst-case behavior ($ B_{\text{max}} $) proves to be slightly more effective than the one based on the average behavior ($ B_{\text{mean}} $) in achieving the desired confinement.
\begin{figure}[h!]
	\centering
	\includegraphics[width=0.3\linewidth]{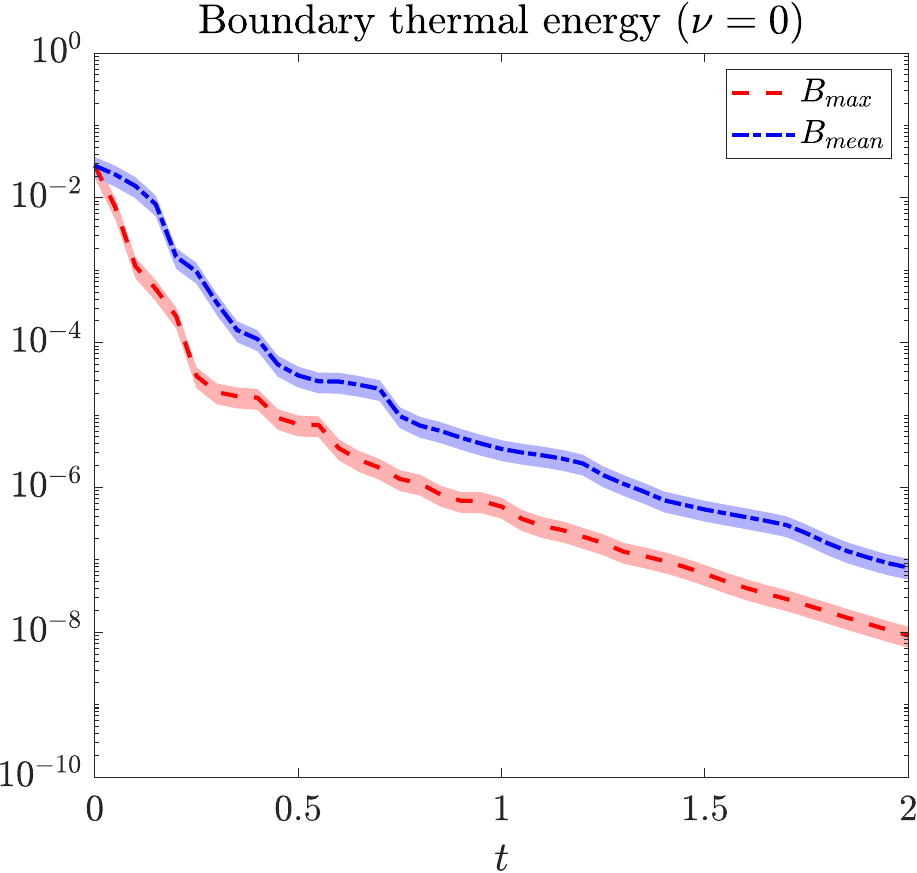}  
	\includegraphics[width=0.3\linewidth]{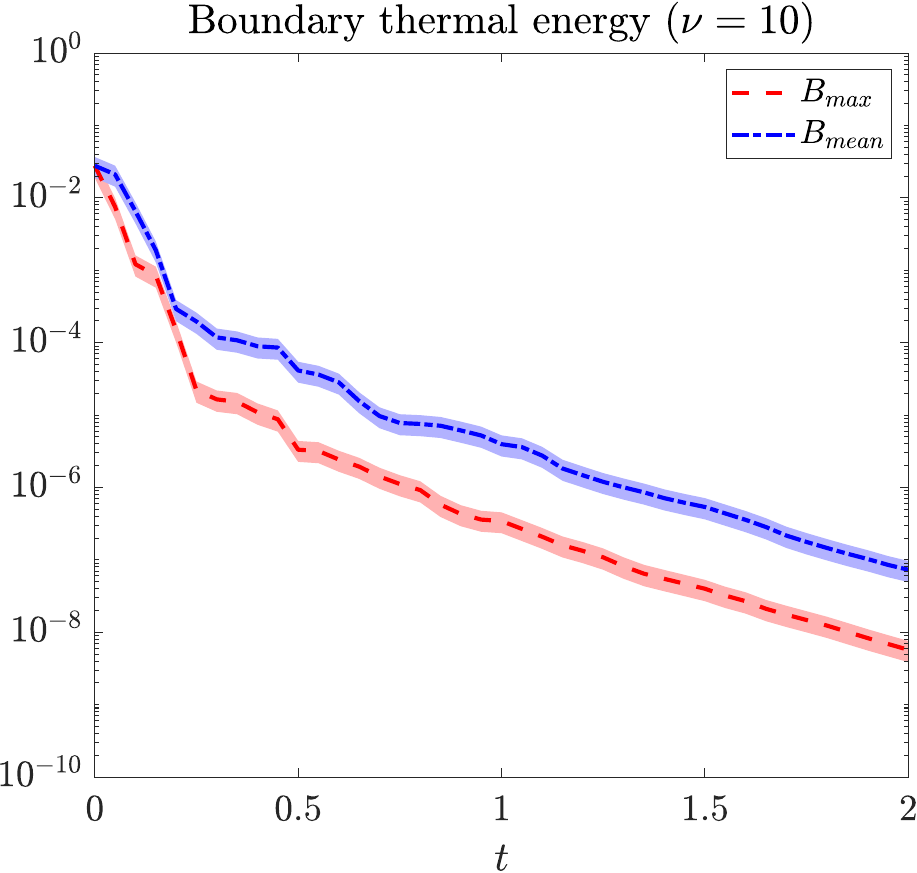}  
	\includegraphics[width=0.3\linewidth]{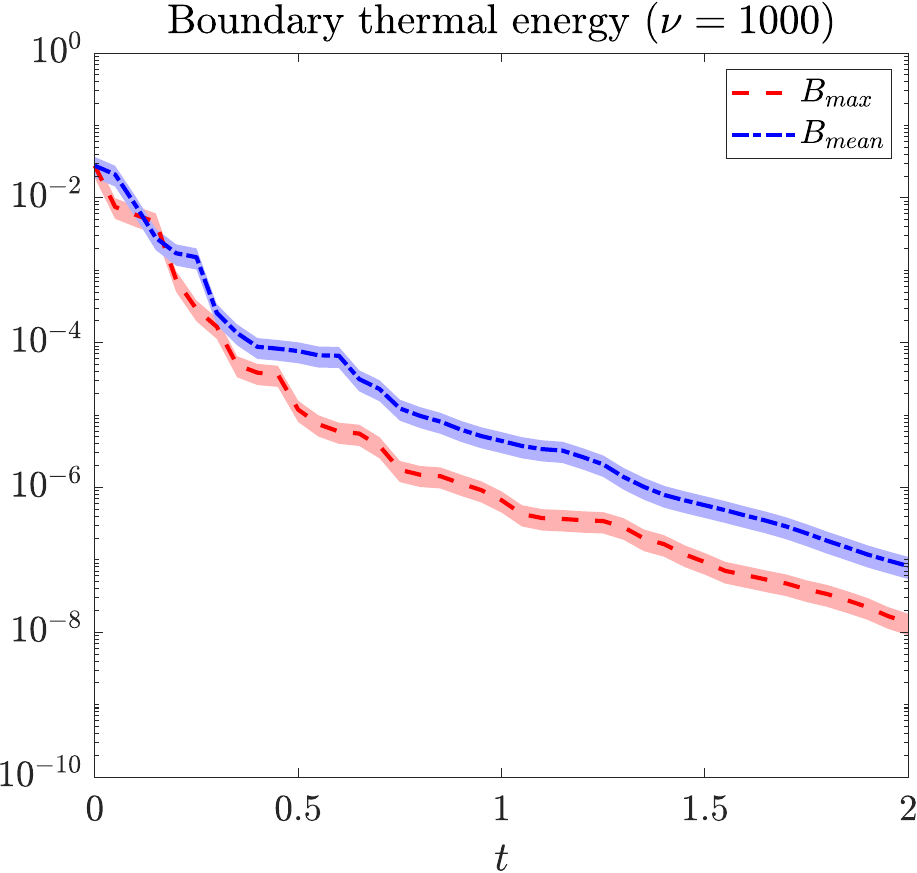}  
	\caption{Two dimensional sod shock tube test with control. Mean thermal energy and relative standard deviation at the boundaries using $\mathcal{P}(\cdot)$ as defined in \eqref{eq:R_max} (referred to as $B_{max}$) and \eqref{eq:R_mean} (referred to as $B_{mean}$). On the left $\nu = 0$, in the centre $\nu = 10$, and on the right $\nu = 1000$. }
	\label{fig:comparison_Bmax_Bmean}
\end{figure} 

We now test the proposed strategy in long time by assuming $T=6$, and that the magnetic field remains fixed for a certain number of time steps, $n_{switch}$.  We rely on a simplify scenario, setting $z=0$, and we suppose to be in the fully collisional regime, setting $\nu = 1000$. Figure \ref{fig:B_fixed} on the left shows the thermal energy at the boundary assuming $n_{switch} = 1,2,4,6,8,10$, and in the center and on the right the value of the control in time for $n_{switch} = 2$ and $n_{switch} = 10$ respectively.  While keeping the control fixed over multiple time steps proves effective in the short term, we observe that over longer times the thermal energy begins to oscillate. This indicates that, even though the values remain small, part of the mass gradually moves toward the $y$-boundaries of the domain, as expected. The shape of the control for $n_{switch}=10$ suggests that in long time the strategy is no more effective, due to the presence of oscillations in the value of the thermal energy at the boundaries. 
\begin{figure}[h!]
	\centering
	\includegraphics[width=0.29\linewidth]{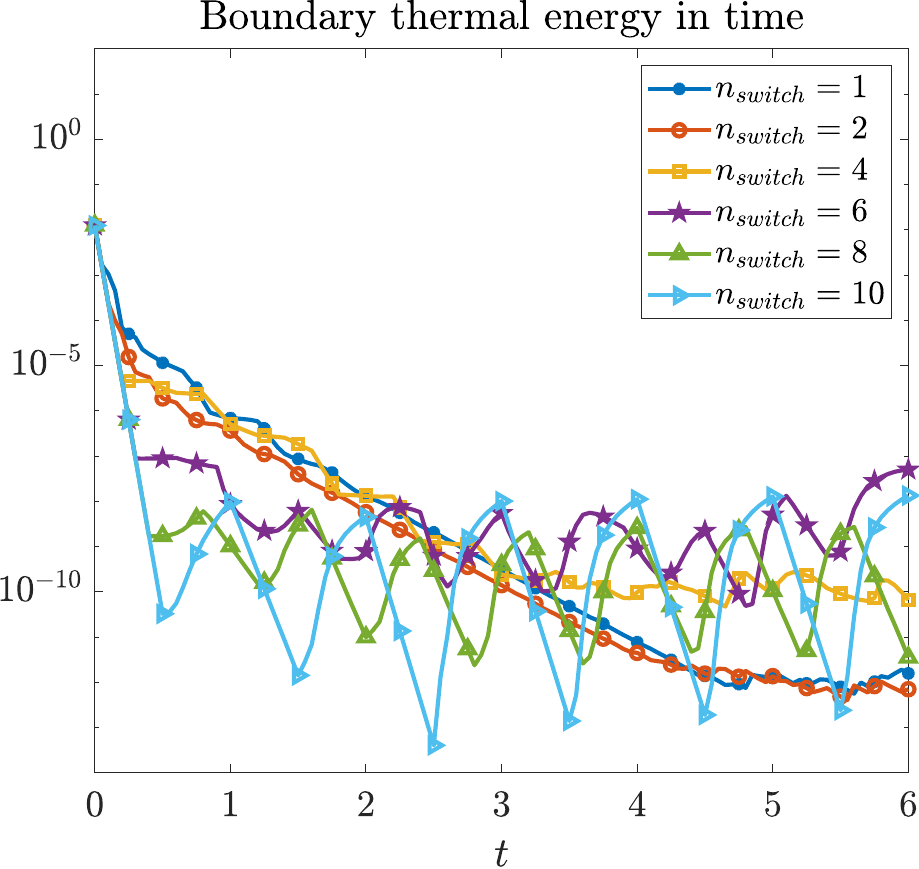}  
	\includegraphics[width=0.3\linewidth]{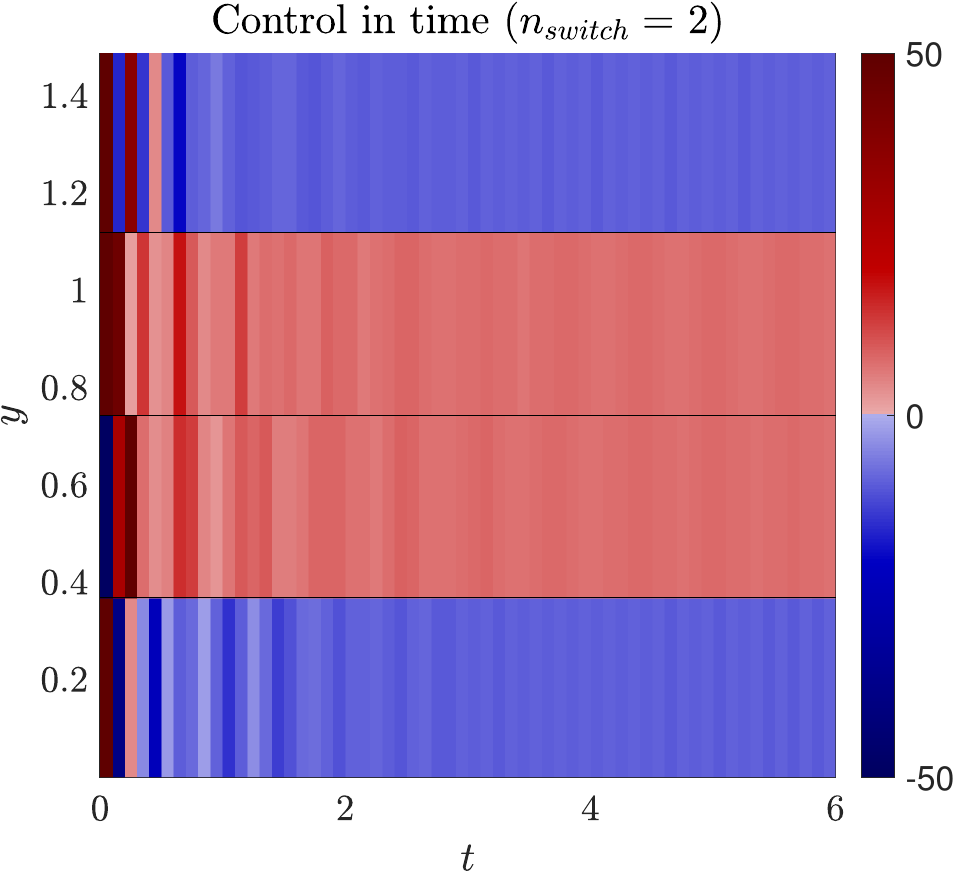}  
	\includegraphics[width=0.3\linewidth]{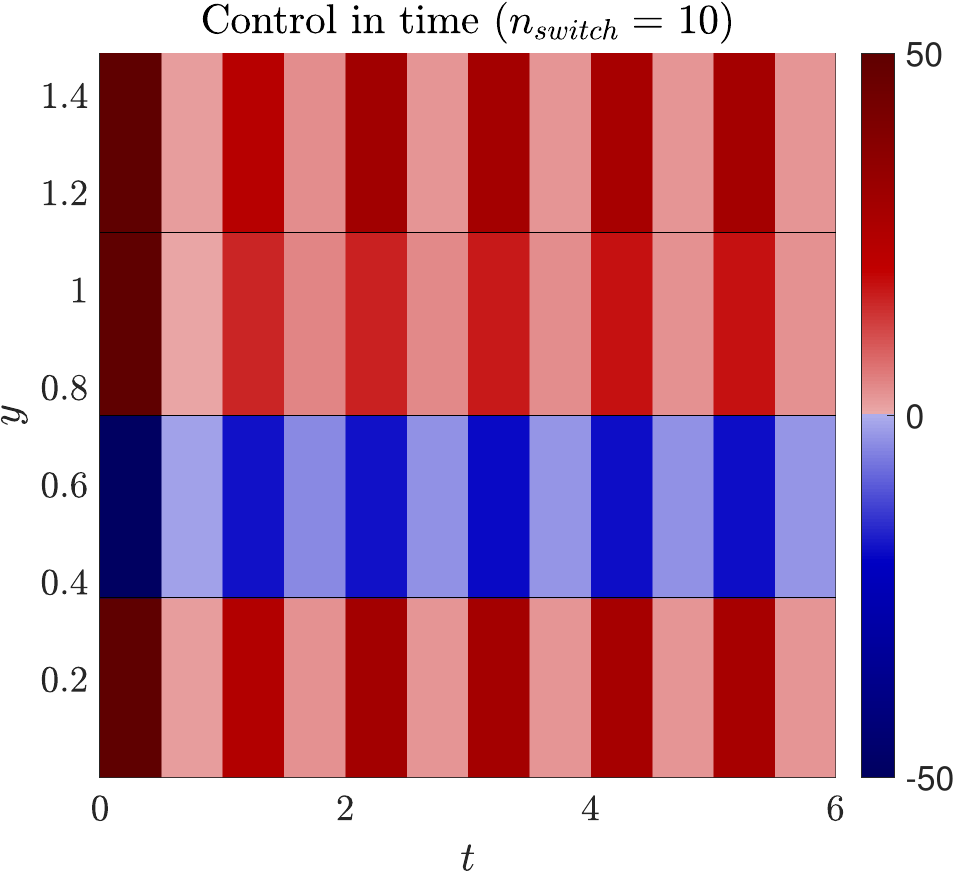}  
	\caption{Two dimensional sod shock tube test with control. On the left, thermal energy at the boundaries as the magnetic field remains fixed for $n_{switch}$ time steps. In the center and on the right, the value of the control $B$ in time for $n_{switch} = 2$ and $n_{switch} = 10$. }  
	\label{fig:B_fixed}
\end{figure}   

We conclude by testing the effectiveness of the proposed control strategy under more challenging conditions, by increasing both the initial temperature of the plasma and the number of horizontal cells $ N_c $ over which the control is active.  
We consider a deterministic setting by fixing $ z = 0 $, and we assume the system operates in the fully collisional regime.  
Starting from the initial temperature $ T_0 $ defined in equation~\eqref{eq:rho0_T0_2D} (as used in the previous tests), we introduce three additional temperature profiles defined as follows:
\begin{equation}\label{eq:T0_10}
	\begin{split}
		\tilde{T}_0(\xx) &= \chi(y \in [0,0.5) \cup (1,1.5)) + 10\, \chi(y \in [0.5,1]), \\
		\bar{T}_0(\xx)    &= 5\, \chi(y \in [0,0.5) \cup (1,1.5)) + 50\, \chi(y \in [0.5,1]), \\
		\hat{T}_0(\xx)    &= 10\, \chi(y \in [0,0.5) \cup (1,1.5)) + 100\, \chi(y \in [0.5,1]).
	\end{split}
\end{equation}
The first row of Figure~\ref{fig:comparison_T0} shows, on the left, the evolution of thermal energy at the boundaries over time as the initial temperature increases, and on the right, the effect of varying the number of horizontal control cells $ N_c $, for a fixed initial temperature $ T_0(\xx) = \hat{T}_0 $.  
The second row reports the evolution in time of the control field for different values of $ N_c $, again using $ \hat{T}_0 $ as the initial temperature.
The results confirm that the effectiveness of the control strategy is inversely proportional to the initial plasma temperature, as expected.  
Moreover, the plot on the right suggests that increasing the number of horizontal cells does not lead to significant improvements in thermal energy reduction.
\begin{figure}[h!]
	\centering
	\includegraphics[width=0.3\linewidth]{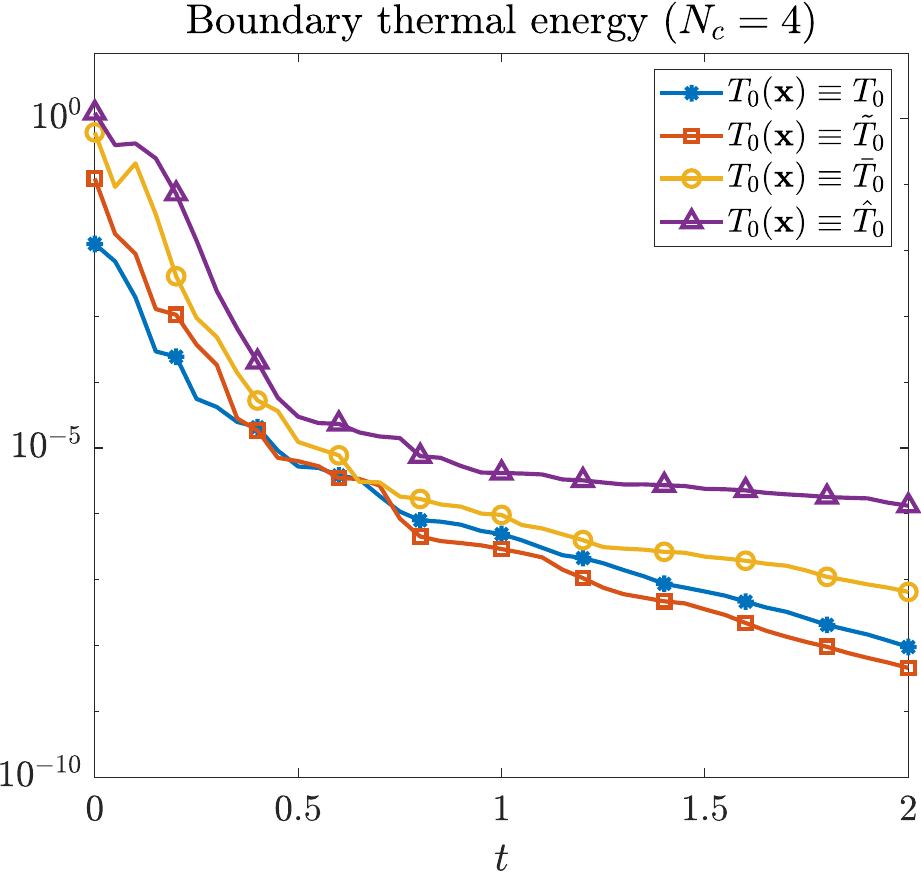} 
	\includegraphics[width=0.3\linewidth]{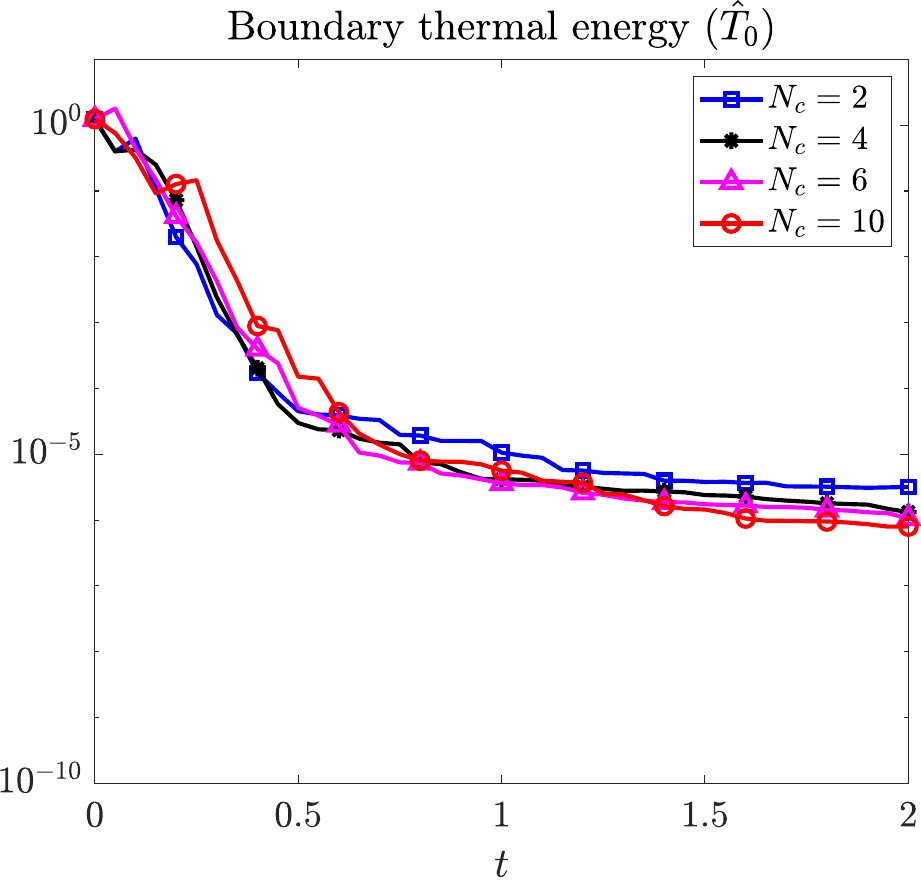}  \\
	\includegraphics[width=0.3\linewidth]{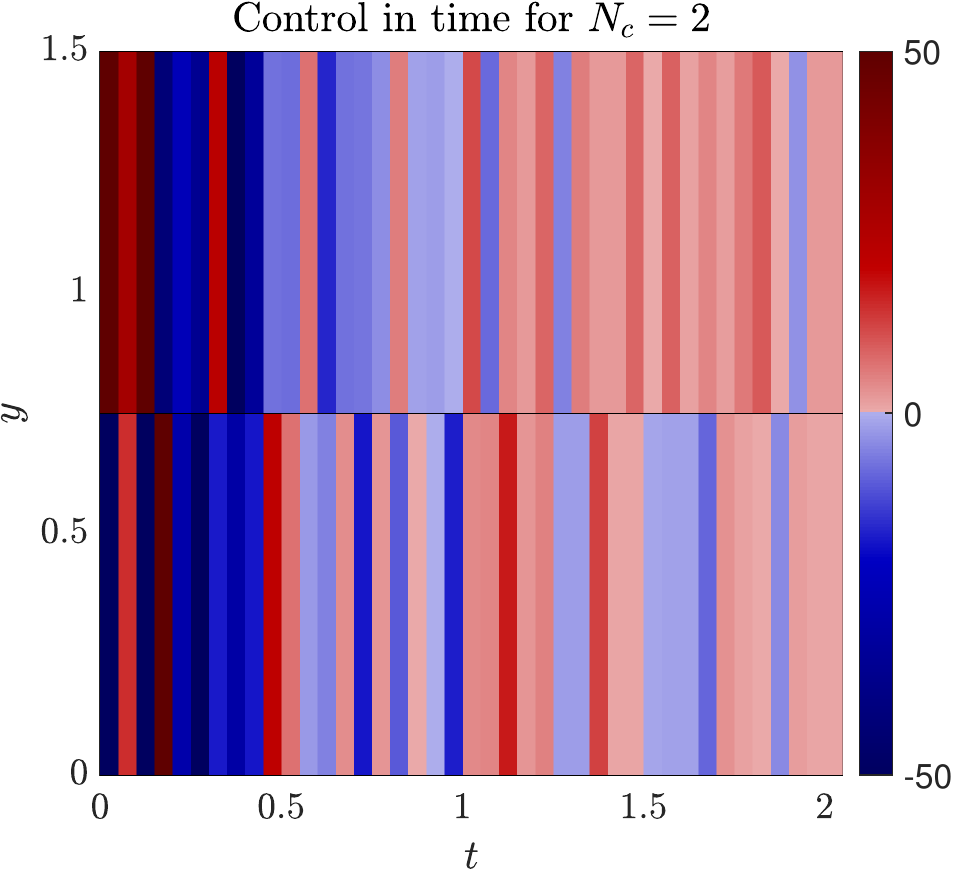}
	\includegraphics[width=0.3\linewidth]{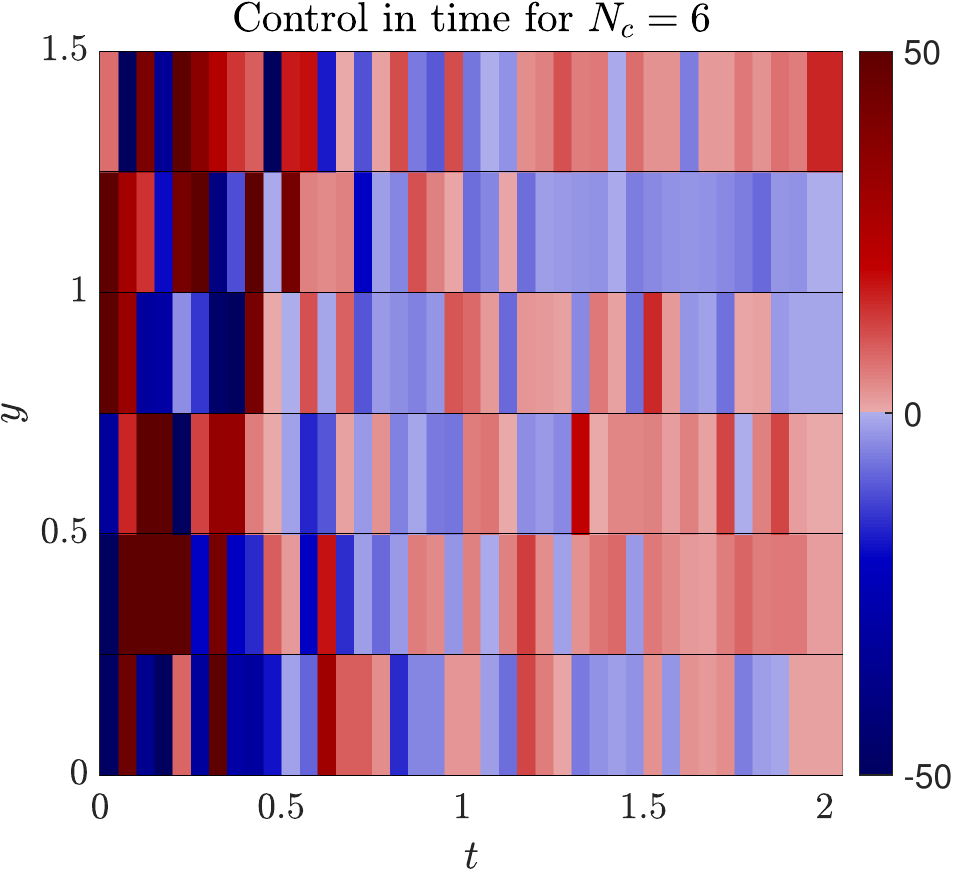}		\includegraphics[width=0.3\linewidth]{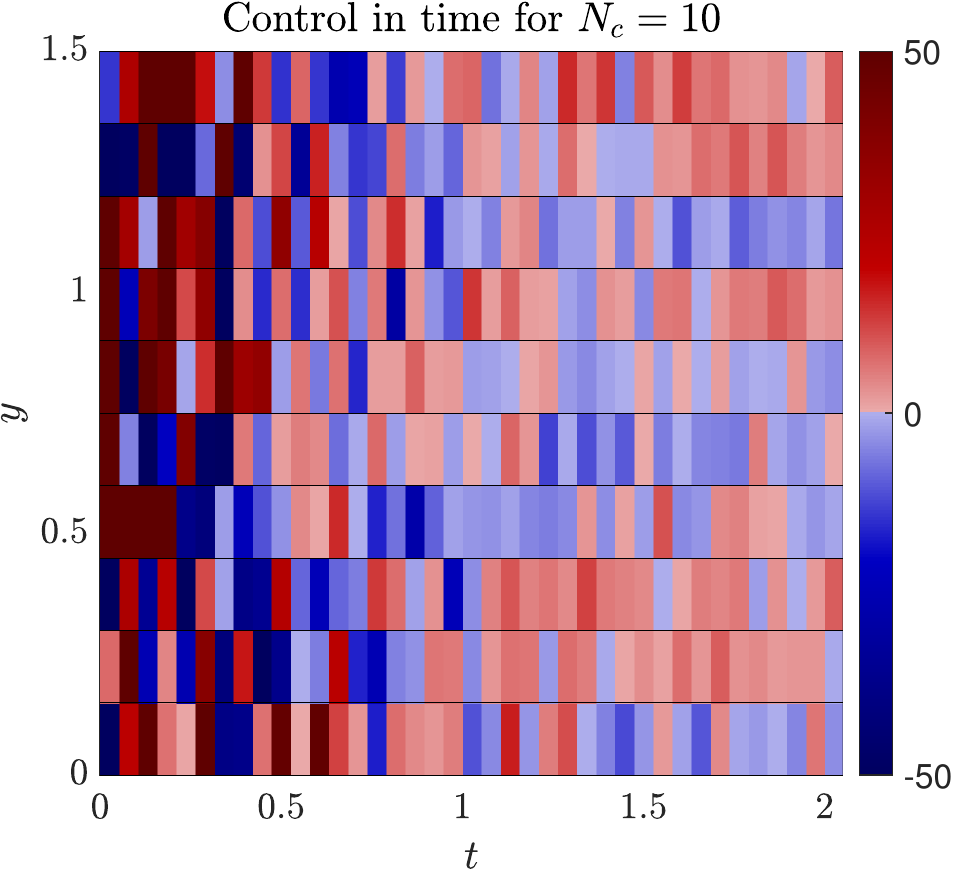}
	\caption{Two dimensional sod shock tube test with control. First row: thermal energy at the boundaries for different values of the initial temperature $T_0$ (left), thermal energy at the boundaries for different values of the number of cells $N_c$ (right). Second row: the value of the control in time as $N_c$ varies. }
	\label{fig:comparison_T0}
\end{figure} \\

Finally, we assess the effectiveness of the control in the case of larger plasma $\beta$ larger. The plasma $\beta$
\begin{equation}\label{eq:plasma_beta}
	\beta \approx \frac{\rho T}{B^2},
\end{equation} 
where $\rho$ and $T$ are the plasma density and temperature, and $B$ is the magnetic field, is a dimensionless parameter that quantifies the ratio between the plasma thermal pressure and the magnetic pressure, measuring how effective the applied magnetic field is at confining the plasma \cite{freidberg2008plasma}. 
A value of $\beta \ll 1$ indicates a magnetically dominated regime, where the plasma pressure is much smaller than 
the magnetic pressure and the magnetic field lines strongly determine the plasma dynamics. Conversely, 
$\beta \gg 1$ corresponds to a plasma-pressure dominated regime, where magnetic confinement becomes ineffective. 
For magnetically confined fusion plasmas, efficient operation requires intermediate values, where the plasma pressure 
is sufficiently high to sustain fusion reactions but still small enough to remain stably confined.
In practice, stability constraints impose an upper bound on $\beta$, \cite{troyon1984mhd}. In particular, the Troyon limit provides an 
empirical criterion for the maximum achievable $\beta$ in a realistic Tokamak, that is $\beta \approx 0.01 \text{--} 0.05$. Figure \ref{fig:energy_M} (left) shows the boundary thermal energy at time $t=2$, computed according to \eqref{eq:energy}, for different values of the initial temperature and magnetic field strength $M$. The results indicate that the control strategy is effective except for cases with low $M$.  
On the right, Figure \ref{fig:energy_M} shows the temporal evolution of the thermal energy for different values of plasma $\beta$. The results indicate that in regimes with relatively large $\beta$, the control strategy becomes less effective, suggesting that further improvements may be necessary.
\begin{figure}[h!]
	\centering
	\includegraphics[width=0.35\linewidth]{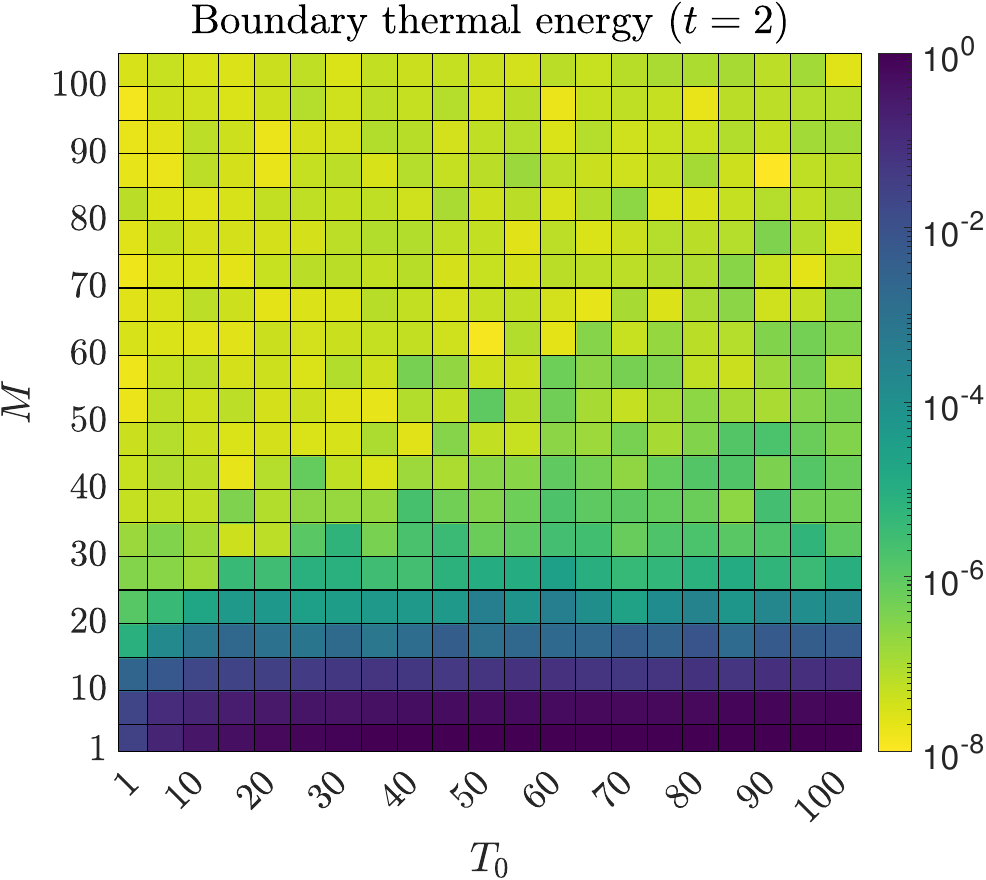}
	\includegraphics[width=0.33\linewidth]{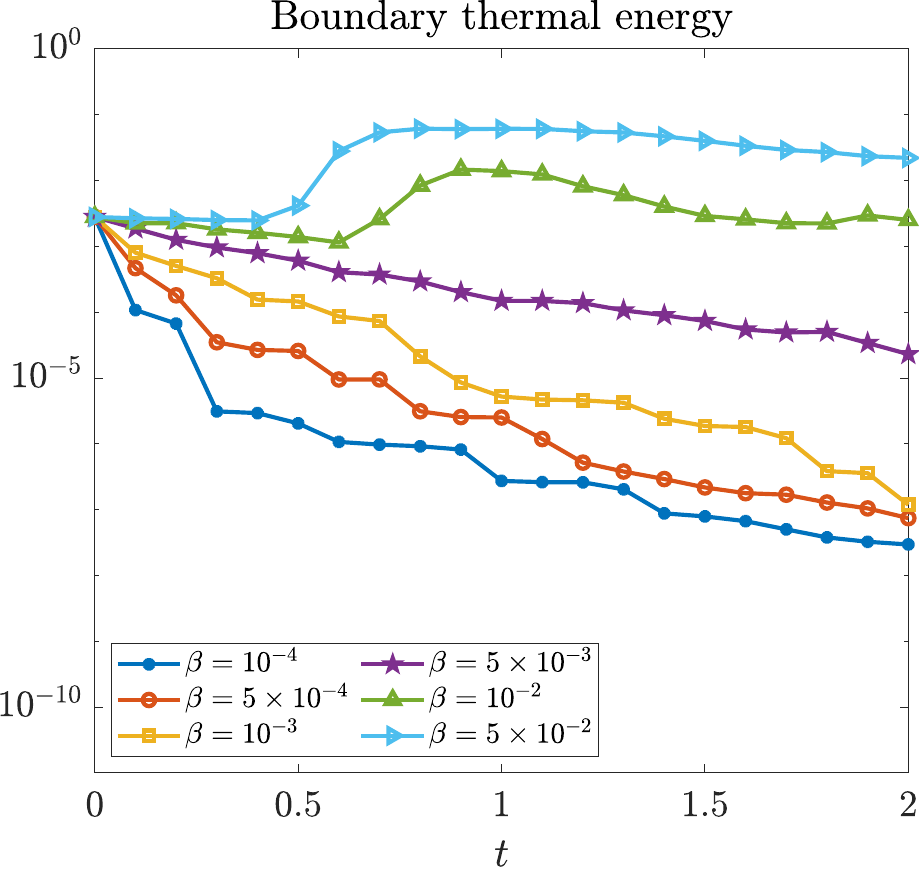}
	\caption{Two dimensional sod shock tube test with control. On the left, the boundary thermal energy at time $t=2$ for different temperature values and control strengths. On the right, the boundary thermal energy in time as the plasma $\beta$ varies. }
	\label{fig:energy_M}
\end{figure}
\subsection{Kelvin-Helmholtz instability}
In this section, we examine a variant of the Kelvin--Helmholtz instability in the context of charged particles; see \cite{crouseilles2010conservative,sonnendrucker1999semi,chacon2016}.  
We conduct an analysis similar to that of the previous section, comparing the controlled and uncontrolled scenarios in the presence of both uncertainty and collisions.  
Periodic boundary conditions are imposed in the $ x $-direction, and reflective boundary conditions are imposed in the $ y $-direction.  
The computational domain is defined as $ x \in [0, 40] $, $ y \in [-5, 5] $, and the initial density is given by
\begin{equation}\label{eq:kelvin_initial_density}
	\rho_0(\xx) = \frac{1.5}{2\pi} \, \text{sech}\left(\frac{y}{0.9}\right) \left(1 + \epsilon_0 \cos(3k_0 x + \epsilon_1 \sin(k_0 x))\right),
\end{equation}
with parameters $ k_0 = 0.15 $, $ \epsilon_0 = 0.1 $, and $ \epsilon_1 = 0.001 $.
The initial distribution function is defined as
\begin{equation}\label{eq:f0_kelvin}
	f_0(\xx,\vv,z) = f_0^+(\xx,\vv,z) \chi(y \geq 0) + f_0^-(\xx,\vv,z) \chi(y < 0),
\end{equation}
where
\begin{equation}\label{eq:fo_kelvin_p_m}
	f_0^\pm(\xx,\vv,z) = \frac{\rho_0(\xx)}{2\pi T_0(\xx,z)} \exp\left( -\frac{(v_x \pm u_x)^2 + v_y^2}{2 T_0(\xx,z)} \right),
\end{equation}
and $ u_x = 1 $ is the mean velocity in the $ x $-direction.  
The initial temperature includes uncertainty and is defined by
\begin{equation}\label{eq:temp_sigma_kelvin}
	T_0(\xx,z) = 0.15 + 0.25 z,
\end{equation}
where $ z \sim p(z) $ follows a uniform distribution over $[0,1]$.

Figure~\ref{fig:ic_kelvin} displays the estimated initial mean distribution and thermal energy, reconstructed over a grid of size $ m_x \times m_y $, with $ m_x = m_y = 128 $.  
Simulations are performed using $ N = 10^7 $ particles, up to final time $ T = 250 $, with time step $ h = 0.5 $.  
In the uncontrolled case, we set $ B = 1.5 $, whereas in the controlled scenario, the control acts on $ N_c = 4 $ horizontal cells.  
The mean density and thermal energy at the boundaries are computed as in equation~\eqref{eq:energy}, with the boundary region defined as $ \Omega_b = [-5, -4.8] \cup [4.8, 5] $.

\subsubsection{Uncontrolled case}

\begin{figure}[h!]
	\centering
	\includegraphics[width=0.35\linewidth]{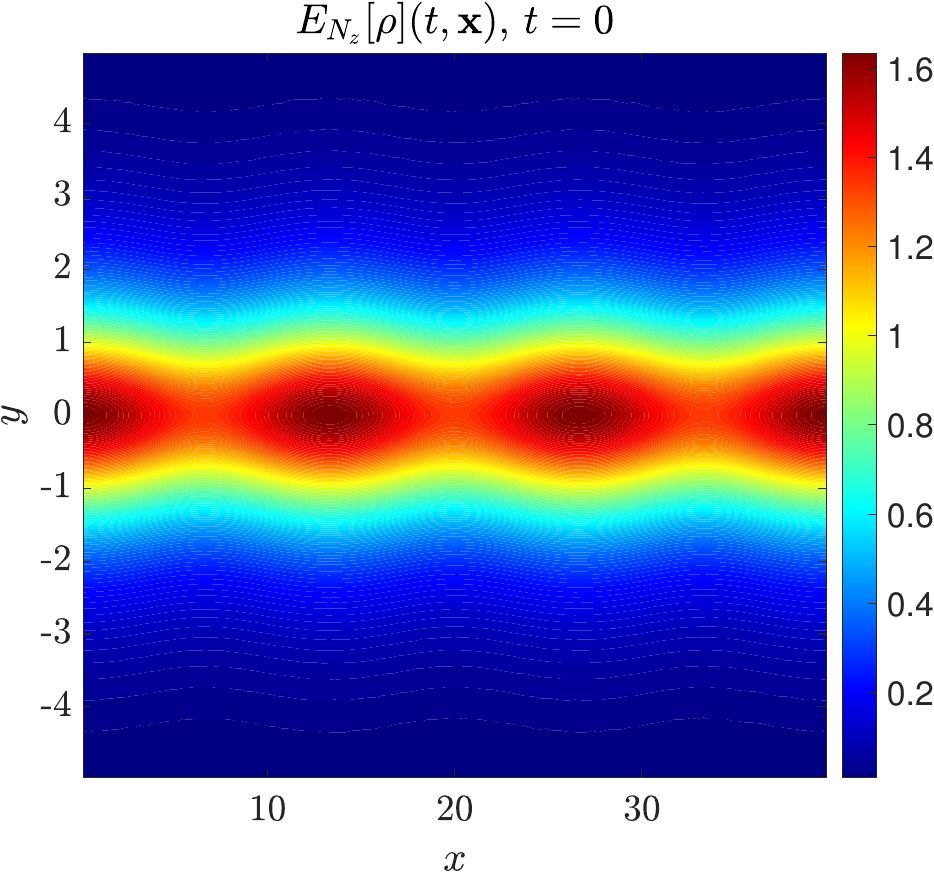} 
	\includegraphics[width=0.35\linewidth]{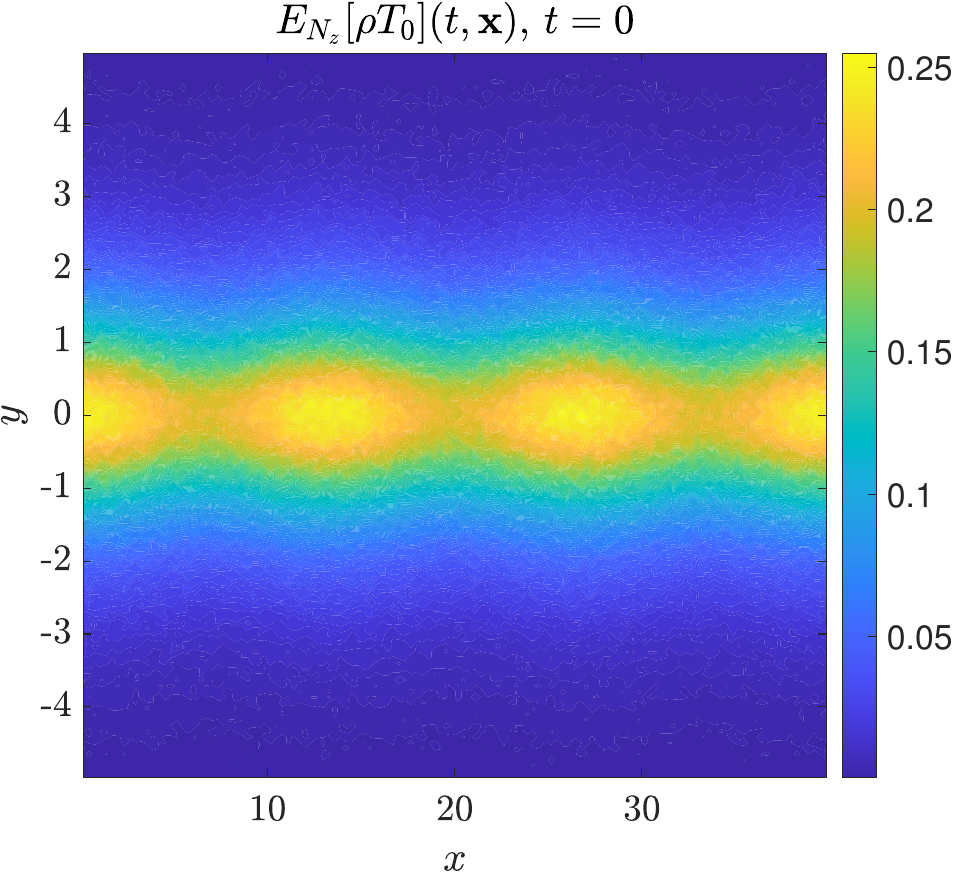} 
	\caption{Kelvin-Helmholtz instability. Initial mean density and thermal energy.}
	\label{fig:ic_kelvin}
\end{figure} 
Figure~\ref{fig:no_control_kelvin} shows the density field at times $ t = 50 $, $ t = 75 $, and $ t = 200 $ for the uncontrolled case with $ B = 1.2 $, across the three collisional regimes. In the fully collisional regime ($ \nu = 1000 $), the instability develops rapidly, whereas in the collisionless regime ($ \nu = 0 $) it emerges more gradually over time. An intermediate behavior is observed for the quasi-collisional case with $ \nu = 10 $.
\begin{figure}[h!]
	\centering
	\includegraphics[width=0.3\linewidth]{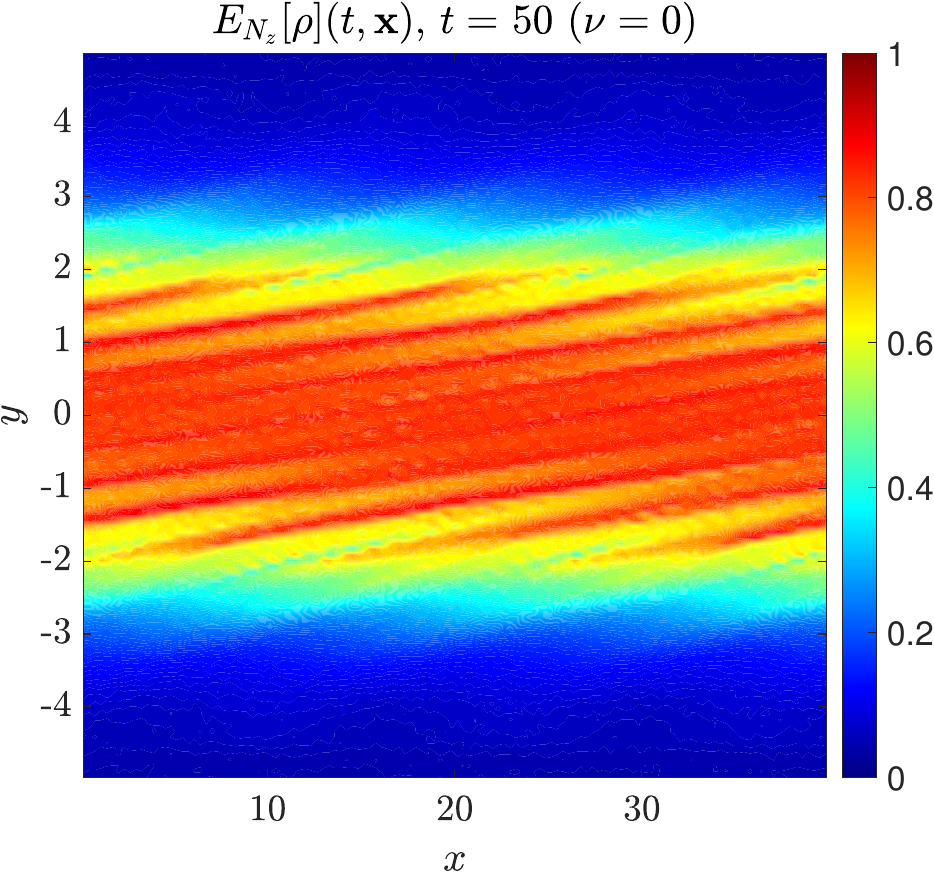} 
	\includegraphics[width=0.3\linewidth]{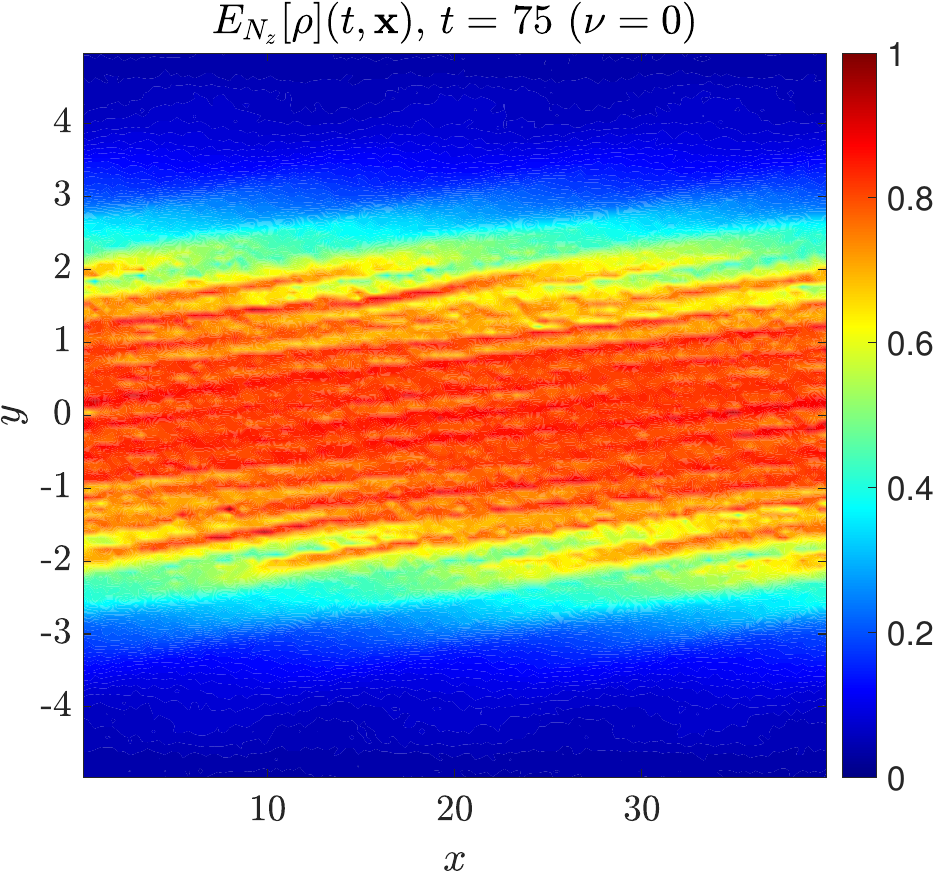} 
	\includegraphics[width=0.3\linewidth]{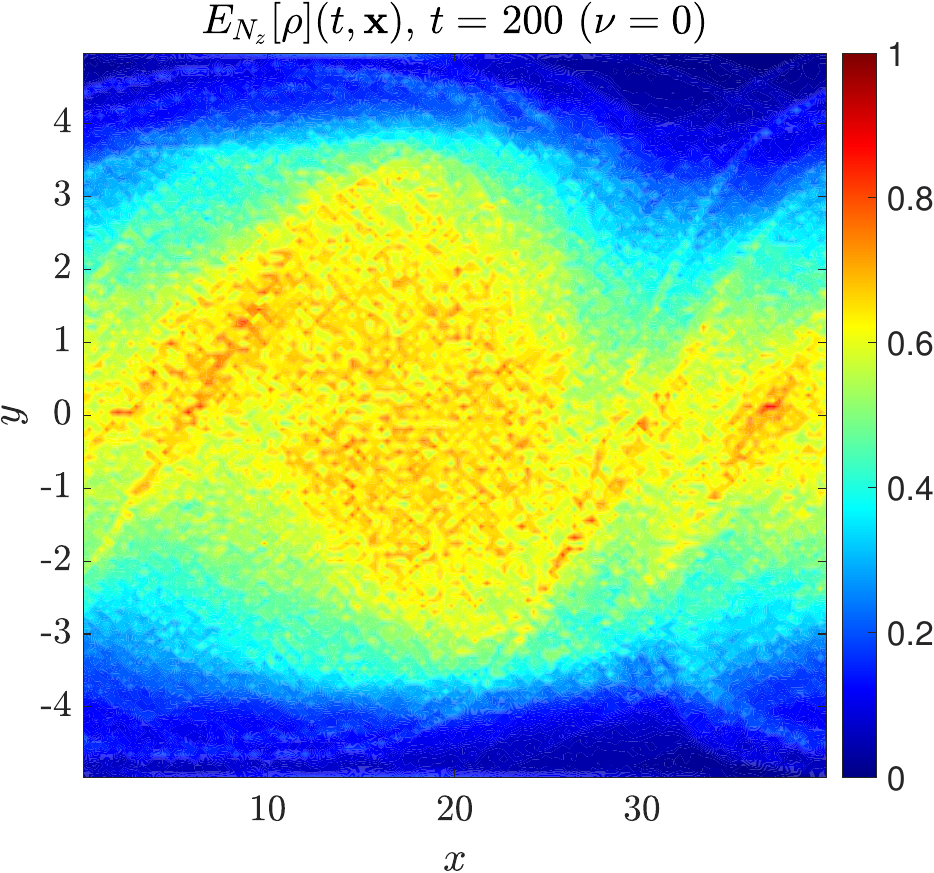}  \\
	\includegraphics[width=0.3\linewidth]{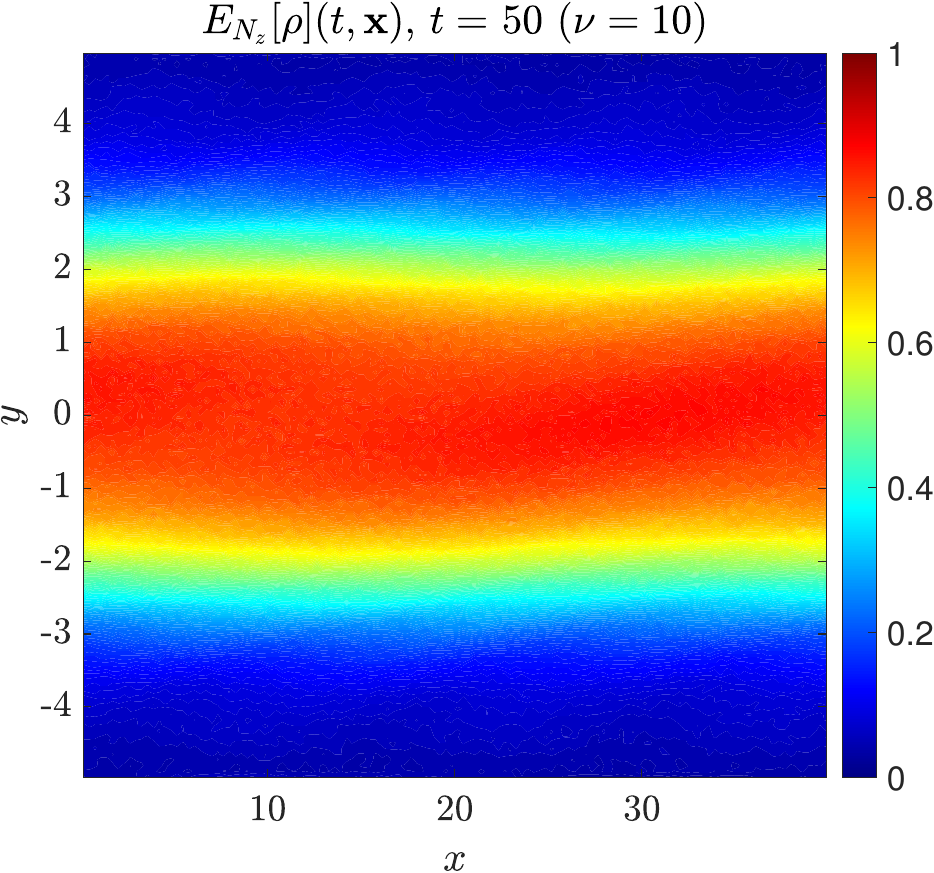} 
	\includegraphics[width=0.3\linewidth]{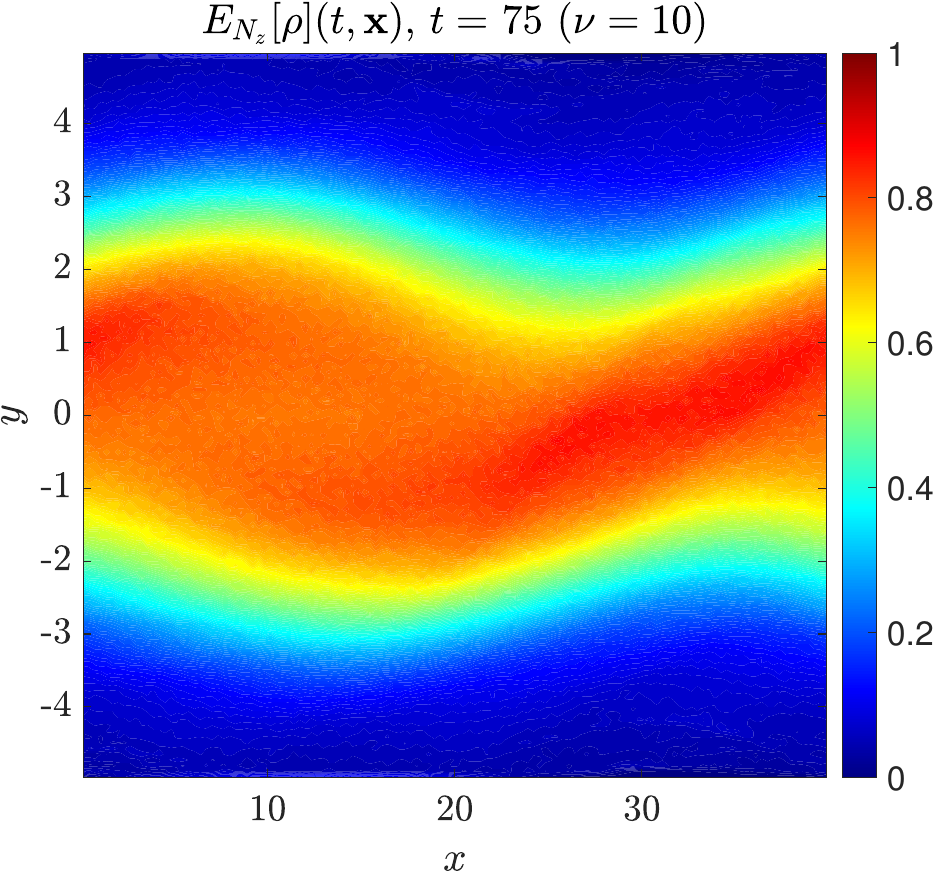} 
	\includegraphics[width=0.3\linewidth]{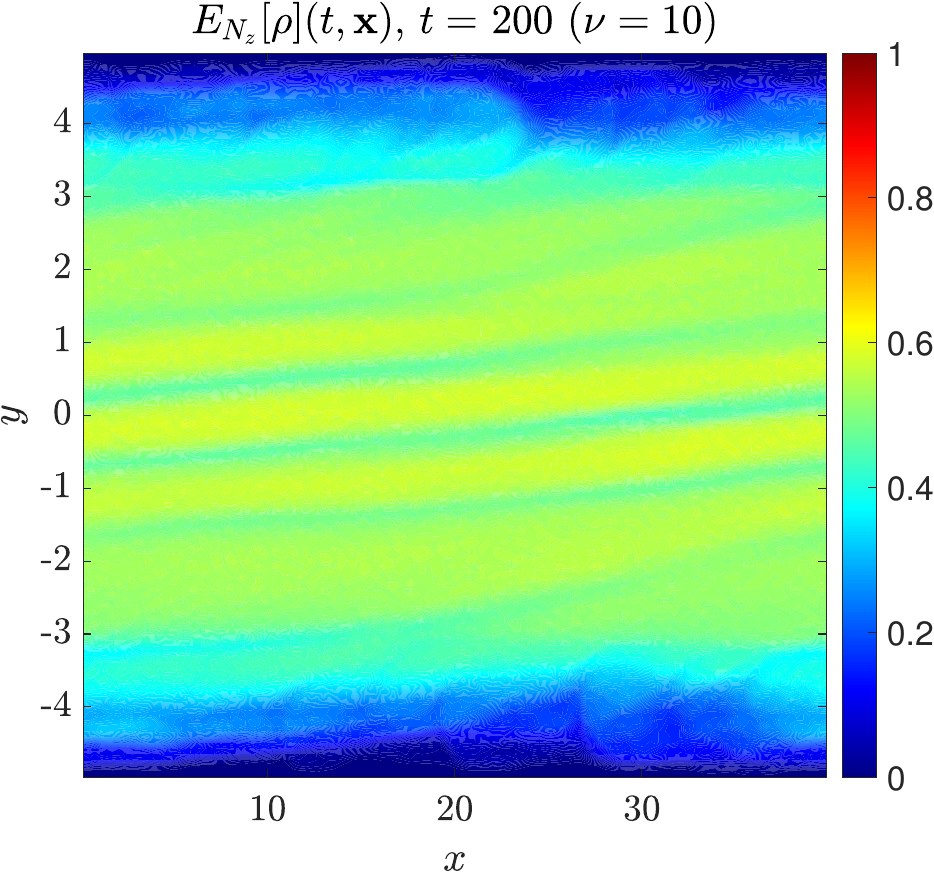}  \\
	\includegraphics[width=0.3\linewidth]{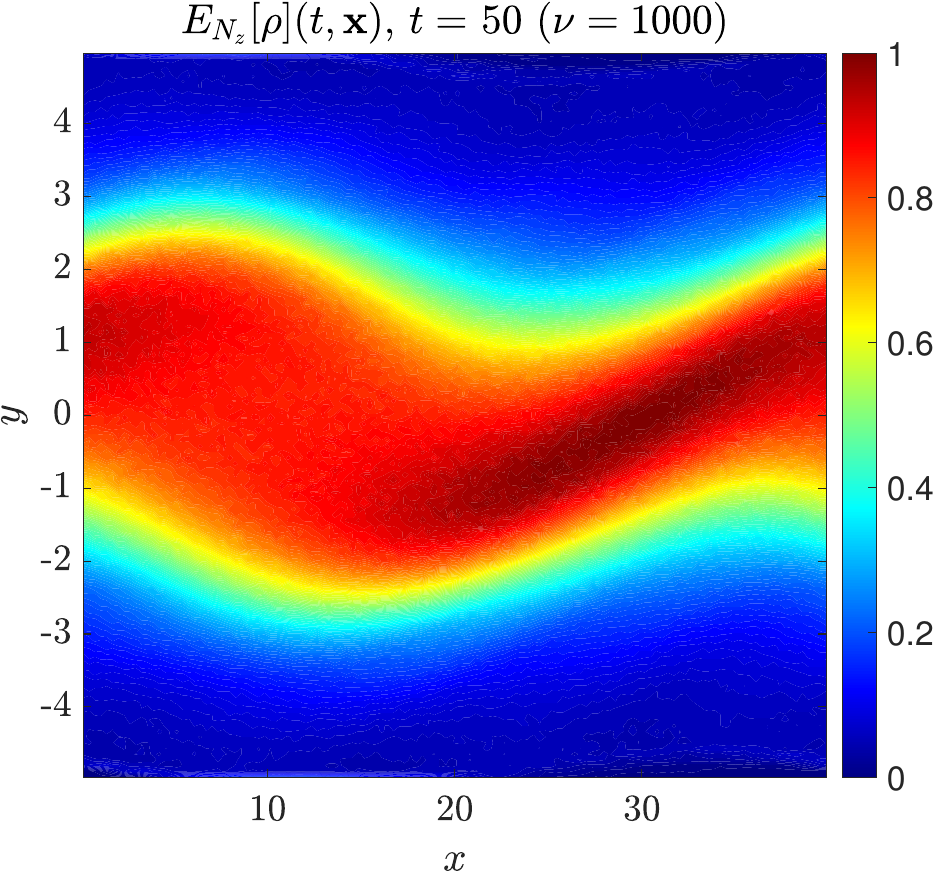} 
	\includegraphics[width=0.3\linewidth]{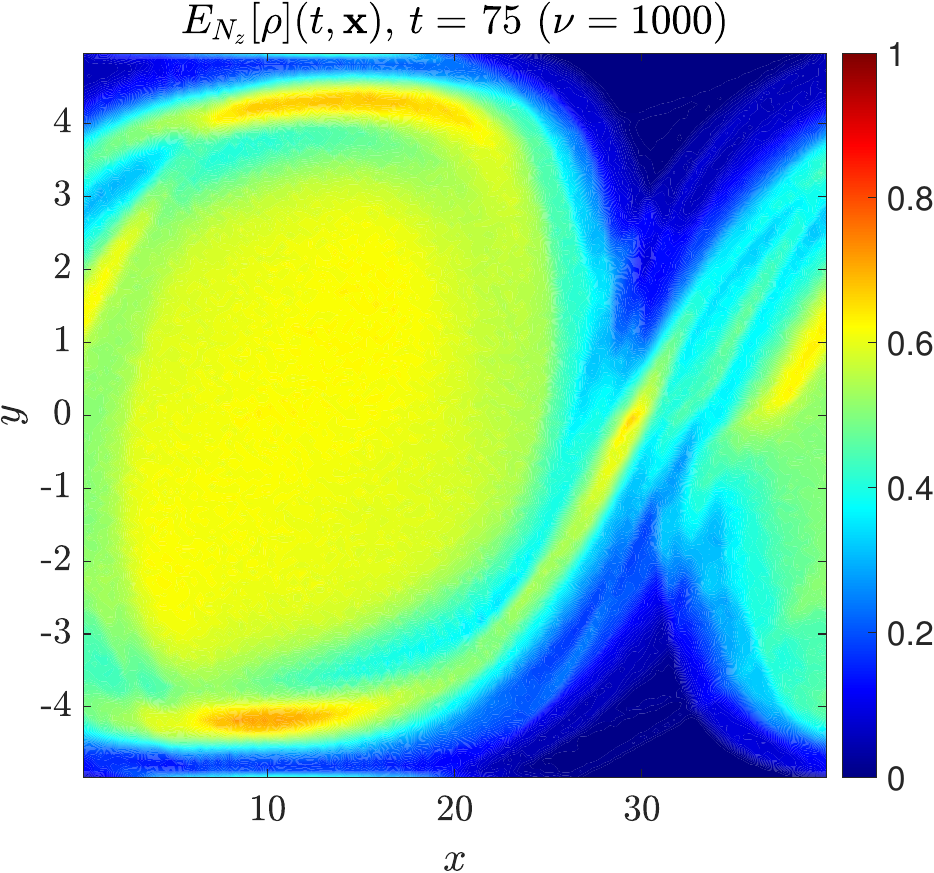} 
	\includegraphics[width=0.3\linewidth]{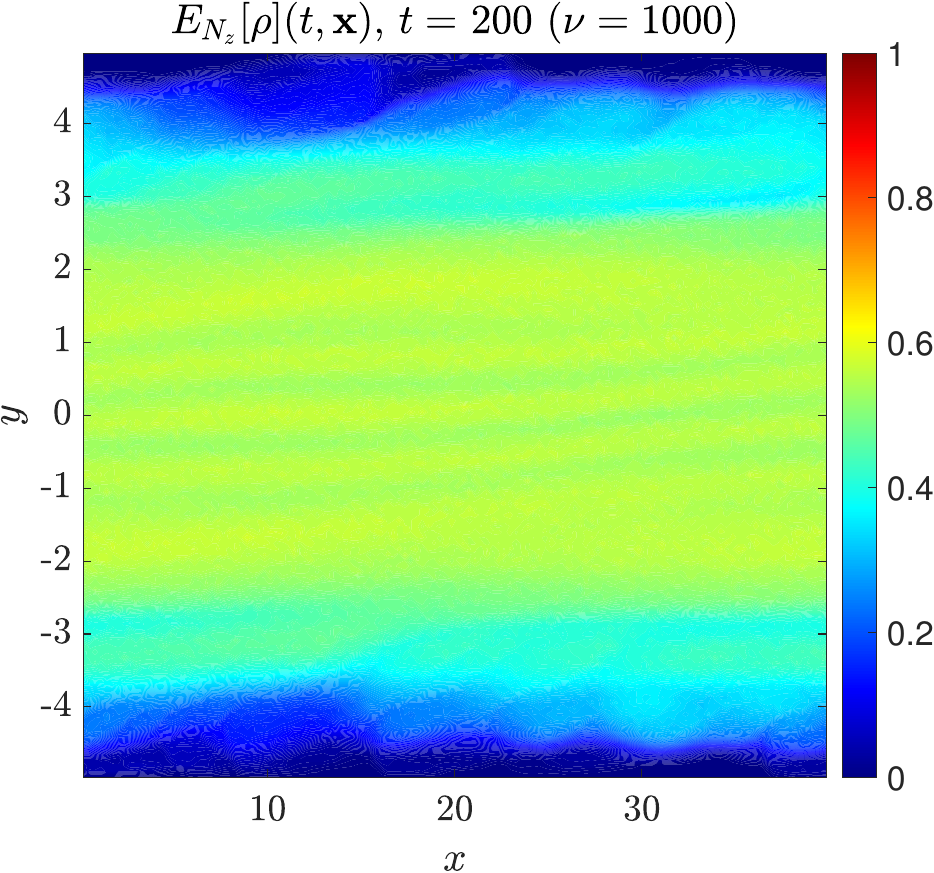}   
	\caption{Kelvin-Helmholtz instability with $B(t,\xx) = 1.2$. Mean density at $t=50$, $t =75$ and $t=200$. First row: $\nu = 0$. Second row: $\nu=10$. Third row: $\nu=1000$.  }
	\label{fig:no_control_kelvin}
\end{figure} 
In Figure \ref{fig:kelvin_energy_noControl} the thermal energy at the boundary in time for $\nu=0$ (on the left), $\nu =10$ (in the centre), and $\nu =1000$ (on the right) is shown. Once that the instability arises, the boundary thermal energy starts to increase. 
\begin{figure}[h!]
	\centering 
	\includegraphics[width=0.3\linewidth]{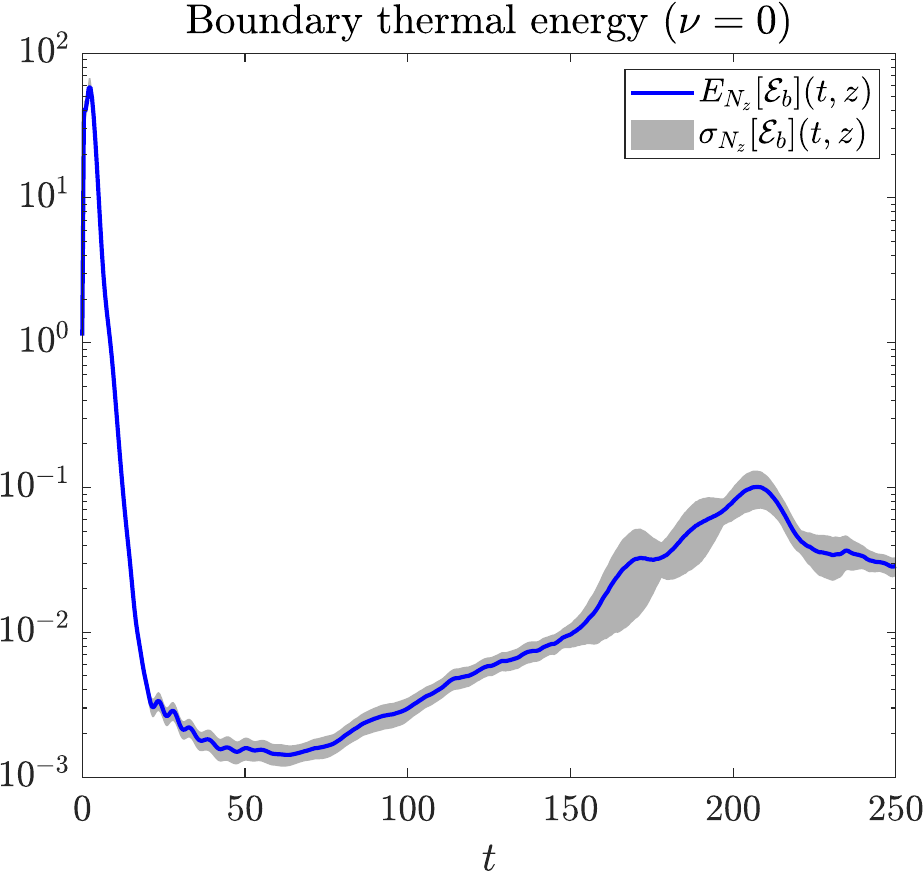} 
	\includegraphics[width=0.3\linewidth]{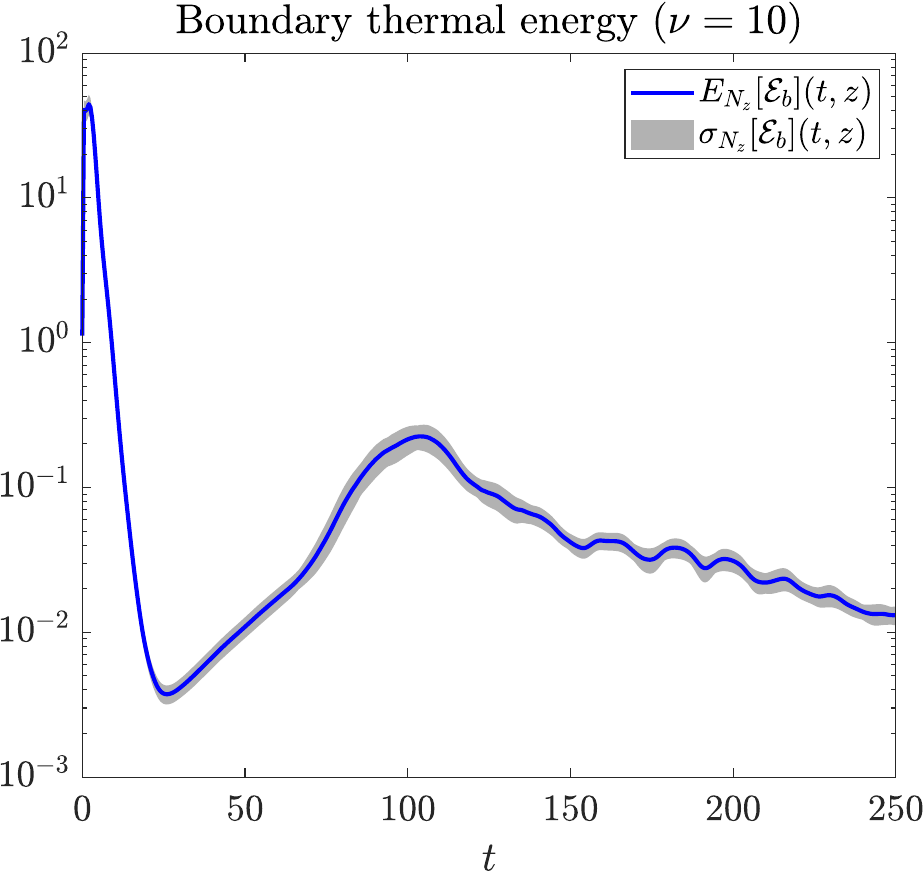} 
	\includegraphics[width=0.3\linewidth]{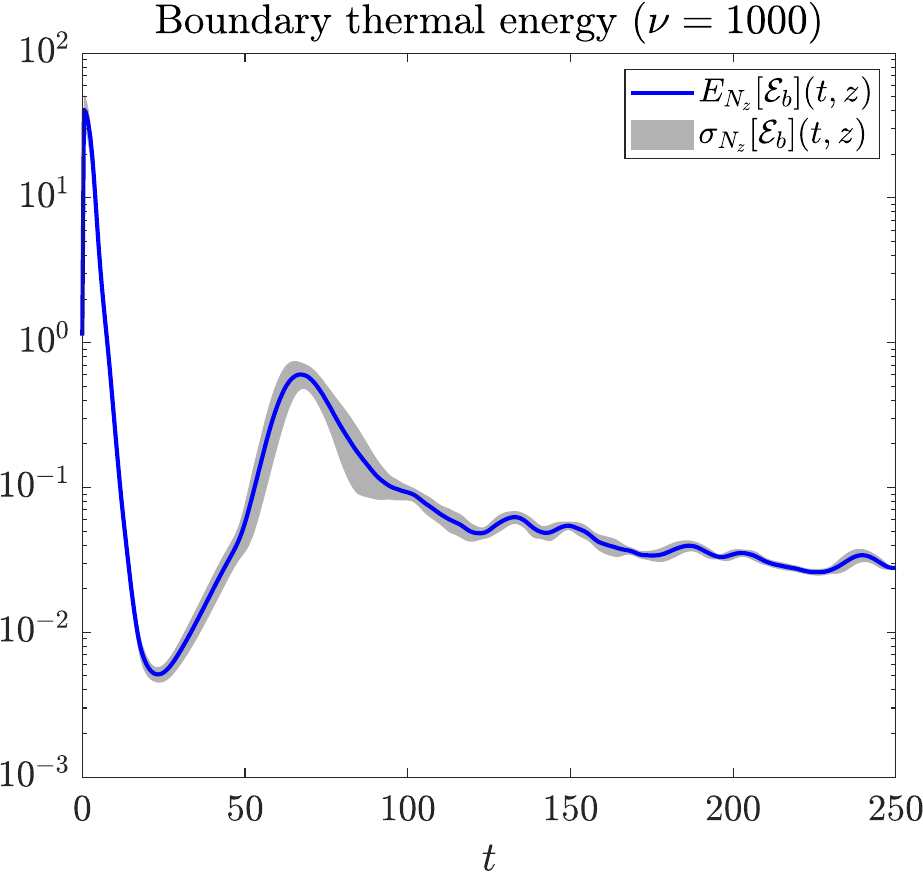}   
	\caption{Kelvin-Helmholtz instability with $B(t,\xx) = 1.2$. Thermal energy at the boundaries for $\nu = 0$ (on the left), $\nu = 10$ (in the centre) and $\nu = 1000$ (on the right). The mean value is depicted in blue, while the standard deviation as a shaded area. }
	\label{fig:kelvin_energy_noControl}
\end{figure}

\subsubsection{Feedback controlled case}
The case of the robust control is shown in Figure \ref{fig:control_kelvin}, for the non collisional regime. We set $\alpha_\textrm{x} = 5$,$ \alpha_\textrm{v} = 15$, $\beta_\textrm{x} =2$,$\beta_\textrm{v}= 12$, and $\gamma = 2.5\times10^{-3}$, $M=100$, with target $\hat{y}=0$, to confine the mass at the center of the domain as for the previous case. In the first row, three snapshots of the dynamics taken at time $t=50$, $t=75$ and $t=200$ are depicted. In the second row the plot of the boundary thermal energy in time (on the left), and on control values in time in the four horizontal cells (on the right) can be observed. The figure demonstrates the effectiveness of the control strategy developed.  Similar results are obtained in the quasi and fully collisional regimes for $\nu =10$, and $\nu=1000$, and are not shown for brevity.
\begin{figure}[h!]
	\centering
	\includegraphics[width=0.3\linewidth]{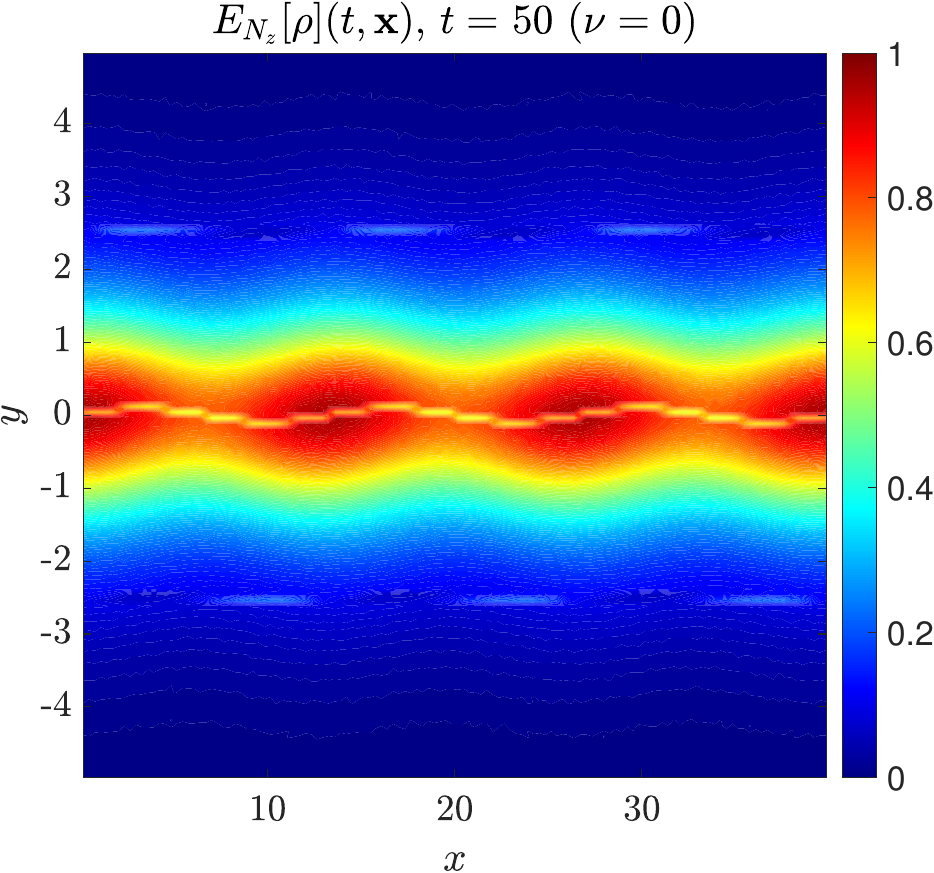} 
	\includegraphics[width=0.3\linewidth]{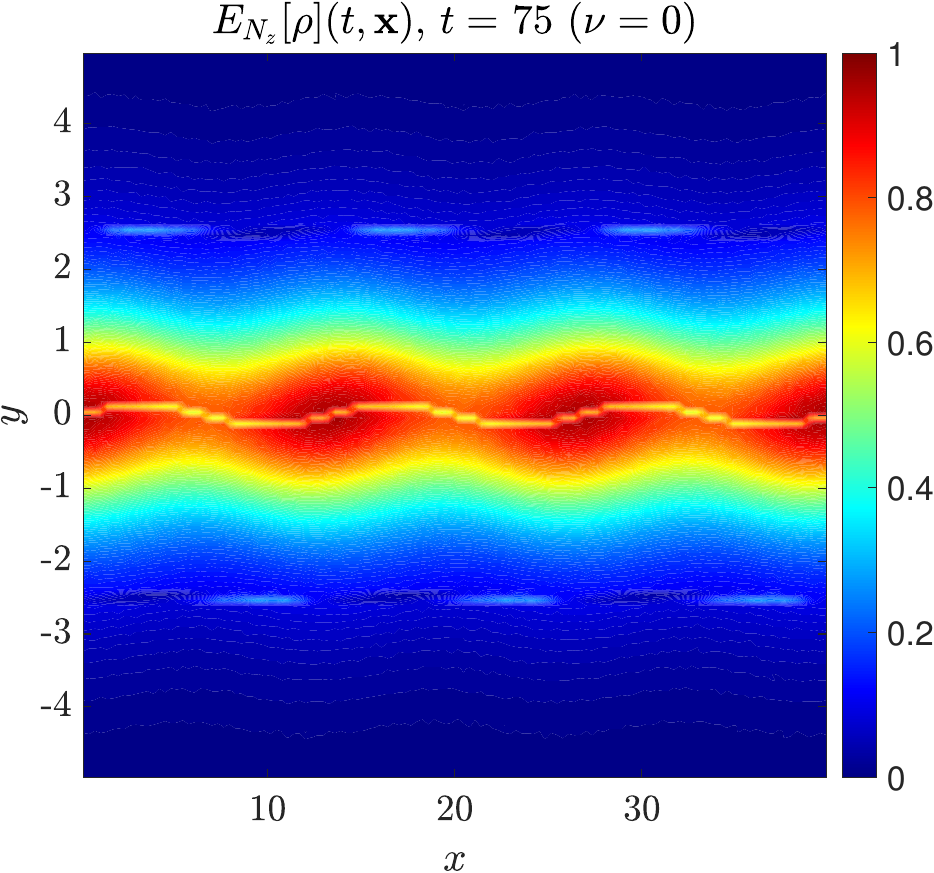} 
	\includegraphics[width=0.3\linewidth]{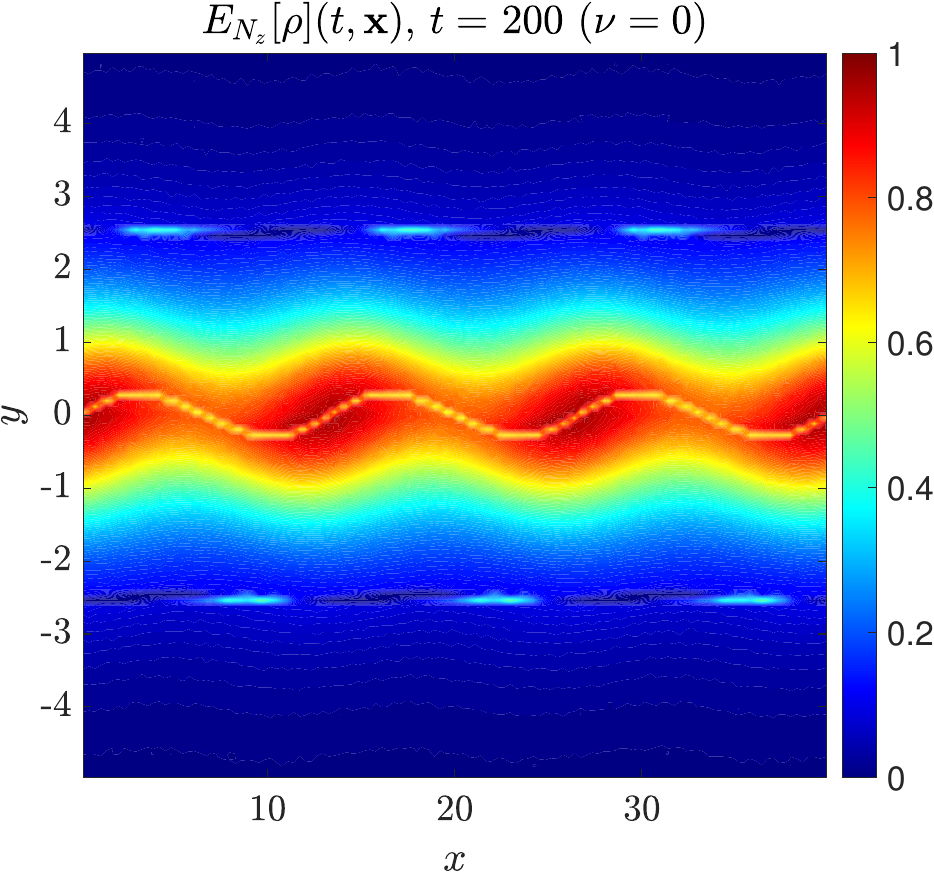}  \\
	\includegraphics[width=0.3\linewidth]{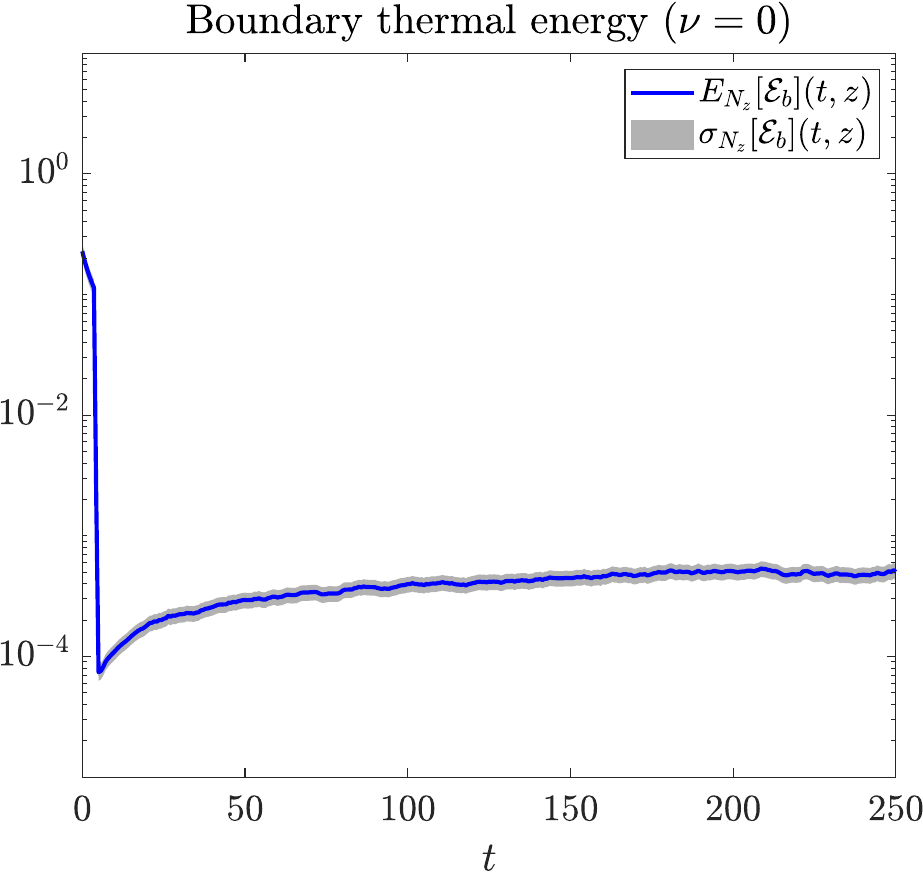} 
	\includegraphics[width=0.3\linewidth]{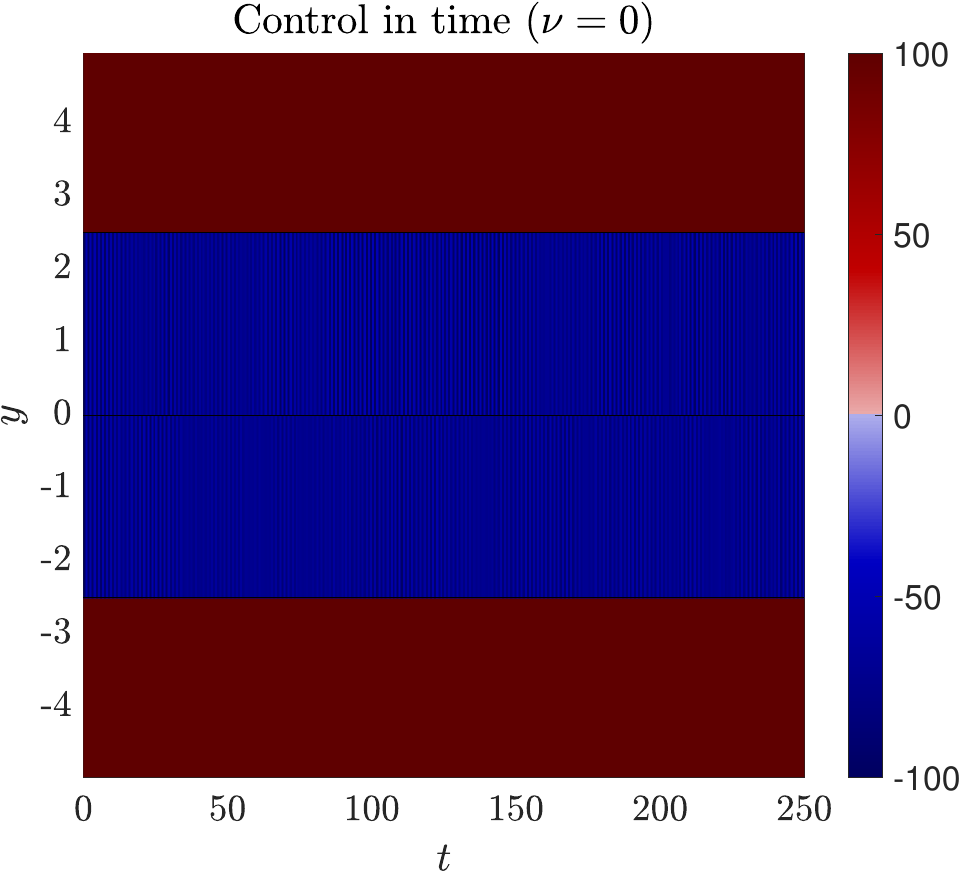} 
	\caption{Kelvin-Helmholtz instability with control for $\nu = 0$. First row: mean density at $t=50$, $t =75$ and $t=200$. Second row: thermal energy at the boundaries (on the left) and control (on the right). The mean thermal energy is depicted in blue, while the standard deviation as a shaded area.   }
	\label{fig:control_kelvin}
\end{figure} 
In this case, the control field is essentially time-independent, taking large values at the boundaries and weaker values at the center of the domain. This behavior closely resembles that of a magnetic mirror, although the configuration—with a strong magnetic field at the edges and a weaker one in the middle—is not prescribed a priori but rather obtained as the solution of the control problem \cite{freidberg2008plasma}. When particles move from regions of low magnetic field to regions of high magnetic field, conservation of the magnetic moment implies that an increase in $B$ produces a corresponding increase in $v_y$. Since the total kinetic energy must also be conserved, this increase in $v_y$ is compensated by a decrease in $v_x$. Eventually, $v_x$ vanishes and the particle reverses its motion, becoming confined along the magnetic field lines, which in our framework correspond to cell $C_k$. Particles with larger initial $v_y$ are reflected sooner, while those with smaller $v_y$ may or may not be reflected depending on the field strength. In the Kelvin–Helmholtz instability, mass tends to accumulate near the boundaries of the internal fictitious cells, while some particles are driven toward the outer boundary cells. This behavior arises from the initial particle distribution, which is characterized by small values of $v_y$. As a consequence, a slight increase in thermal energy appears at the boundaries, reaching values on the order of $10^{-4}$.

In contrast, in the two-dimensional Sod test, particles start with a larger $v_y$ component and are therefore reflected. In the short-time regime, the control exhibits oscillations aimed at steering the particles toward the desired configuration at the center of the domain. Over longer times, however, the control progressively reduces the $v_y$ component of the particle velocities. As a result, the control field converges to a stationary profile, as illustrated in Figure~\ref{fig:mirror_sod2D}, characterized by stronger intensity at the boundaries and weaker intensity at the center, thereby reproducing the magnetic mirror effect in this setting as well.
\begin{figure}[h!]
	\centering
	\includegraphics[width=0.35\linewidth]{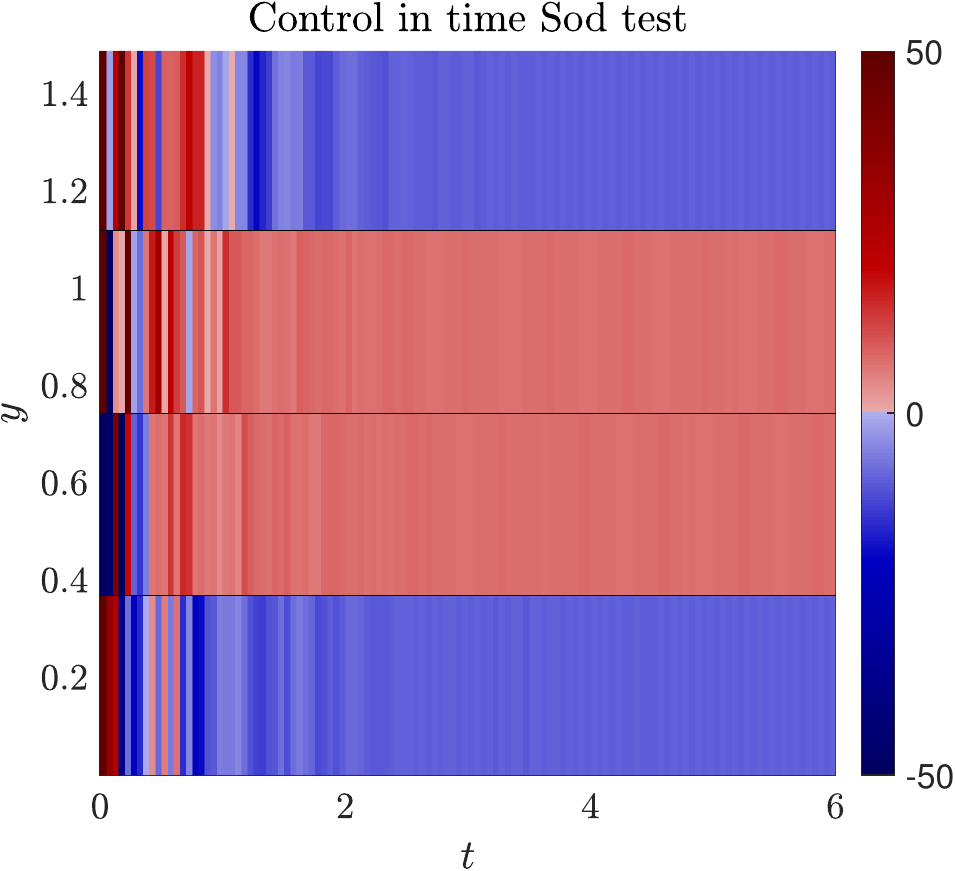} 
	\caption{Two-dimensional Sod test: magnetic mirror effect. Value of the control in time. }
	\label{fig:mirror_sod2D}
\end{figure}

\section{Conclusions} \label{sec:conclusion}
In this work, we have proposed a new control strategy for the collisional Vlasov–Poisson–BGK system under uncertainty.
The central idea is to develop an efficient instantaneous feedback control framework that steers the plasma toward a desired configuration through the application of an external magnetic field, which is constructed to be independent of the underlying randomness in the system.

To address the resulting optimization problem, we employed a semi-implicit Particle-In-Cell discretization for the Vlasov–Poisson system, combined with Monte Carlo sampling for the BGK collision process and a stochastic Gauss–Legendre quadrature method to represent uncertainty.

The control problem is formulated over a single time step, leading to an instantaneous feedback law derived via a simplified time integrator. The corresponding optimality system is solved through an augmented Lagrangian approach, ensuring enforcement of control constraints. The resulting feedback control is then integrated into the semi-implicit dynamics by considering the continuous-time limit as the time step vanishes.

Numerical experiments validate the effectiveness of the proposed control strategy across various collisional regimes, demonstrating its ability to confine the plasma and prevent boundary interactions.
As directions for future work, we plan to explore alternative uncertainty quantification strategies, including multi-fidelity control variates, which aim to reduce computational cost by coupling low- and high-fidelity models. We also intend to tackle the additional complexity arising from the full Maxwell–Vlasov–BGK system, and to investigate the incorporation of more realistic collision operators, such as the Landau operator, within the proposed control framework.
\section*{Acknowledgments}
This work has been written within the activities of GNCS and GNFM groups of INdAM (Italian National Institute of High Mathematics).
GA has been partially supported by MUR-PRIN Project 2022 No. 2022N9BM3N   ``Efficient numerical schemes and optimal control methods for time-dependent partial differential equations" financed by the European Union - Next Generation EU.
GA and GD thank the European Union — NextGenerationEU, MUR–PRIN 2022 through the PNRR Project No. P2022JC95T “Data-driven discovery and control of multi-scale interacting artificial agent systems”. GD and FF thank the Italian Ministry of University and Research (MUR) through the PRIN 2020 project (No. 2020JLWP23) ``Integrated Mathematical Approaches to Socio–Epidemiological Dynamics”. LP has been partially funded by the European Union– NextGenerationEU under the program “Future Artificial Intelligence– FAIR” (code PE0000013), MUR PNRR, Project “Advanced MATHematical methods for Artificial Intelligence– MATH4AI”.  LP acknowledges the support by the Royal Society under the Wolfson Fellowship ``Uncertainty quantification, data-driven simulations and learning of multiscale complex systems governed by PDEs" and by MIUR-PRIN 2022 Project (No. 2022KKJP4X), ``Advanced numerical methods for time dependent parametric partial differential equations with applications". The partial support by ICSC -- Centro Nazionale di Ricerca in High Performance Computing, Big Data and Quantum Computing, funded by European Union -- NextGenerationEU is also acknowledged. 

\appendix
\section{Comparison of different robust control strategies}\label{app:old_control}
In this Appendix, we first extend the control strategy introduced in \cite{albi2024instantaneous} to the setting with uncertainty, and then compare it with the approach proposed in this work. Unlike the continuous control problem formulated in equation~\eqref{eq:control_pb_continuos}, where the control is computed for each particle and subsequently interpolated to obtain an average magnetic field, the strategy described in \cite{albi2024instantaneous} directly computes a piecewise constant control (or magnetic field) within each fictitious cell $ C_k $.

We briefly recall here the derivation of the instantaneous control strategy introduced in \cite{albi2024instantaneous}, extending it to account for uncertainty while, for simplicity, we assume a collisionless setting. We first formulate the problem at the continuous level and over a finite time horizon $[0,t_f]$ as follows
\begin{equation}\label{eq:continuos_pb}
	\min_{B\in \mathcal{B}_{adm}}  \sum_{k=1}^{N_c} \mathcal{J}_k(B_k^{ext};f_k,f^0_k),\qquad 
	\textrm{s.t.}~\eqref{eq:Vlasov}-\eqref{eq:Poisson},
\end{equation}
where  $f_k= f_k (t,\xx,\vv,\zz)$ corresponds to the normalized particle density restricted to a single cell $C_k$
\[
f_k (t,\xx,\vv,\zz) = \frac{f(t,\xx,\vv,\zz)}{\rho_k(t)},\qquad \rho_k(t) = \int_{\Omega_k}f(t,\xx,\vv,\zz)\,d\xx, d\vv, 
\]
with $\rho_k(t)>0$ the total cell density and with $B=(B_1,\ldots,B_{N_c})$ now representing the vector of $z$ components of $\BB(t,\xx)$   within each cell $C_k$, $\mathcal{B}_{adm}$ the set of admissible controls such that
$\mathcal{B}_{adm} = \{B_k^{ext} | B_k^{ext}\in[-M,M], M>0,\, k = 1,\ldots, N_c\}$, and where, for each $k = 1,\ldots, N_c$, the cost functional is defined as follows
\begin{equation}\label{eq:J}
	\begin{split}
		\mathcal{J}_k(B_k; f_k,f_k^0) = & \int_{0}^{t_f} \left( \mathcal{P}\left[ \mathcal{D}(f_k)(t,\zz) \right]   + \right. \\ & + \left. \frac{\gamma}{2} \mathcal{P}\left[ \int_{\Omega_k} | B_k(t)|^2 f_k(t,\xx,\vv,\zz) d\xx d\vv\right] \right)\, dt,
	\end{split}
\end{equation}
where $\gamma>0$ is a penalization
term, $\mathcal{P}[\cdot]$ is a suitable statistical operator taking into account the presence
of the uncertainties,  $\mathcal{D}(\cdot)$ aims at enforcing a specific configuration in the distribution function, and $\Omega_k = C_k \times \Omega_v$, being $\Omega_v$ the velocity domain.  
We consider a short time horizon of
length $h > 0$ and formulate a time discretize optimal control problem through the
functional $\mathcal{J}_k$ restricted to the interval $[t, t + h]$, as follows
\begin{equation}\label{eq:min_prob_cell}
	\min_{B_k\in \mathcal{B}_{adm}} \mathcal{J}^{N,h}_{k}(B_k; f_k^N,f_k^{N,0}),
\end{equation}
subject to a semi-implicit in time discretized Vlasov dynamics, fully explicit for the velocity terms
\begin{equation}\label{eq:explicit_dinamics_nu0}
	\begin{split}
		&x_m^{n+1}(\zz) = x_m^n(\zz) + h v_{x_m}^{n+1}(\zz),\\
		&y_m^{n+1}(\zz,\xi) = y_m^n(\zz) + h v_{y_m}^{n+1}(\zz),\\
		&v_{x_m}^{n+1}(\zz) = v_{x_m}^n(\zz)+h v_{y_m}^n(\zz)B_m^{n+1} + h E_{x_m}^n(\zz),\\
		&v_{y_m}^{n+1}(\zz) = v_{y_m}^n(\zz) - h v_{x_m}^n(\zz)B_m^{n+1} + h E_{y_m}^n(\zz).
	\end{split}
\end{equation}
Using the rectangle rule for approximating the integral in time, and under the assumption that the magnetic field is independent of $\zz$, the functional in \eqref{eq:min_prob_cell} reads as follows
\begin{equation}\label{eq:discr_J_1}
	\mathcal J_k^{N,h}(B_k;f_k^N,f_k^{N,0}) =h\left(\mathcal{P}[\mathcal{D}(f_k^N)(t^{n+1},\zz)]+ \frac{\gamma}{2} |B_k|^2 \right).
\end{equation}
Here we assume 
\begin{equation}
	\mathcal{D}(f_k^N)(t,\zz) = \sum_{\ell\in{\{\textrm{x},\textrm{v}\}}}	\mathcal{D}_k(f_k^N,\phi_{\ell})(t^{n+1},\zz)
\end{equation}
with  $\mathcal{D}_k(f_k^N,\phi_{\ell})(\cdot,\zz)$ as in \eqref{eq:D}, where we replace the full domain $\Omega$ with $\Omega_k$, for any $k=1,\ldots,N_c$. 
Thus, by setting $\phi_\textrm{x} = y^{n+1}$, $\phi_{\textrm{v}} = v_y^{n+1}$, $\hat{\phi}_{\textrm{x},k} = \hat{y}_k$ and $\hat{\phi}_{\textrm{v},k} = \hat{v}_{y_k}$, target states,  and by direct computation over the empirical densities, we can rewrite the functional in \eqref{eq:discr_J_1} as
\begin{equation}\label{eq:discr_J}
	\begin{split}
		& \mathcal J_k^{N,h}(B_k)= \mathcal{P}\left[  \frac{h\alpha_\emph{v}}{2}  \vert \bar{v}_{y,k}^{n+1}(\zz) -\hat{v}_{y_k}\vert^2 + \frac{h\beta_\emph{v}}{2 N_k} \sum_{i\in C_k} \vert v_{y_i}^{n+1}(\zz)-\bar{v}_{y,k}^{n}(\zz) \vert^2 + \right.\\
		&\left.  + \frac{h\alpha_\emph{x}}{2}  \vert \bar{y}_{k}^{n+1}(\zz) -\hat{y}_k\vert^2 + \frac{h\beta_\emph{x}}{2 N_k} \sum_{i\in C_k} \vert y_{i}^{n+1}(\zz)-\bar{y}_k^{n}(\zz) \vert^2+  \frac{h\gamma}{2}  \vert B^{n+1}_k \vert^2 \right] ,
	\end{split}
\end{equation}
with 
\begin{equation}\label{eq:mean_quantities}
	\begin{split} 	
		\bar{y}_{k}(\zz)  = \frac{1}{N_{k}} \sum_{j\in C_k} y_{j}(\zz) ,\qquad
		\bar{v}_{y,k}(\zz)  = \frac{1}{N_{k}} \sum_{j\in C_k} v_{y_j}(\zz),
	\end{split}
\end{equation}
denoting the mean position and velocity over cell $C_k$, $k=1,\ldots,N_c$.
We extend now the result proved in \cite{albi2024instantaneous} in the case of uncertainty.
\begin{proposition}
	Assume the parameters to scale as 
	\begin{equation}\label{eq:scaling}
		\alpha_\emph{x} \rightarrow \frac{\alpha_\emph{x}}{h}, \qquad \beta_\emph{x} \rightarrow \frac{\beta_\emph{x}}{h}, \qquad \gamma \rightarrow \gamma h,  
	\end{equation}
	then the feedback control at cell $C_k$ associated to \eqref{eq:discr_J} reads as follows
	\begin{equation}\label{eq:L2_control_space_velocity}
		B_k = \mathbb{P}_{[-M,M]}\left( \frac{\mathcal{P}[\mathcal{R}^{N,n}_{\emph{v},k}(\zz)  + \mathcal{R}^{N,n}_{\emph{x},k}(\zz)] }{\gamma +\mathcal{P}[\mathcal{S}^{N,n}_{\emph{v},k}(\zz) + \mathcal{S}^{N,n}_{\emph{x} ,k}(\zz)] }\right), 
	\end{equation}
	where $\gamma>0$, 
	\begin{equation}\label{eq:terms_in_B}
		\begin{split}
			&\mathcal{R}^{N,n}_{\emph{v},k}(\zz)   = \alpha_\emph{v}   (\bar{v}_{y,k}^n(\zz) + h \bar{E}_{y,k}^n(\zz) -\hat{v}_{y_k}) \bar{v}^n_{x,k}(\zz)  + \\ & \qquad \qquad \qquad + \frac{\beta_\emph{v} }{N_k} \sum_{i=1}^{N_k}\left[ (v_{y_i}^n(\zz) + h E_{y_i}^n(\zz) - \bar{v}_{y,k}^n(\zz))v_{x_i}^n (\zz)\right],\\
			&\mathcal{R}^{N,n}_{\emph{x},k}(\zz)   =   \alpha_\emph{x} (\bar{y}_{k}^n(\zz) + h(\bar{v}_{y,k}^n(\zz)+ h\bar{E}_{y,k}^n(\zz)) -\hat{y}_k) \bar{v}^n_{x,k}(\zz) +\\&\qquad \qquad \qquad + \frac{\beta_\emph{x}}{N_k} \sum_{i=1}^{N_k} \left[ (y_i^n(\zz) + h(v_{y_i}^n(\zz) + h E_{y_i}^n(\zz))-\bar{y}_{k}^n(\zz))v_{x_i}^n(\zz) \right],\\
			&\mathcal{S}^{N,n}_{\emph{v},k}(\zz)  =  h \left( \alpha_\emph{v} (\bar{v}_{x,k}^n(\zz))^2 + \frac{\beta_\emph{v}}{N_k} \sum_{i=1}^{N_k} (v_{x_i}^n(\zz))^2\right) ,\\
			&\mathcal{S}^{N,n}_{\emph{x},k}(\zz) =   h^2 \left( \alpha_\emph{x} (\bar{v}_{x,k}^n(\zz))^2 + \frac{\beta_\emph{x}}{N_k} \sum_{i=1}^{N_k} (v_{x_i}^n(\zz))^2\right),
		\end{split}
	\end{equation}
	and with $\mathbb{P}_{[-M,M]}(\cdot)$ denoting the projection over the interval $[-M,M]$. 
	In the limit $h\to 0$ the control at the continuous level reads, 
	\begin{equation}\label{eq:L2_control_continuos}
		B_k(t) = \mathbb{P}_{[-M,M]} \left( \frac1\gamma\left(\mathcal{P}\left[	\mathcal{R}^N_{\emph{v},k} (t,\zz) + 	\mathcal{R}^{N}_{\emph{x},k} (t,\zz)\right] \right)\right),
	\end{equation}
	with 
	\begin{equation*}
		\begin{split}
			&	\mathcal{R}^{n}_{\emph{v},k} (t,\zz) = \alpha_\emph{v} (\bar{v}_{y,k}(t,\zz) -\hat{v}_{y_k})\bar{v}_{x,k}(t,\zz) + \frac{\beta_\emph{v}}{N_k} \sum_{i\in C_k} (v_{y_i}(t,\zz) -\bar{v}_{y,k}(t,\zz)) v_{x_i}(t,\zz),\\
			&	\mathcal{R}^{n}_{\emph{x},k} (t,\zz) = \alpha_\emph{x} (\bar{y}_{k}(t,\zz)-\hat{y}_{k})\bar{v}_{x,k}(t,\zz) + \frac{\beta_\emph{x}}{N_k} \sum_{i\in C_k} (y_{i}(t,\zz) -\bar{y}_{k}(t,\zz)) v_{x_i}(t,\zz).
		\end{split}
	\end{equation*}
\end{proposition}

We now focus on the two-dimensional Sod test discussed in Section~\ref{sec:2D_sod_shock}, using the same numerical setting.  
We consider a collisionless regime and introduce uncertainty in the system by sampling $ N_z = 10 $ Gauss--Legendre nodes.  
We compare the instantaneous controls defined in equations~\eqref{eq:L2_control_interp} and~\eqref{eq:L2_control_continuos}, using the parameters  
$ \alpha_\text{x} = 5 $, $ \beta_\text{x} = 2 $, $ \alpha_\text{v} = 15 $, $ \beta_\text{v} = 12 $, $ \gamma = 2.5 \times 10^{-3} $, and $ M = 50 $,  
with $ \mathcal{P}(\cdot) $ defined as in equation~\eqref{eq:R_max}.  
The simulation is carried out up to final time $ t_f = 2 $, with a time step $ h = 0.05 $. Figure~\ref{fig:comparison} summarizes the results.  
On the left, we report the mean thermal energy at the boundaries along with the corresponding standard deviation.  
In the center and on the right, we show the magnetic field values.  
In the first row, the initial temperature is defined as in equation~\eqref{eq:rho0_T0_2D}, while in the second row it corresponds to the high-temperature configuration given in equation~\eqref{eq:T0_10} (last equation).
In both cases, the new control strategy proposed in this work proves to be more effective in reducing thermal energy compared to the approach introduced in \cite{albi2024instantaneous}, with the improvement being particularly significant in high-temperature scenarios.

\begin{figure}[h!]
	\centering
	\includegraphics[width=0.3\linewidth]{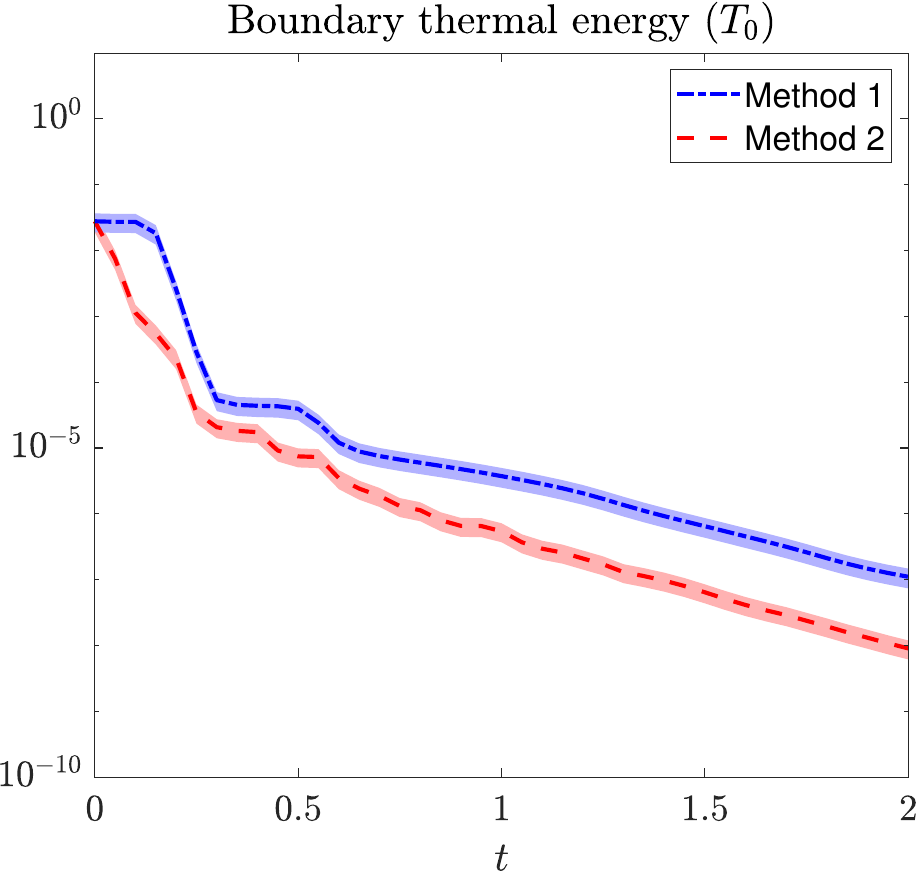}
	\includegraphics[width=0.3\linewidth]{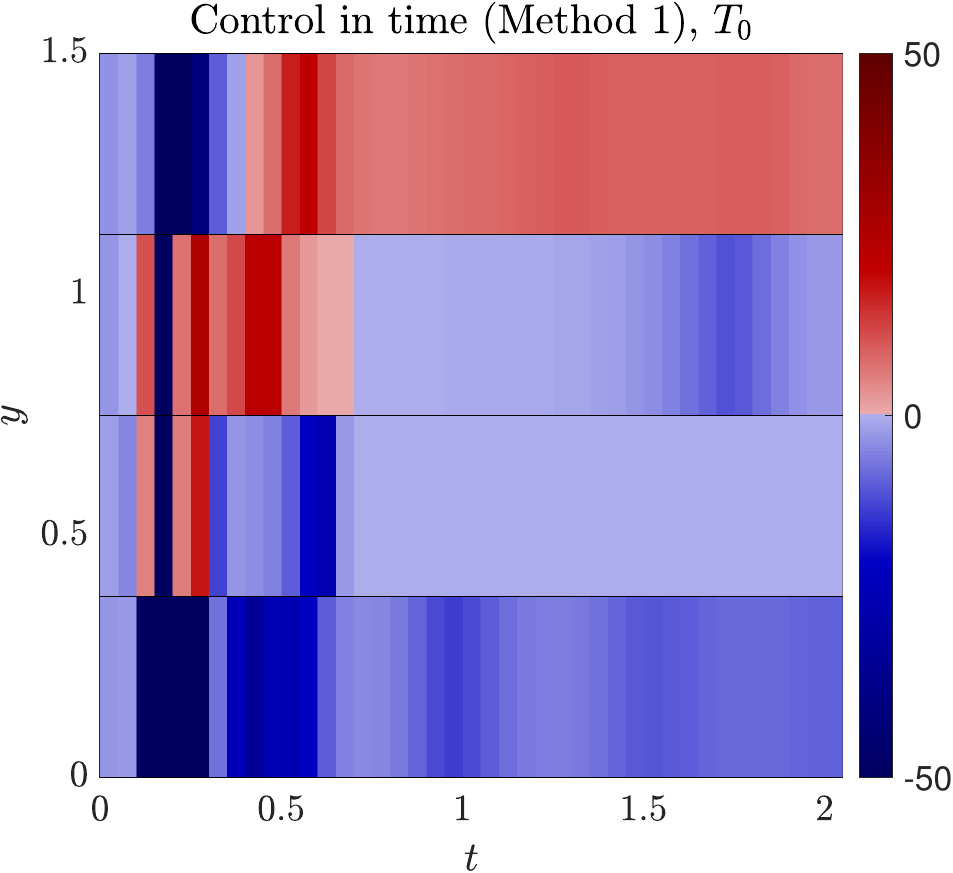} 
	\includegraphics[width=0.3\linewidth]{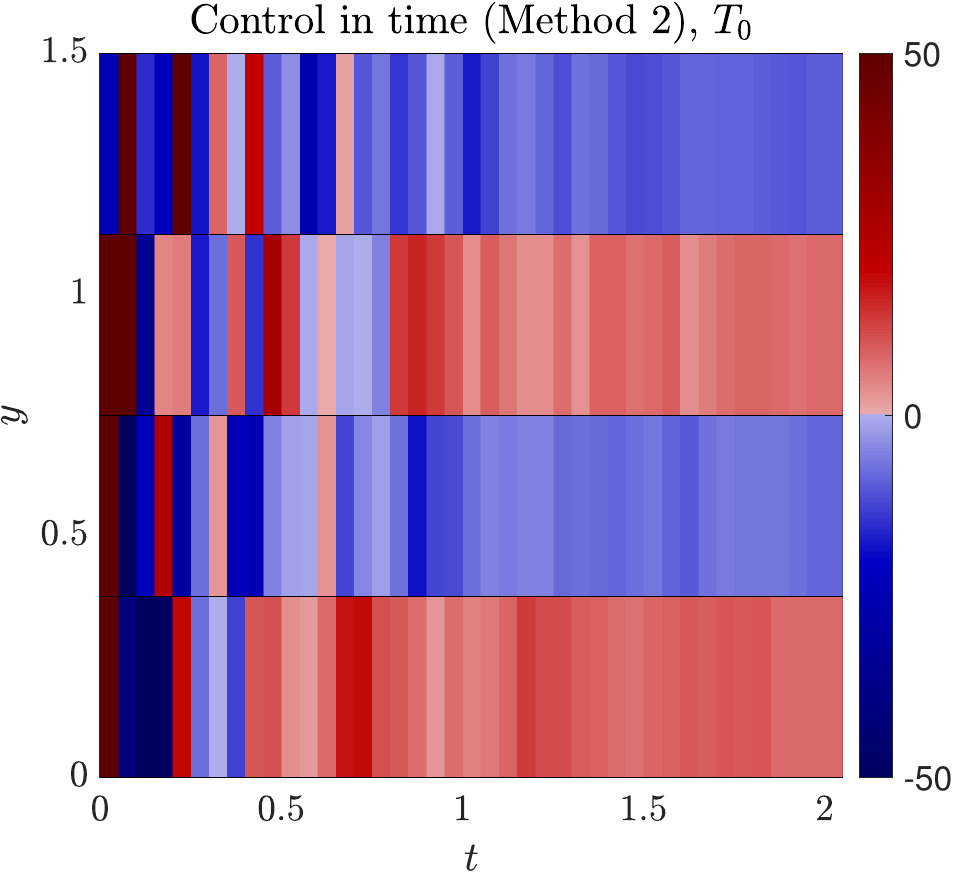} \\
	\includegraphics[width=0.3\linewidth]{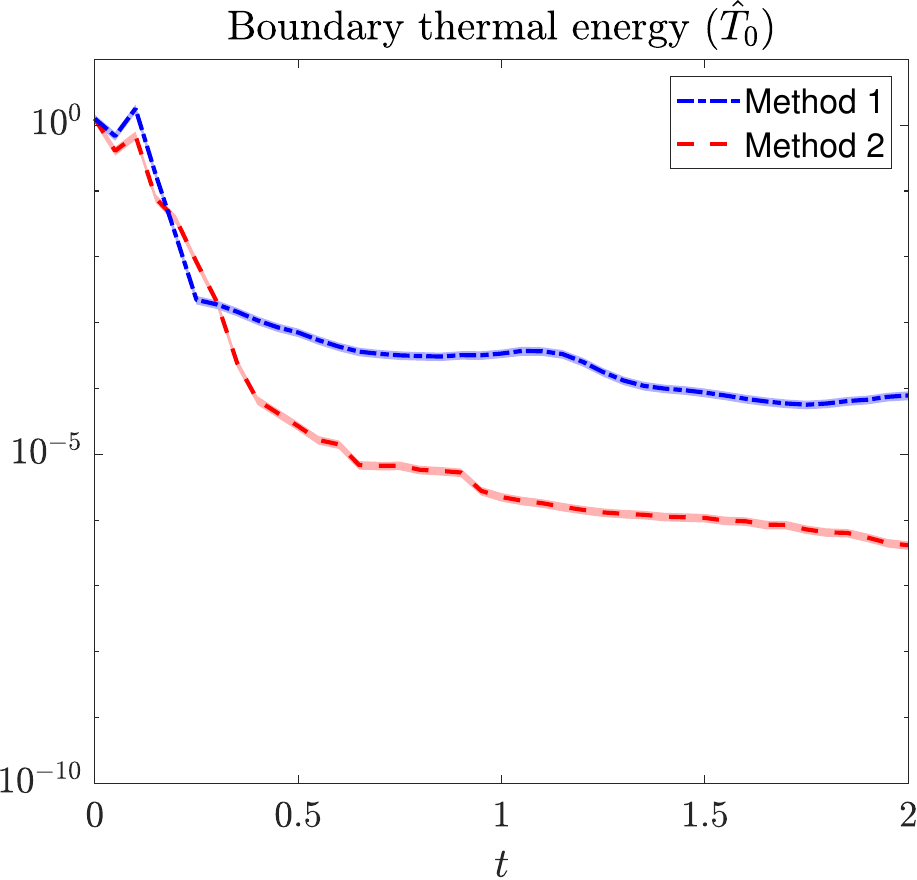}  
	\includegraphics[width=0.3\linewidth]{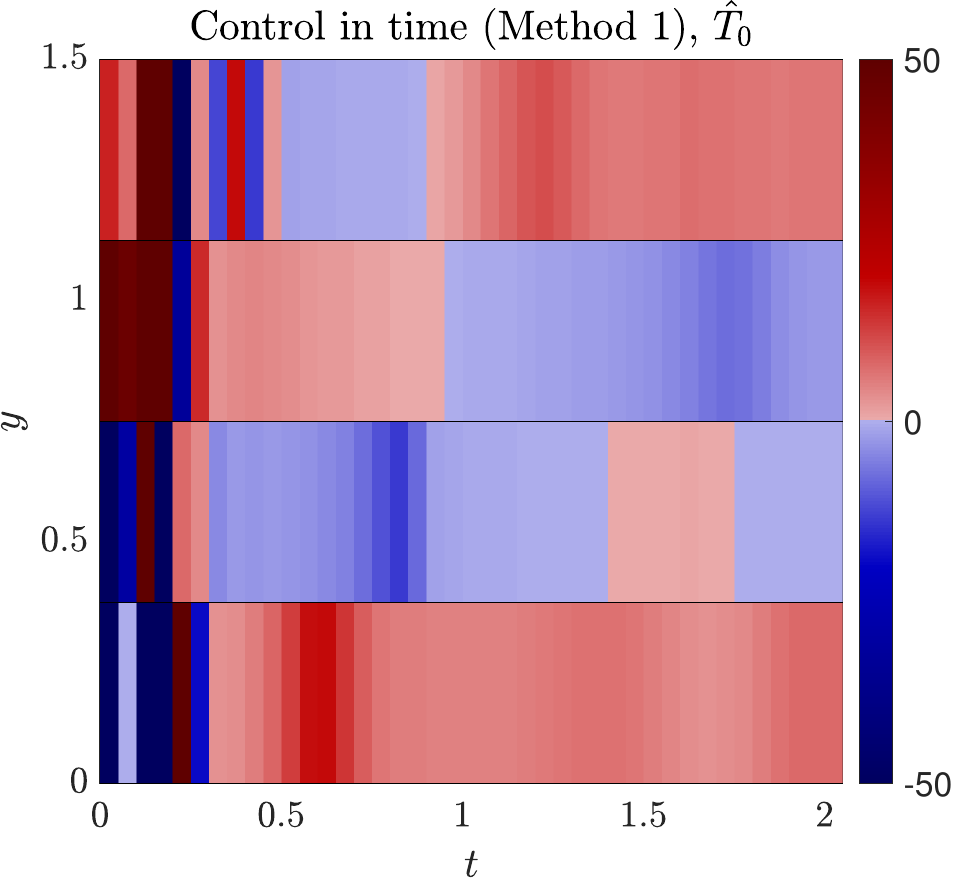}
	\includegraphics[width=0.3\linewidth]{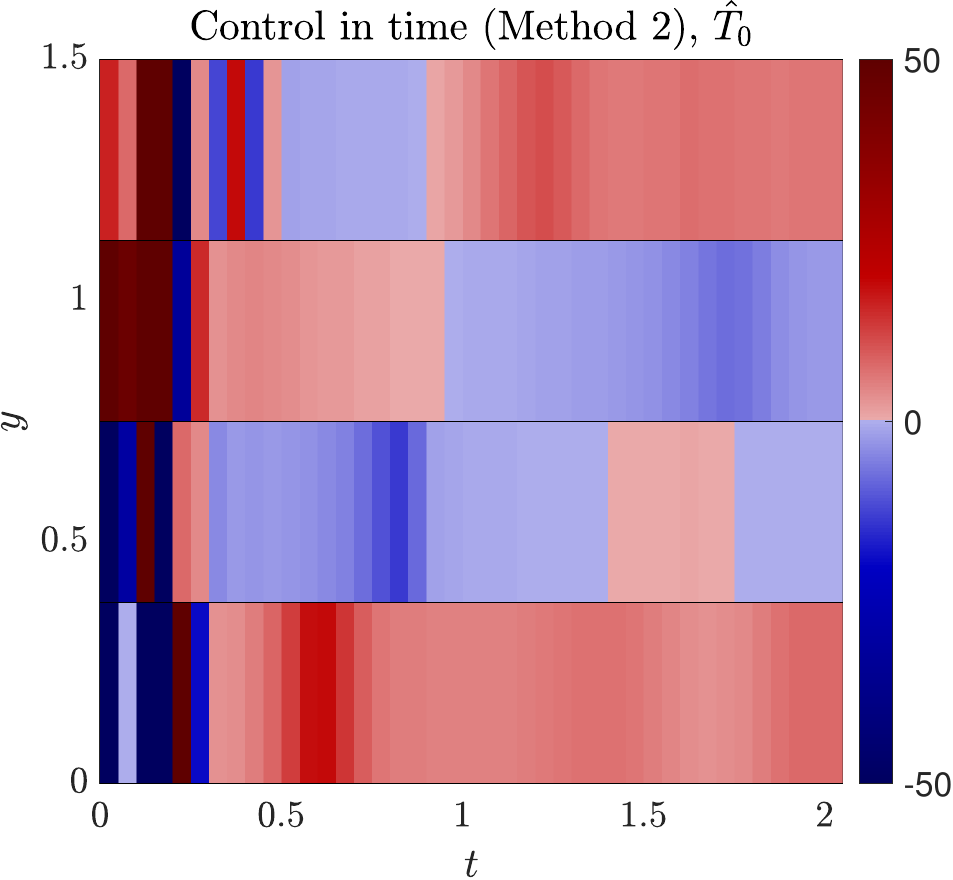}
	\caption{Two dimensional Sod shock tube test with control. Comparison between the two control strategies, \eqref{eq:L2_control_space_velocity} (Method 1) and \eqref{eq:L2_control_interp} (Method 2) in terms of thermal energy at the boundaries (on the left) and of magnetic field value (in the center and on the right). The different lines depict the mean thermal energy at the boundaries, while the shaded areas represent the standard deviation. }
	\label{fig:comparison}
\end{figure} 

\bibliographystyle{abbrv}
\bibliography{biblio}
\end{document}